\documentclass[twoside, 11pt]{amsart}

\usepackage[english]{babel}

\usepackage[T1]{fontenc}
\usepackage[lf]{Baskervaldx} % lining figures
\usepackage[bigdelims,vvarbb]{newtxmath} % math italic letters from Nimbus Roman
\usepackage[cal=boondoxo]{mathalfa} % mathcal from STIX, unslanted a bit

\usepackage{graphicx}
\usepackage{fullpage}

\usepackage{amsmath,amsfonts,amscd,amsthm,mathrsfs} %bug avec amssymb ???
\usepackage{stmaryrd}
\usepackage{latexsym}
\usepackage{booktabs}

\usepackage{comment}
\usepackage{enumitem}

\let\oldtocsection=\tocsection
\let\oldtocsubsection=\tocsubsection

\renewcommand{\tocsection}[2]{\hspace{0em}\vspace{0.1em}\rule{0pt}{14pt}\oldtocsection{#1}{#2}\bf}
\renewcommand{\tocsubsection}[2]{\hspace{2em}\oldtocsubsection{#1}{#2}}

\usepackage{xcolor}
\definecolor{Chocolat}{rgb}{0.36, 0.2, 0.09}
\definecolor{BleuTresFonce}{rgb}{0.215, 0.215, 0.36}

\usepackage[colorlinks,final,hyperindex]{hyperref}
\hypersetup{
	pdftex,
	linkcolor=Chocolat,
	citecolor=BleuTresFonce,
	filecolor=BleuTresFonce,
	urlcolor=BleuTresFonce,
	pdftitle=Homotopy double Poisson gebras and pre-Calabi--Yau algebras,
	pdfauthor=Johan Leray,
	pdfsubject=,
	pdfkeywords=
}

\usepackage[noabbrev,capitalize]{cleveref}

\usepackage{tikz,tikz-cd}
\usetikzlibrary{
	decorations.pathmorphing,
	arrows,
	arrows.meta,
	calc,
	shapes,
	shapes.geometric,
	decorations.pathreplacing, 
	math}
\tikzcdset{arrow style=tikz, diagrams={>=stealth}}
\tikzset{>=stealth'}



\theoremstyle{plain}
\newtheorem{thm}{Th\'eor\`eme}[section]
\newtheorem{proposition}[thm]{Proposition}
\newtheorem{theorem}[thm]{Theorem}
\newtheorem{corollary}[thm]{Corollary}

\newtheorem{lemma}[thm]{Lemma}
\newtheorem*{thmintro}{Theorem}

\theoremstyle{definition}
\newtheorem{definition}[thm]{Definition}

\newtheorem{remark}[thm]{\sc Remark}
\newtheorem{example}[thm]{\sc Example}
\newtheorem{examples}[thm]{\sc Examples}

\usepackage{xparse}       	% Parseur
\ExplSyntaxOn
\NewDocumentCommand{\createbunch}{ m O{} m }
 {
  \clist_map_inline:nn { #3 } { \cs_new_protected:cpn { #2 ##1 } { #1 { ##1 } } }
 }
\ExplSyntaxOff
\createbunch{\mathbb}   [bb]{A,B,C,D,E,F,G,H,I,J,K,L,M,N,O,P,Q,R,S,T,U,V,W,X,Y,Z}
\createbunch{\mathcal}  [cal]{A,B,C,D,E,F,G,H,I,J,K,L,M,N,O,P,Q,R,S,T,U,V,W,X,Y,Z}
\createbunch{\mathscr}  [scr]{A,B,C,D,E,F,G,H,I,J,K,L,M,N,O,P,Q,R,S,T,U,V,W,X,Y,Z}
\createbunch{\mathsf}   [sf]{A,B,C,D,E,F,G,H,I,J,K,L,M,N,O,P,Q,R,S,T,U,V,W,X,Y,Z}
\createbunch{\mathrm}   [rm]{A,B,C,D,E,F,G,H,I,J,K,L,M,N,O,P,Q,R,S,T,U,V,W,X,Y,Z,a,b,c,d,e,f,g,h,i,j,k,l,m,o,p,q,r,s,t,u,v,w,x,y,z}
\createbunch{\mathfrak} [frak]{A,B,C,D,E,F,G,H,I,J,K,L,M,N,O,P,Q,R,S,T,U,V,W,X,Y,Z}
\createbunch{\mathfrak} [frak]{a,b,c,d,e,f,g,h,i,j,k,l,m,n,o,p,q,r,s,t,u,v,w,x,y,z}
\createbunch{\mathcal} [cal]{a,b,c,d,e,f,g,h,i,j,k,l,m,n,o,p,q,r,s,t,u,v,w,x,y,z}

\createbunch{\mathsf}   {sSet,Ch,Mod,Ho,Set,Ab,cdga,dga,Fin,Cat,Op,Gpd,LMod,Alg,Bij}
\createbunch{\operatorname} {Hom,Mor,End,id,colim,Ind,Func,ho,Path,Cyl}

\newcommand{\cev}[1]{\reflectbox{\ensuremath{\vec{\reflectbox{\ensuremath{#1}}}}}}

\newcommand{\sgn}{\mathrm{sgn}}

\newcommand{\Sy}{\bbS}

\newcommand{\la}{\langle}
\newcommand{\ra}{\rangle}

\newcommand{\nec}{\mathfrak{neck}}
\newcommand{\cnec}{\mathfrak{cneck}}
\newcommand{\convDPois}{\mathfrak{DPois}}
\newcommand{\hhc}{\mathfrak{hc}}

\newcommand{\eps}{\varepsilon}

\newcommand{\C}{\mathbb{C}}

\newcommand{\NN}{\mathbb{N}}

\DeclareMathSymbol{\antishrieksymbol}{\mathord}{operators}{'74}
\makeatletter
	\DeclareRobustCommand{\antishriek}{{\mathpalette\anti@shriek\relax}}
	\newcommand\anti@shriek[2]{%
		\raisebox{\depth}{$\m@th#1\antishrieksymbol$}%
	}
\makeatother
\newcommand{\antish}{\antishriek}
\newcommand{\ac}{\antishriek}

\newcommand{\Gs}{\mathrm{G}}

\newcommand{\bGs}{\overline{\mathrm{G}}}

\renewcommand{\Bar}{\mathrm{B}}

\newcommand{\Cobar}{\mathrm{\Omega}}

\newcommand{\As}{\mathrm{Ass}}

\newcommand{\DLie}{\mathrm{DLie}}

\newcommand{\DPois}{\mathrm{DPois}}
\newcommand{\cDPois}{\mathrm{cDPois}}

\newcommand{\catofalgebras}[1]{{#1}\text{-}\mathsf{alg}}

\newcommand{\catofmon}[1]{\mathsf{Mon}\!\left(#1\right)}

\newcommand{\Ibox}{\mathrm{I}}

\newcommand{\Sbimod}{\Sy\mbox{-}\mathsf{bimod}}

\newcommand{\Properad}{\mathsf{properads}}

\newcommand{\dgVect}{\mathsf{dgVect}}

\newcommand{\prop}{\mathrm{P}}

\newcommand{\op}{{\mathrm{op}}}

\newcommand{\too}{\longrightarrow}

\newcommand{\field}{\mathbb{k}}
\newcommand{\susp}{\mathrm{s}}
\newcommand{\asusp}{\mathrm{s}^{-1}}

\renewcommand{\gg}    {\mathrm{g}}

\newcommand{\DB}[2]{\{\!\!\{#1,#2\}\!\!\}}

\newcommand{\DDBL}[3]{\left\{\!\!\left\{#1,\DB{#2}{#3}\right\}\!\!\right\}_L}

\def\point{\vcenter{\hbox{\scalebox{0.5}{$\bullet$}}}}

\def\pap#1#2#3{#1\stackrel{#2}{\vcenter{\hbox{\text{\scalebox{1.2}{$\Join$}}}}}#3}
\def\1{\mathbb{1}}

\def\s1{\mathrm{s}^{-1}}

\newcommand{\Cyc}{\mathsf{Cyc}}
\newcommand{\CycMod}{\mathsf{CycMod}}
\newcommand{\F}{\mathcal{F}}
\newcommand{\M}{\mathcal{M}}
\renewcommand{\P}{\mathcal{P}}
\newcommand{\CC}{\mathcal{C}}
\newcommand{\PT}{\mathrm{PT}}
\newcommand{\bbPT}{\mathbb{P}\mathbb{T}}
\newcommand{\K}{\field}
\newcommand{\g}{\mathfrak{g}}
\newcommand{\MC}{\mathrm{MC}}
\renewcommand{\d}{\mathrm{d}}
\newcommand{\whP}{\widehat{\mathcal{P}}}

\DeclareMathOperator{\EEnd}{\mathcal{E}\!\mathcal{n}\!\mathcal{d}}
\DeclareMathOperator{\AAs}{\mathcal{A}\!\mathcal{s}}
\DeclareMathOperator{\uAAs}{\mathcal{u}\mathcal{A}\!\mathcal{s}}
\DeclareMathOperator{\cAAs}{\mathcal{c}\mathcal{A}\!\mathcal{s}}
\DeclareMathOperator{\Tw}{Tw}

\DeclareMathOperator{\eend}{end}

\newcommand{\ibt}{\underset{\scriptscriptstyle (1,1)}{\boxtimes}}
\newcommand{\libt}{\lhd_{(*)}}
\newcommand{\ribt}{{\vphantom{\rhd}}_{(*)}\!\rhd}
\newcommand{\Cop}[2]{{\vphantom{\Delta}}_{{#1}}\Delta{_{#2}}}

\title{Pre-Calabi--Yau algebras and homotopy double Poisson gebras}
\author{Johan Leray}
\address{Universit\'e de Nantes, 
Laboratoire de Math\'ematiques Jean LERAY (LMJL), CNRS, 
UMR 6629, 
2 Chemin de la Houssini\`ere, BP 92208,  44322 Nantes Cedex 3, France}

\email{Johan.Leray@univ-nantes.fr}
\author{Bruno Vallette}
\address{Universit\'e Sorbonne Paris Nord, Laboratoire de G\'eom\'etrie, Analyse et Applications, LAGA, CNRS, UMR 7539, 93430, Villetaneuse, France}
\email{vallette@math.univ-paris13.fr}

\date{\today}
\thanks{
\newline
The first author was supported by a postdoctoral grant of the DIM Math Innov -- R\'egion \^Ile de France and the second author was supported by the Institut Universitaire de France and the project ANR-20-CE40-0016 HighAGT}
\dedicatory{"Le souvenir de ta r\'eciproque" -- Roberta Scarsella}

\begin{document}

\begin{abstract}
We prove that the notion of a curved pre-Calabi--Yau algebra is equivalent to the notion of a curved homotopy double Poisson gebra, thereby settling the equivalence between the two ways to define derived noncommutative Poisson structures. 
We actually prove that the respective differential graded Lie algebras controlling both deformation theories are isomorphic.
This approach allows us to apply the recent developments of the properadic calculus in order to establish the homotopical properties of pre-Calabi--Yau algebras and homotopy double Poisson gebras: $\infty$-morphisms, homotopy transfer theorem, formality, Koszul hierarchy, and twisting procedure. 
\end{abstract}

\maketitle

\tableofcontents

\section*{Introduction}

\paragraph*{\bf Poisson geometry}
In commutative geometry, there are two equivalent ways to define the notion of  a \emph{Poisson structure}. 
One can first consider a Lie bracket
\[
	\{-,-\} \colon C^{\infty}(M) \otimes C^{\infty}(M) \too C^{\infty}(M)
\]
satisfying the Leibniz rule with respect to the commutative product of the algebra of smooth functions on a manifold $M$~. 
Then, one can notice that this  data is equivalent to a bivector field satisfying the Maurer--Cartan equation with respect to the Schouten--Nijenhuis bracket of polyvector fields.
In general, an algebra made up of a compatible pair of a commutative product and a Lie bracket is called a \emph{Poisson algebra}. 

\medskip

\paragraph*{\bf Noncommutative Poisson geometry}
How can one develop a noncommutative analogue of a Poisson structure? 
Let us recall the guiding \emph{Kontsevich--Rosenberg principle of noncommutative geometry} \cite{KR00}: 
the noncommutative analogue of a type of structures on schemes should be a type of structures on associative algebras $A$, viewed as noncommutative affine schemes, which induces the original type of structures on the affine schemes of representations of $A$~. 
M. Van den Bergh solved this question  in \cite{VdB08}: the noncommutative analogue of a Poisson structure is a \emph{double Lie bracket}
\[
	\DB{-}{-} \colon A \otimes A \too A\otimes A
\]
satisfying a Leibniz relation with respect to the associative product on $A$~. 
Similarly, one can coin an algebra of 
noncommutative polyvector fields which induces polyvector fields on representation schemes: it is given by the 
tensor algebra on derivations from $A$ to $A\otimes A$~, equipped with a  Schouten--Nijenhuis type (double) bracket. 
A  \emph{noncommutative bivector field} is a bitensor satisfying the associated Maurer--Cartan equation. 
In the smooth case, these two types of structures were shown to be equivalent in \cite[Section~4]{VdB08}.

\medskip

The aforementioned type of bialgebras is called a 
\emph{double Poisson gebra}; it has already found applications in many fields of mathematics like quantum algebra, 
low-dimensional geometry, symplectic geometry, integrable systems, and 
mathematical physics. An exhaustive list of references can be found on the web site \cite{FaiWeb} of M. Fairon. 
(In this paper, as in our previous works like \cite{HLV20}, we prefer to use the terminology \emph{gebra} to refer to algebraic structure with several inputs and several outputs. We reserve the terminology ``algebra'' for algebraic structure with exactly one output, like associative or Lie algebras.)

\medskip

\begin{center}
	\begin{tabular}{|c|c|}
	\hline
	\rule{0pt}{12pt}
	{\sc Commutative Geometry}&
	{\sc \ Noncommutative Geometry}\\
	\hline
	\rule{0pt}{11pt}
	representation varietes&
	associative algebras\\
	\hline
	\rule{0pt}{11pt}
	symplectic structures
	&
	bisymplectic structures
	 \\
	 \hline
	\rule{0pt}{11pt}
	Poisson structures
	&
	 \emph{double Poisson structures} \\
	\hline
\end{tabular}
\end{center}

\medskip

\paragraph*{\bf Derived Poisson geometry}
The passage to \emph{derived} (algebraic) geometry amounts basically to working now with \emph{differential graded} (or simplicial) commutative algebras. 
One can use either one or the other definition to extend the classical notion of a Poisson structure on this derived level. 
In the differential graded context, the proper generalization of the notion of a Poisson algebra, satisfying the expected homotopical properties, is that of a \emph{Poisson algebra up to homotopy}. 
Such an algebraic structure is made up of infinite series of operations with various arities which relax up to homotopy all the relations satisfied by  the commutative product and the Lie bracket of a Poisson algebra. 
Since the combinatorics of these higher operations is quite involved, one enjoys encoding them conceptually using operads: this is made  possible by the fact that the operad encoding Poisson algebras is Koszul, see \cite[Section~13.3.7]{LV12}.

\medskip

In derived geometry, a homotopical version of \emph{shifted polyvector fields}, equipped with 
a Schouten--Nijenhuis type bracket, was introduced in \cite{PTTV13, CPTVV17}. 
A \emph{shifted Poisson structure} is nothing but a shifted bivector field satisfying the associated Maurer--Cartan equation. 
V. Melani proved in \cite{Mel16} that these two notions of derived Poisson structures are equivalent for 
derived affine stacks. 

\medskip

\paragraph*{\bf Derived noncommutative Poisson geometry}
How can one develop a noncommutative analogue of a derived Poisson structure? According to the Kontsevich--Rosenberg principle, one should consider some structures on some \emph{derived representation schemes} of a differential graded  associative algebra. 
Fortunately this latter notion was developed by Y. Berest, G. Khachatryan, and A. Ramadoss  in \cite{BKR13}. 
So one can try first to come up with a notion of a \emph{double Poisson gebras up to homotopy} such that this kind of structure on $A$ induces a homotopy Poisson algebra structure on  the derived representation schemes of $A$~. 
Here again, the guiding algebraic toolbox already exists: the notion of a double Poisson gebra can be encoded by a properad \cite{Val07}, which is a generalisation of the notion of an operad allowing several outputs. 
This properad has recently been proved to be Koszul by the first named author in \cite{Ler19ii}. 
In the first section of this paper (\cref{sec:ProperadDoublePoissInfty}), we make explicit the notion of a homotopy double Poisson gebra obtained in this way. 

\medskip

Dually, the derived noncommutative space of polyvector fields is provided by the \emph{generalised necklace Lie-admissible algebra}
made up of symmetric and nonsymmetric tensors of $A$ and $A^*$ and equipped with a Schouten--Nijenhuis type bracket, see \cref{subsec::pCY_alg}. 
M. Kontsevich, A. Takeda, and Y. Vlassopoulos defined what should be a ``noncommutative shifted bivector'' as a Maurer--Cartan element in it and, this way,  they came to the notion of a \emph{pre-Calabi--Yau algebra} structure on $A$, see \cite{KTV21}. This notion is actually equivalent to the notion of a (genus $0$) \emph{$\rmV_\infty$-gebra} introduced earlier by T. Tradler and M. Zeinalian in \cite{TZ07Bis} and to the notion of \emph{$\rmA_\infty$-algebras with boundary} of P. Seidel \cite{Seidel12}.
W.-K. Yeung proved that this definition of a derived noncommutative Poisson structure satisfies the Kontsevich--Rosenberg principle: 
any pre-Calabi--Yau algebra structure on $A$ induces a shifted Poisson structure on 
its derived moduli stack of representations \cite[Corollary~4.78]{Yeu18}. 

\bigskip

\begin{center}
	\begin{tabular}{|c|c|}
	\hline
	\rule{0pt}{12pt}
	{\sc Derived Geometry} &
	{\sc Derived Noncommutative Geometry}\\
	\hline
	\rule{0pt}{11pt}
	polyvector fields
	&
	higher cyclic complex
	\\
	\hline
	\rule{0pt}{11pt}
	shifted symplectic structures
	&
	Calabi--Yau structures
	 \\
	 \hline
	\rule{0pt}{11pt}
	shifted Poisson structures
	&
	 \emph{homotopy double Poisson gebras}/\emph{pre-Calabi--Yau algebras}\\
	\hline
\end{tabular}
\end{center}

\bigskip

Several authors have already compared pre-Calabi--Yau algebras to some homotopical variations of  double Poisson gebras. 
First, N. Iyudu, M. Kontsevich, and Y. Vlassopoulos have shown that any double Poisson algebra structure induces canonically a pre-Calabi--Yau algebra structure, see \cite{IK18,IK19,IKV21}. 
Another proof of this result was given by W.-K.Yeung in \cite[Example~1.35]{Yeu18}. 
D. Fern\'andez and E. Herscovich have extended this result in \cite{FH20, FH19} to double Poisson-infinity gebras, as defined by T. Schedler in \cite{Sch06},  and to double quasi-Poisson gebras \cite{VdB08}.

\medskip

\paragraph*{\bf Present achievements}

The first result of the present paper extends \emph{the equivalence between the two definitions of a Poisson structure in classical geometry  to the level of derived noncommutative geometry}: 

\bigskip

\begin{center}
\fbox{curved pre-Calabi--Yau algebras $\cong$ curved homotopy double Poisson gebras}~.
\end{center}

\bigskip

Notice that, for both notions to be strictly equivalent, it is mandatory to consider some mild generalisations including curvatures on both sides. 
Our main theorem actually holds on the level of the differential graded Lie-admissible algebras encoding both notions.  

\begin{thmintro}[\cref{thm:main}]
There exists canonical morphisms of differential graded Lie-admissible algebras
\[\begin{tikzcd}[column sep=normal, row sep=tiny]
 \cnec_A \arrow[r, hook] & \hhc_A  \arrow[r, "\cong"]    & \mathfrak{c}\convDPois_A~,  
	\end{tikzcd}
\]
encoding respectively tensorial curved pre-Calabi--Yau algebras, curved pre-Calabi--Yau algebras, and  curved homotopy double Poisson gebras,
where the first one is an isomorphism if and only if $A$ is degree-wise finite dimensional. 
\end{thmintro}

\paragraph*{\bf Homotopy theory of algebraic structures.}
The categories of pre-Calabi--Yau algebras, homotopy double Poisson gebras, and $\rmV_\infty$-algebras are all encoded by structural (co)properads, which is an algebraic object used to represent operations with several inputs and several outputs. 
This bright point opens the doors to general homotopical properties using of the properadic calculus developed recently in 
\cite{HLV20, HLV22}.  
This allows us to settle a suitable notion of an \emph{$\infty$-morphism} for 
pre-Calabi--Yau algebras, homotopy double Poisson gebras, and $\rmV_\infty$-algebras, see \cref{sec:InftyMor}. 
We show that the associated categories carry all the required homotopical properties for this type of structures: 
homotopy transfer theorem (\cref{thm:HTT}), homological invertibility of $\infty$-quasi-isomorphisms (\cref{thm:InvInfQI}), and
equivalence between zig-zags of quasi-isomorphisms and $\infty$-quasi-isomorphisms (\cref{thm:MainInftyQi}). 

\medskip

A few words are in order about the present methods. In homotopical algebra, one can introduce new structures or objects, like $\infty$-morphisms, transferred or twisted structures, by explicit formulas and then prove their respective properties, which can amount to long and complicated computations, especially checking carefully all the signs. Like in \cite{HLV20, DSV21} for instance, we use a different approach here: we obtain all these structures and objects by applying the conceptual toolbox provided by the properadic calculus. This way, we know that all of them will  automatically satisfy the required properties. The only remaining piece of work amounts to making them explicit, which  is obtained by describing all the various decomposition maps of the structural coproperads. We perform this task for 
pre-Calabi--Yau algebras and homotopy double Poisson gebras in \cref{subsec:CoproduitsCpCY} and \cref{sec:InftyMor}. A notion of $\infty$-morphism for pre-Calabi--Yau algebras was introduced independently by M. Kontsevich, A. Takeda, and Y. Vlassopoulos in \cite[Definition~25]{KTV21}:  the abovementioned explicit description shows that this definition agrees with the properadic one and thus it establishes all the expected homotopical properties. 

\medskip

In \cite{CV25I}, we integrate dg Lie-admissible algebras of properadic convolution type into an explicit deformation gauge group.
In the present case, its action on Maurer--Cartan elements induces two universal ways to 
create derived noncommutative Poisson structures. 
The first one, called the \emph{Koszul hierarchy}, produces a homotopy double Poisson gebra, and thus a pre-Calabi--Yau algebra, from the data of a shifted Koszul dual double Poisson gebra and a chain complex (\cref{thm:KoszulHierach}). 
The second one, called the \emph{twisting procedure},  perturbes (curved) pre-Calabi--Yau algebras with the data of just one (Maurer--Cartan) element (\cref{thm:TwistProc}). Altogether, these homotopical constructions and properties form the second main contribution of the present paper. 

\medskip

A particularly careful attention was paid to signs: throughout this text, we make them explicit. 
Although almost all of them come from a direct application of the Koszul sign rule and the Koszul sign convention, but computing them is often a tedious exercice. The proof of the main result is an illuminating example. In order to preserve the flow of exposition, we postponed it to \cref{appendice}. 
 
\medskip

\paragraph*{\bf Layout}
In the first section, we make explicit the notion of a (curved) homotopy double Poisson gebra prescribed by the Koszul duality theory for properads. 
The second section is devoted to curved pre-Calabi--Yau algebras using the deformation theory of morphisms of cyclic non-symmetric operads. 
We conclude it with the main result: the two differential graded Lie-admissible algebras encoding these two types of structures are isomorphic. 
In a third section, we consider the notion of pre-Calabi--Yau algebras, that we compare with homotopy double Poisson gebras and $\rmV_\infty$-algebras, and we make explicit the various decomposition maps of the related coproperads. 
We apply this to settle the homotopy theory of pre-Calabi--Yau algebras, homotopy double Poisson algebras, and 
$\rmV_\infty$-algebras via their key notion of $\infty$-morphisms in the forth section. 
The proof of the main theorem is given in an appendix. 

\medskip

\paragraph*{\bf Conventions}
We work over a  field $\field$ of characteristic $0$ and its category $\dgVect$ of differential graded (dg) vector spaces. Since every object will be differential graded, we will drop the prefix "dg" for simplicity.
We use the homological degree convention  with the differential of degree $-1$.
The degree of a homogenous element $a$ will be denoted by  $|a|$~. 
The symmetric monoidal category structure on dg vector spaces carries the \emph{Koszul sign rule}
\[(12)\cdot(a\otimes b) \coloneq (-1)^{|a||b|} b\otimes a\]
and the  \emph{Koszul sign convention} lies in the following definition of the tensor product $f\otimes g$ of two maps
\[(f\otimes g)(a\otimes b)\coloneq (-1)^{|a||g|}f(a) \otimes g(b)~.\]
The linear dual of dg vector space $A$ is defined by $(A^*)_n \coloneq \Hom (A_{-n},\field)$ and by $d_{A^*}(f)\coloneq -(-1)^{|f|}f \circ d_A$. 
We denote the homological suspension (resp. desuspension) by $\susp$ (resp. $\asusp$), i.e. the one-dimensional dg vector space concentrated in degree $1$ (resp. $-1$). Notice that $\susp^*=\asusp$~, and thus $\asusp \susp=1=-\susp\asusp$~.

\medskip

We denote the symmetric groups by $\Sy_n$ and the cyclic groups by $\C_n$, which generating cycle $\tau_n\coloneq (12\ldots n)$~. 
We will represent graphically  the various cyclic constructions in a planar way, for instance with boxes instead of ringed bands, in order to address the main issue which is to compute signs coming from  the Koszul sign rule and convention. 
We follow the conventions for operads  given  in  \cite{LV12}. 

\medskip

\paragraph*{\bf Acknowledgements}
We would like to acknowledge enlightening discussions with 
Marion Boucrot, Damien Calaque, Ricardo Campos, Estanislao Herscovich, Eric Hoffbeck, Guillaume Laplante-Anfossi, Geoffrey Powell, Salim Rivi\`ere, Victor Roca i Lucio,  Vladimir Rubtsov, and Alex Takeda. 
We express our deep appreciation to all the participants of the Villaroger 2021 workshop on double Poisson structures. 
The first author thanks Catherine and Victor for the warm welcomes at home and Charles De Clercq for introducing him to the great poetry.
The second author would like to thank Catherine for her help with the terminology and Victor for his help with the figures. 

\section{Double Poisson gebras up to homotopy}\label{sec:ProperadDoublePoissInfty}

In this section, we first recall the notion of a \emph{double Poisson gebra} before to making explicit  its  homotopy version as prescribed by the Koszul duality of properads. To ease the reading, we provide the reader with two subsections recalling the basic properties of properads and the deformation theory of their morphisms. One subsection is devoted to the computation of the Koszul dual (co)properad of double Poisson gebras. 
Finally, we go one step further by settling the even more general notion of a \emph{curved homotopy double Poisson gebra}.

\subsection{Double Poisson gebra}

\begin{definition}[Double Lie gebra]\label{def:gebra_structDLIE}
	A \emph{double Lie gebra} amounts to a data $(A,\DB{-}{-})$ made up of a
	dg vector space $A$ and a morphism of chain complexes
	\[
		\DB{-}{-}\colon A\otimes A \too A \otimes A~,
	\]
	called the \emph{double bracket}, satisfying the following two relations, where we use the Sweedler notation
	\[
		\DB{a}{b}=\DB{a}{b}'\otimes\DB{a}{b}''
	\]
	for any $a, b\in A$~.
	\begin{description}
		\item[\sc Skew-symmetry]
		      \[
			      \DB{a}{b} = -(-1)^{|a||b|+|\DB{a}{b}'||\DB{a}{b}''|} \DB{b}{a}''\otimes \DB{b}{a}'
		      \]

		\item[\sc Double Jacobi relation]
		      \[
			      \DDBL{a}{b}{c}
			      + (-1)^{|a|(|b|+|c|)}\, (123)\cdot\DDBL{b}{c}{a}
			      + (-1)^{|c|(|a|+|b|)}\, (132)\cdot\DDBL{c}{a}{b}
			      =0 \ ,
		      \]
		      where $\DDBL{a}{b}{c} \coloneq \{\!\!\{a,\DB{b}{c}'\}\!\!\} \otimes \DB{b}{c}''$ and where
		      \[
			      \begin{aligned}
				      (123)\cdot(u\otimes v\otimes w)\coloneq
				       & (-1)^{|w|(|u|+|v|)}w\otimes u\otimes v \ , \\
				      (132)\cdot(u\otimes v\otimes w)\coloneq
				       & (-1)^{|u|(|v|+|w|)}v\otimes w\otimes u \ .
			      \end{aligned}
		      \]
	\end{description}
\end{definition}

\begin{definition}[Double Poisson gebra \cite{VdB08}]\label{def:gebra_struct}
	A \emph{double Poisson gebra} is a triple $(A, \mu,\DB{-}{-})$ made up of a dg associative product  $\mu\colon A \otimes A \to A$ and a double Lie gebra structure $(A, \DB{-}{-})$ satisfying the following compatibility relation.
	\begin{description}
		\item[\sc Derivation]
		      \[
			      \DB{a}{\mu(b,c)} =
			      (-1)^{|a||b|} \mu(b, \DB{a}{c}') \otimes \DB{a}{c}''
			      + \DB{a}{b} '\otimes \mu(\DB{a}{b}'',c)
		      \]
	\end{description}
\end{definition}

\begin{example}\leavevmode
	\begin{enumerate}
		\item In \cite[Section 6]{VdB08}, the author shows that for every finite quiver $Q$, one can define a double Poisson bracket on the path algebra of the double of $Q$.

		\item In \cite{ORS13}, the authors classify double Poisson brackets on the free noncommutative associative algebra $\bbC\langle x,y\rangle$.

		\item For other significant examples, the interested reader can consult the exhaustive website \cite{FaiWeb} of M. Fairon, who has listed all the articles about double brackets.
	\end{enumerate}
\end{example}

\subsection{Properads}\label{subsec:Properads}
In this section, we  briefly recall the definitions of a properad and a coproperad, following mainly the presentation of \cite[Section~2]{HLV20}.
Let $\Bij$ be the groupoid of non empty finite sets with bijections.

\begin{definition}[$\Sy$-bimodule]
	A \emph{$\Sy$-bimodule} is a module over the groupoid $\Bij\times\Bij^\op$. The associated category is denoted by $\Sbimod$.
\end{definition}

\begin{example}
	Let $A$ and $B$ be two dg vector spaces. We define the $\Sy$-bimodule $\End^A_B$ by
	\[
		\End^A_B(Y,X)\coloneq \Hom\left(\bigotimes_{x\in X} A , \bigotimes_{y\in Y} B\right) \ ,
	\]
	for any non-empty finite sets $X$ and $Y$.
\end{example}

The groupoid $\Bij$ admits for skeletal category the one made up of the sets $\underline{n}\coloneq \{ 1,\cdots, n \}$ equipped with the symmetric groups $\Sy_n$ for automorphisms.
So  the data of a $\Sy$-bimodule $\rmM$ is equivalent to a collection $\{\rmM(m,n)\}_{m,n\in \NN^*}$ of dg vector spaces equipped with two compatible actions of the symmetric groups, one of $\Sy_m$ on the left and one of $\Sy_n$ on the right, under the formula
\[
	\rmM(Y,X)\coloneq
	\left(\prod_{f\in \Bij\times\Bij^\op(\underline{m}\times \underline{n},\, Y\times X)} \rmM(m,n)\right)\big/\sim\ ,
\]
where $|X|=n, |Y|=m$
and where
$
	(f, \mu)\sim
	\big(
	g, {g}^{-1}f\cdot\mu
	\big)~.
$\\

We consider the set $\rmG$ of connected graphs directed by a global flow and the endofunctor
$\scrG\colon \Sbimod \to \Sbimod$ defined by
\[
	\scrG(\rmM)(m,n)\coloneq \bigoplus_{\rmg\in \rmG(m,n)} \rmg(\rmM)
\]
where $\rmg(\rmM)\coloneq \bigotimes_{v\in \mathrm{vert}(\rmg)}\rmM(m(v), n(v))$~,
where $n(v)$ stands for the number of inputs of the vertex $v$ and where $m(v)$
stands for the number of outputs of the vertex $v$~.
The operation of forgetting the nesting of connected flow-directed graphs in $\scrG(\scrG(\rmM))$, produces elements of $\scrG(\rmM)$ and thus induces a monad structure on $\scrG$. We call this monad the \emph{monad of connected flow-directed graphs}.

\begin{definition}[Properad]
	A \emph{properad} is an algebra over the monad $\scrG$ of connected flow-directed graphs.
\end{definition}

This definition of a properad is actually not the original one. In \cite{Val07}, the second author defined a monoidal product $\boxtimes$ called \emph{the connected composition product} on the category $\Sbimod$, which amounts to composing operations along  2-level connected flow-directed graphs, as it is illustrated in \Cref{fig:two_levelled_graph}. We refer the reader to \cite[Definition~1.13]{HLV20} for a formal definition, see also \cite[Definition~3.11]{Ler19i}. Its unit $\rmI$ is the $\Sy$-bimodule made up of a one-dimensional space in arity $(1,1)$.

\begin{figure*}[h!]
	\begin{tikzpicture}[scale=0.7]
		\draw[thick] (2,1) to[out=270,in=90] (6,-1) to[out=270,in=90] (7,-2);
		\draw[thick] (1,1) to[out=270,in=90] (3,-1);
		\draw[thick] (0,-1) to[out=270,in=90] (1,-2);
		\draw[thick] (1,2) to[out=270,in=90] (0,1) to[out=270,in=90] (1,-1) -- (3,-1) to[out=270,in=90] (2,-2);
		\draw[thick] (8,2) to[out=270,in=90] (6,1) to[out=270,in=90] (8,-1) to[out=270,in=90] (10,-2);
		\draw[thick] (3,2) to[out=270,in=90] (5,1) ;
		\draw[thick] (10,1) to[out=270,in=90] (9,-1);
		\draw[draw=white,double=black,double distance=2*\pgflinewidth,thick] (1,-1) to[out=270,in=90] (0,-2);
		\draw[draw=white,double=black,double distance=2*\pgflinewidth,thick] (9,-1) to[out=270,in=90] (8,-2);
		\draw[draw=white,double=black,double distance=2*\pgflinewidth,thick] (7,-1) to[out=270,in=90] (4,-2);
		\draw[draw=white,double=black,double distance=2*\pgflinewidth,thick] (7,2) to[out=270,in=90]  (9,1) to[out=270,in=90] (10,-1);
		\draw[draw=white,double=black,double distance=2*\pgflinewidth,thick] (0,2) to[out=270,in=90] (2,1);
		\draw[draw=white,double=black,double distance=2*\pgflinewidth,thick] (8,1) to[out=270,in=90] (4,-1);
		\draw[draw=white,double=black,double distance=2*\pgflinewidth,thick] (5,2) to[out=270,in=90] (4,1) to[out=270,in=90] (0,-1);
		\draw[fill=white] (-0.3,0.8) rectangle (2.3,1.2);
		\draw[fill=white] (3.7,0.8) rectangle (6.3,1.2);
		\draw[fill=white] (7.7,0.8) rectangle (10.3,1.2);
		\draw[fill=white] (-0.3,-1.2) rectangle (4.3,-0.8);
		\draw[fill=white] (5.7,-1.2) rectangle (10.3,-0.8);
		\draw (1,1) node {\small{$\nu_1$}};
		\draw (5,1) node {\small{$\nu_2$}};
		\draw (9,1) node {\small{$\nu_3$}};
		\draw (2,-1) node {\small{$\mu_1$}};
		\draw (8,-1) node {\small{$\mu_2$}};
	\end{tikzpicture}
	\caption{An element of $\rmM\boxtimes \rmN$.}
	\label{fig:two_levelled_graph}
\end{figure*}
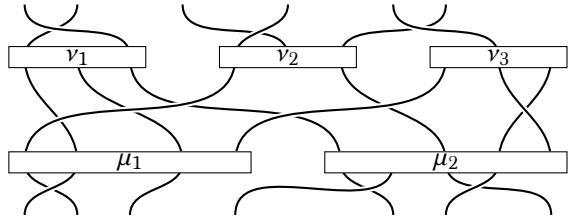

\begin{proposition}[\cite{Val07}]
	The category of algebras over the monad $\scrG$ is isomorphic to the category of monoids with respect to  the connected composition product:
	\[
		\catofalgebras{\scrG} \cong 	\catofmon{\Sbimod,\boxtimes,\Ibox}\ .
	\]
\end{proposition}

\begin{example}
	For any  dg vector space $A$, we consider the properad structure on the $\Sy$-bimodule $\End_A\coloneq\End_A^A$ defined by  the composite of maps. It is called the \emph{endomorphism properad} of $A$.
	For instance, given two maps $f : A^{\otimes n} \to A^{\otimes m}$ and $g : A^{\otimes n'} \to A^{\otimes m'}$, their partial composite $f \circ_i^j g$ along the $i$th input of $f$ and the $j$th output of $g$ is organised as follows.

	\begin{figure*}[h!]
		\begin{tikzpicture}[scale=0.8,baseline =(n.base)]
			\node (n) at (0,0) {};
			%%% input g
			\draw[thick]
			(-0.4,2) -- (-0.4,3)
			(0.2,2) -- (0.2,3)
			(0.8,2) -- (0.8,3)
			(1.4,2) -- (1.4,3);
			%%% output f
			\draw[thick]
			(-0.5,-2) -- (-0.5,-3)
			(-1.25,-2) -- (-1.25,-3)
			(0.25,-2) -- (0.25,-3);
			%%% output \circ^j_i
			\draw[thick]
			(0,-1) -- (0,1);
			%%% output g
			\draw[thick]
			(-0.5,1) to[out=270,in=90] (-4,-3)
			(0.5,1) to[out=270,in=90] (3,-3)
			(1,1) to[out=270,in=90] (3.5,-3)
			(1.5,1) to[out=270,in=90] (4,-3);
			%%% input f
			\draw[draw=white,double=black,double distance=2*\pgflinewidth,thick]
			(-4,3) to[out=270,in=90] (-1.5,-1)
			(-3.5,3) to[out=270,in=90] (-1,-1)
			(-3,3) to[out=270,in=90] (-0.5,-1)
			(4,3) to[out=270,in=90] (0.5,-1);
			%%%  labels
			\draw
			(0,-1) node[above right] {$\scriptstyle{i}$}
			(0,1) node[below right] {$\scriptstyle{j}$};
			%%% RECTANGLES
			\draw[fill=white]
			(-2,-2) rectangle (1,-1) 	node[midway] 	{${f}$}
			(-1,1) rectangle (2,2) 	node[midway] 	{${g}$};
		\end{tikzpicture}
		\caption{The partial composite of two multilinear maps.}
		\label{fig:PartCompoEnd}
	\end{figure*}
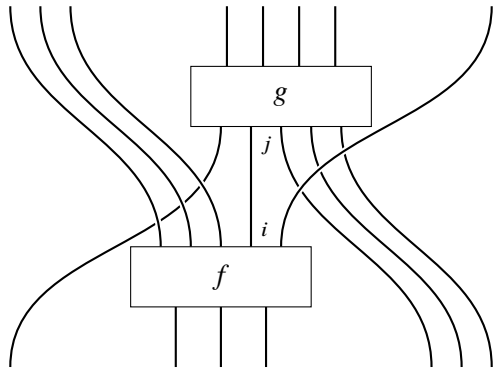

\end{example}

The most important  examples of properads used in this paper are the  following two ones.

\begin{definition}[$\DLie$ and $\DPois$]
	The properad $\DLie$ is defined by the following presentation:
	\[
		\DLie \coloneq
		\dfrac{
			\scrG\left(
			\begin{tikzpicture}[baseline=0ex,scale=0.2]
				\draw (0,0)  node[below] {$\scriptscriptstyle{1}$}
				-- (0,1.5)node[above] {$\scriptscriptstyle{1}$} ;
				\draw (2,0) node[below] {$\scriptscriptstyle{2}$}
				-- (2,1.5) node[above] {$\scriptscriptstyle{2}$};
				\draw[fill=white] (-0.3,0.5) rectangle (2.3,1);
			\end{tikzpicture}
			\ = \ - \
			\begin{tikzpicture}[baseline=0ex,scale=0.2]
				\draw (0,0)  node[below] {$\scriptscriptstyle{2}$}
				-- (0,1.5)node[above] {$\scriptscriptstyle{2}$} ;
				\draw (2,0) node[below] {$\scriptscriptstyle{1}$}
				-- (2,1.5) node[above] {$\scriptscriptstyle{1}$};
				\draw[fill=white] (-0.3,0.5) rectangle (2.3,1);
			\end{tikzpicture}
			\right)
		}{
			\left(
			\begin{tikzpicture}[scale=0.2,baseline=-1]
				\draw[thin] (0,-1) node[below] {$\scriptscriptstyle{1}$}
				-- (0,1.5) node[above] {$\scriptscriptstyle{1}$};
				\draw[thin] (2,-1) node[below] {$\scriptscriptstyle{2}$}
				-- (2,1.5) node[above] {$\scriptscriptstyle{2}$};
				\draw[thin] (4,-1) node[below] {$\scriptscriptstyle{3}$}
				-- (4,1.5) node[above] {$\scriptscriptstyle{3}$};
				\draw[fill=white] (1.7,0.5) rectangle (4.3,1);
				\draw[fill=white] (-0.3,-0.5) rectangle (2.3,0);
			\end{tikzpicture}
			~+~
			\begin{tikzpicture}[scale=0.2,baseline=-1]
				\draw[thin] (0,-1) node[below] {$\scriptscriptstyle{2}$}
				-- (0,1.5) node[above] {$\scriptscriptstyle{2}$};
				\draw[thin] (2,-1) node[below] {$\scriptscriptstyle{3}$}
				-- (2,1.5) node[above] {$\scriptscriptstyle{3}$};
				\draw[thin] (4,-1) node[below] {$\scriptscriptstyle{1}$}
				-- (4,1.5) node[above] {$\scriptscriptstyle{1}$};
				\draw[fill=white] (1.7,0.5) rectangle (4.3,1);
				\draw[fill=white] (-0.3,-0.5) rectangle (2.3,0);
			\end{tikzpicture}
			~+~
			\begin{tikzpicture}[scale=0.2,baseline=-1]
				\draw[thin] (0,-1) node[below] {$\scriptscriptstyle{3}$}
				-- (0,1.5) node[above] {$\scriptscriptstyle{3}$};
				\draw[thin] (2,-1) node[below] {$\scriptscriptstyle{1}$}
				-- (2,1.5) node[above] {$\scriptscriptstyle{1}$};
				\draw[thin] (4,-1) node[below] {$\scriptscriptstyle{2}$}
				-- (4,1.5) node[above] {$\scriptscriptstyle{2}$};
				\draw[fill=white] (1.7,0.5) rectangle (4.3,1);
				\draw[fill=white] (-0.3,-0.5) rectangle (2.3,0);
			\end{tikzpicture}
			\right)
		}~,
	\]
	where the generator has degree $0$. The properad $\DPois$ is defined by the following presentation:
	\[
		\DPois \coloneq
		\dfrac{
			\scrG\left(
			\begin{tikzpicture}[baseline=0ex,scale=0.10]
				\draw (0,4) node[above] {$\scriptscriptstyle{1}$}
				-- (2,2);
				\draw (4,4) node[above] {$\scriptscriptstyle{2}$}
				-- (2,2) -- (2,0) node[below] {$\scriptscriptstyle{1}$};
				\draw[fill=white] (2,2) circle (10pt);
			\end{tikzpicture}
			~ ; ~
			\begin{tikzpicture}[baseline=0ex,scale=0.2]
				\draw (0,0)  node[below] {$\scriptscriptstyle{1}$}
				-- (0,1.5)node[above] {$\scriptscriptstyle{1}$} ;
				\draw (2,0) node[below] {$\scriptscriptstyle{2}$}
				-- (2,1.5) node[above] {$\scriptscriptstyle{2}$};
				\draw[fill=white] (-0.3,0.5) rectangle (2.3,1);
			\end{tikzpicture}
			\ = \ - \
			\begin{tikzpicture}[baseline=0ex,scale=0.2]
				\draw (0,0)  node[below] {$\scriptscriptstyle{2}$}
				-- (0,1.5)node[above] {$\scriptscriptstyle{2}$} ;
				\draw (2,0) node[below] {$\scriptscriptstyle{1}$}
				-- (2,1.5) node[above] {$\scriptscriptstyle{1}$};
				\draw[fill=white] (-0.3,0.5) rectangle (2.3,1);
			\end{tikzpicture}
			\right)
		}{
			\left(
			\begin{tikzpicture}[baseline=0.5ex,scale=0.1]
				\draw (0,4) node[above] {$\scriptscriptstyle{1}$} --(4,0)
				-- (4,-1) node[below] {$\scriptscriptstyle{1}$};
				\draw (4,4) node[above] {$\scriptscriptstyle{2}$} -- (2,2);
				\draw (8,4) node[above] {$\scriptscriptstyle{3}$}-- (4,0);
				\draw[fill=white] (2,2) circle (10pt);
				\draw[fill=white] (4,0) circle (10pt);
			\end{tikzpicture}
			-
			\begin{tikzpicture}[baseline=0.5ex,scale=0.1]
				\draw (0,4) node[above] {$\scriptscriptstyle{1}$}
				-- (4,0) -- (4,-1) node[below] {$\scriptscriptstyle{1}$};
				\draw (4,4) node[above] {$\scriptscriptstyle{2}$} -- (6,2);
				\draw (8,4) node[above] {$\scriptscriptstyle{3}$} -- (4,0);
				\draw[fill=white] (6,2) circle (10pt);
				\draw[fill=white] (4,0) circle (10pt);
			\end{tikzpicture}
			~~;~~
			\begin{tikzpicture}[scale=0.2,baseline=1ex]
				\draw (0,0) node[below] {$\scriptscriptstyle{1}$} -- (0,1)
				to[out=90,in=270] (-1,2.5) node[above] {$\scriptscriptstyle{1}$};
				\draw (2,0) node[below] {$\scriptscriptstyle{2}$}
				-- (2,1.5) -- (1,2.5) node[above] {$\scriptscriptstyle{2}$};
				\draw (2,1.5) -- (3,2.5) node[above] {$\scriptscriptstyle{3}$};
				\draw[fill=white] (2,1.5) circle (6pt);
				\draw[fill=white] (-0.3,0.5) rectangle (2.3,1);
			\end{tikzpicture}
			-
			\begin{tikzpicture}[scale=0.2,baseline=1ex]
				\draw (1,0.5) -- (2,1.5) -- (1,0.5) -- (0,1.5) -- (1,0.5)
				-- (1,-0.5) node[below] {$\scriptscriptstyle{1}$};
				\draw (0,2.5) node[above] {$\scriptscriptstyle{2}$}
				-- (0,1.5);
				\draw (2,2.5) node[above] {$\scriptscriptstyle{1}$}
				-- (2,1.5);
				\draw (4,2) -- (4,2.5) node[above] {$\scriptscriptstyle{3}$};
				\draw (3,-0.5) node[below] {$\scriptscriptstyle{2}$}
				to[out=90,in=270] (4,1.5) ;
				\draw[fill=white] (1,0.5) circle (6pt);
				\draw[fill=white] (1.7,1.5) rectangle (4.3,2);
			\end{tikzpicture}
			-
			\begin{tikzpicture}[scale=0.2,baseline=1ex]
				\draw (1,0.5) -- (2,1.5) -- (1,0.5) -- (0,1.5) -- (1,0.5)
				-- (1,-0.5) node[below] {$\scriptscriptstyle{2}$};
				\draw (0,2.5) node[above] {$\scriptscriptstyle{2}$} -- (0,1.5);
				\draw (2,2.5) node[above] {$\scriptscriptstyle{3}$} -- (2,1.5);
				\draw (-2,2) -- (-2,2.5) node[above] {$\scriptscriptstyle{1}$};
				\draw (-1,-0.5)  node[below] {$\scriptscriptstyle{1}$}
				to[out=90,in=270] (-2,1.5);
				\draw[fill=white] (1,0.5) circle (6pt);
				\draw[fill=white] (-2.3,1.5) rectangle (0.3,2);
			\end{tikzpicture}
			~~;~~
			\begin{tikzpicture}[scale=0.2,baseline=-1]
				\draw[thin] (0,-1) node[below] {$\scriptscriptstyle{1}$}
				-- (0,1.5) node[above] {$\scriptscriptstyle{1}$};
				\draw[thin] (2,-1) node[below] {$\scriptscriptstyle{2}$}
				-- (2,1.5) node[above] {$\scriptscriptstyle{2}$};
				\draw[thin] (4,-1) node[below] {$\scriptscriptstyle{3}$}
				-- (4,1.5) node[above] {$\scriptscriptstyle{3}$};
				\draw[fill=white] (1.7,0.5) rectangle (4.3,1);
				\draw[fill=white] (-0.3,-0.5) rectangle (2.3,0);
			\end{tikzpicture}
			+
			\begin{tikzpicture}[scale=0.2,baseline=-1]
				\draw[thin] (0,-1) node[below] {$\scriptscriptstyle{2}$}
				-- (0,1.5) node[above] {$\scriptscriptstyle{2}$};
				\draw[thin] (2,-1) node[below] {$\scriptscriptstyle{3}$}
				-- (2,1.5) node[above] {$\scriptscriptstyle{3}$};
				\draw[thin] (4,-1) node[below] {$\scriptscriptstyle{1}$}
				-- (4,1.5) node[above] {$\scriptscriptstyle{1}$};
				\draw[fill=white] (1.7,0.5) rectangle (4.3,1);
				\draw[fill=white] (-0.3,-0.5) rectangle (2.3,0);
			\end{tikzpicture}
			+
			\begin{tikzpicture}[scale=0.2,baseline=-1]
				\draw[thin] (0,-1) node[below] {$\scriptscriptstyle{3}$}
				-- (0,1.5) node[above] {$\scriptscriptstyle{3}$};
				\draw[thin] (2,-1) node[below] {$\scriptscriptstyle{1}$}
				-- (2,1.5) node[above] {$\scriptscriptstyle{1}$};
				\draw[thin] (4,-1) node[below] {$\scriptscriptstyle{2}$}
				-- (4,1.5) node[above] {$\scriptscriptstyle{2}$};
				\draw[fill=white] (1.7,0.5) rectangle (4.3,1);
				\draw[fill=white] (-0.3,-0.5) rectangle (2.3,0);
			\end{tikzpicture}
			\right)
		}~,
	\]
	where the generators have degree $0$.
\end{definition}

\begin{definition}[$\prop$-gebra]
	Let $A$ be a dg vector space and let $\prop$ be a properad. A \emph{structure of $\prop$-gebra on $A$} is a morphism of properads $\prop \to \End_A$~.
\end{definition}

\begin{lemma}
	There is a one-to-one correspondence between $\DLie$-gebra
	(respectively $\DPois$-gebra) structures
	and double Lie gebra
	(respectively double Poisson gebra)
	structures.
\end{lemma}

\begin{proof}
	This is straightforward.
\end{proof}

The notion and properties of homotopy $\prop$-gebra are efficiently encoded by the dual notion of a coproperad.

\begin{definition}[Coproperad]
	A \emph{coproperad} is a comonoid in the monoidal category $(\Sbimod,\boxtimes,\Ibox)$.
\end{definition}

\begin{example}
	The Koszul duality theory for properads \cite{Val07} provides us with Koszul dual coproperads. In this article, we will mainly consider the  coproperads
	$\DLie^\antish$ and $\DPois^\antish$, which are  Koszul dual to  the abovementioned  properads, see \cref{subsec::DPois_antish}.
\end{example}

Dual to the monad $\scrG$ of connected flow-directed graphs, we consider the
\emph{comonad $\scrG^c$ of reduced connected flow-directed graphs} $\bGs\coloneq \Gs \backslash \{| \}$:
\[\scrG^c(\rmM) \coloneq \bigoplus_{\gg \in \bGs} \rmg(\rmM) \ , \]
with  coproduct given by   the sum of all the ways to partition the underlying graph $\gg$ into connected 
flow-directed sub-graphs.
A coalgebra over this comonad is called a \emph{comonadic coproperad} in \cite[Section~2.2]{HLV20}.
Adding a coaugmented counit to comonadic coproperad produces coproperads which are called \emph{conilpotent}.

\begin{example}
	The Koszul dual coproperads $\DLie^\antish$ and $\DPois^\antish$ are conilpotent.
\end{example}

The linear dual of a coproperad carries a canonical  structure of a properad. The reverse statement does not hold true in general; one can however recover in this case some interesting coproperadic structure  as follows.
Recall that the (reduced) graph (co)monad admits a presentation with generators given by the summand made up of graphs with two vertices: 
\[\rmM \ibt \rmM\coloneq \scrG(\rmM)^{(2)}\cong \scrG^c(\rmM)^{(2)}~.\]
\begin{definition}[{Infinitesimal} coproperad]
	An \emph{{infinitesimal} coproperad} is an $\Sy$-bimodule $\rmC$ equipped with a  morphism of $\Sy$-bimodules
	\[\Delta_{(1,1)} \colon\rmC \to \rmC \ibt \rmC~,\] 
	which satisfies the properties of the restriction of a comonadic product, see  \cite[Section~2]{MV09I}.
\end{definition}

By definition, any (comonadic) coproperad carries a canonical {infinitesimal} coproperad structure. 
The reverse holds true if and only if  the 
 iterations of the infinitesimal decomposition map $\Delta_{(1,1)}$ on any element of $\rmC$ produce finite sums of labelled graphs.
We already refer the reader to \cref{subsec:CurvedHDPois} for the main example of infinitesimal coproperad considered here and to \cref{subsec:HoThCHDPois} for the way its crucial properties are used.

\begin{remark}
For simplicity, this section is written under the assumption that $n\geqslant 1$, but everything holds true \emph{mutatis mutandis} for $n\geqslant 0$~. We refer the reader to \cite[Section~1]{HLV20} for more details in this latter case. This setting will be mandatory in \cref{subsec:CurvedHDPois}.
\end{remark}

\begin{definition}[Dioperad/codioperad \cite{Gan03}]
Restricting the above-mentioned definitions to connected flow-directed graphs of genus 0, one gets the notion of a \emph{dioperad} and, dually, the notion of a \emph{codioperad}. 
\end{definition}

Any properad induces  a dioperad by forgetting the compositions of elements based on graphs of positive genera. Dually, a codioperad is a coproperad with decomposition map producing only graphs of genus 0. So the notion of a dioperad is a restriction of the notion of a properad whereas the dual notion of a codioperad is a coproperad satisfying some extra property. Notice that the notions of double Lie gebras and double Poisson gebras can be faithfully encoded by a dioperad. 
{The following proposition holds true for dioperads but fails in general for properads due to the existence of an infinite number of directed graphs with two vertices and arbitrary genus.}

\begin{proposition}\label{prop:PropDualPartCoprop}
	The linear dual of any arity-wise finite dimensional dioperad carries a canonical structure of  an infinitesimal codioperad.
\end{proposition}

\begin{proof}
	This is obtained in a  straightforward way by considering the linear dual
\[\rmP^* \to    \left(\scrG_{g=0}(\rmP)^{(2)}\right)^*\cong \scrG_{g=0}\left(\rmP^*\right)^{(2)}\]
	of the infinitesimal composition product
	$  \scrG_{g=0}(\rmP)^{(2)} \to \rmP$
	of the dioperad $\rmP$, where $\scrG_{g=0}(\rmP)^{(2)}$ stands for the module of  genus 0 direct graphs with 2 vertices labelled by $\rmP$, which is arity-wise finite-dimensional.
\end{proof}

\subsection{Deformation theory of morphisms of properads}\label{subsec::def_theo_properad}
In this section, the recall the deformation theory of morphisms of properads after \cite[Section~2]{MV09II}.

\begin{definition}[Totalisation of $\Sy$-bimodule]
	The \emph{totalisation} of an $\Sy$-bimodule $\rmM$ is the dg vector space defined by
	\[
		\widehat{\rmM}	\coloneq \prod_{m,n \in \NN^*} \rmM(m,n)^{\Sy_m^\op\times \Sy_n} \ .
	\]
\end{definition}

Recall that a binary product is called \emph{Lie-admissible} when its commutator satisfies the Jacobi relation, i.e. is a Lie bracket.

\begin{lemma}[{\cite[Proposition~6]{MV09I}}]
	The following assignment defines a functor from the category of properads to the category of Lie-admissible algebras
	\[
		\begin{array}{rcl}
			\mathsf{properads}             & \too        & \mathsf{Lie}\text{-}\mathsf{admissible~algebras} \\
			\big(\rmP, d_\rmP,\gamma \big) & \longmapsto & \left({\bigoplus_{m,n \in \NN^*} \rmP(m,n)^{\Sy_m^\op\times \Sy_n}}, \mathrm{d},\star \right)~,
		\end{array}
	\]
	where the differential $\mathrm{d}$ is induced by the differential $d_\rmP$ and where the binary product $\star$ is defined by
	\[
		\mu \star \nu \coloneq
		\sum_{\mathrm{g} \in \rmG^{(2)}}
		\gamma(\mathrm{g}(\mu,\nu))\ .
	\]
\end{lemma}

In the above lemma, we had to restrict ourselves to finite sums for the product $\star$ well-defined. However, for the some properads, like the convolution properads introduced below, one can extend this result to the totalisation, this is to the product of components. 

\begin{lemma}[{\cite[Lemma~2]{MV09I}}]
	The following assignment defines a functor
	\[
		\begin{array}{rcl}
			\mathsf{coproperads}^\op \times \mathsf{properads}  & \too        & \mathsf{properads}                                   \\
			\big((\rmC,d_\rmC,\Delta),(\rmP,d_\rmP,\gamma)\big) & \longmapsto & \big(\Hom(\rmC,\rmP),\partial, \gamma_{\Hom} \big)~,
		\end{array}
	\]
	where the symmetric groups act by conjugaison, and where the composition
	\[
		\gamma_{\Hom} \colon \Hom(\rmC,\rmP) \boxtimes \Hom(\rmC,\rmP) \too \Hom(\rmC,\rmP)
	\]
	is given by
	\[
		\rmC \xrightarrow{\Delta} \rmC \boxtimes \rmC
		\xrightarrow{(f_1,\ldots,f_k) \boxtimes (g_1,\ldots,g_l)}
		\rmP\boxtimes \rmP
		\xrightarrow{\gamma}
		\rmP \ ,
	\]
	for $f_1,\ldots,f_k,g_1,\ldots,g_l \in \Hom(\rmC,\rmP)$~.
\end{lemma}

\begin{definition}[Properadic convolution algebra]
The \emph{properadic convolution Lie-admissible algebra} is defined by 
	\[
		\widehat{\Hom}(\rmC,\rmP)\coloneq
		\left(
		\prod_{m,n\in \NN^*}\Hom_{\Sy_m^\op\times\Sy_n}(\rmC(m,n),\rmP(m,n)), \partial, \star
		\right) \ ,
	\]
	where the product $f \star g$ is equal to
	\[
		\rmC \xrightarrow{\Delta_{(1,1)}}
		\rmC \underset{(1,1)}{\boxtimes} \rmC
		\xrightarrow{f\underset{(1,1)}{\boxtimes} g}
		\rmP \underset{(1,1)}{\boxtimes} \rmP
		\xrightarrow{\gamma}
		\rmP \ .
	\]
\end{definition}

\begin{remark}\leavevmode

\begin{enumerate}
\item When the coproperad $\rmC$ is coaugmented, we rather consider its coaugmentation coideal $\overline{\rmC}$, which is the cokernel of its coaugmentation map, in the above definition.  

\item Since only the infinitesimal decomposition map $\Delta_{(1,1)}$ of the coproperad is involved in this construction, the same statement holds true when $\rmC$ is only an infinitesimal coproperad.
\end{enumerate}
\end{remark}

Recall that when $\rmP=\scrG(E)/(R)$ is a quadratic properad, its Koszul dual coproperad $\rmP^{\ac}$ is the conilpotent coproperad cogenerated by the suspension $\susp E$  and the double suspension $\susp^2R$, see \cite{Val07} for more details.

\begin{example}
	One important example of properadic convolution algebra in the present paper is the one associated to the  conilpotent coproperad $\DPois^{\ac}$, which is Koszul dual to the properad encoding double Poisson gebras, and the endomorphism properad $\End_A$. We denote it by
	\[
		\convDPois_A\coloneq \widehat{\Hom}\left({\DPois}^\antish , \End_A\right)~.
	\]
\end{example}

\begin{definition}[Properadic twisting morphism]
	The solutions to the Maurer--Cartan equation
	\begin{equation}\label{eq:MCProperad}
		\partial \alpha + \alpha \star \alpha=0
	\end{equation}
	in the properadic convolution algebra $\widehat{\Hom}(\rmC,\rmP)$ are called the  \emph{twisting morphisms} from $\rmC$ to $\rmP$; their set is denoted by $\Tw(\rmC,\rmP)$~.
\end{definition}

\begin{definition}[$\prop_\infty$-gebra]
	When $\prop$ is a Koszul properad, a gebra over the cobar resolution $\Omega \prop^{\ac} \xrightarrow{\sim} \prop$ is called
	a \emph{homotopy $\prop$-gebra} or \emph{$\prop_\infty$-gebra}.
\end{definition}

\begin{proposition}[{\cite[Proposition~17]{MV09I}}]\label{prop:twisting}
	Homotopy $\prop$-gebra structures on a dg vector space $A$ are in canonical one-to-one correspondence with
	twisting morphisms $\Tw(\rmP^\antish,\End_A)$~.
\end{proposition}

Let us apply this to the case of double Poisson gebras;
the most important result of the Ph.D. Thesis of the first author lies in the following statement.

\begin{theorem}[{\cite[Theorem~5.11]{Ler19ii}}] \label{thm:DPois_Koszul}
	The properads $\DLie$ and $\DPois$ are Koszul.
\end{theorem}

\begin{proof}
	Let us just mention that the strategy of this proof relies on the introduction and the development of a new notion dubbed \emph{protoperads}, see \cite{Ler19i}.
\end{proof}

Thanks to this result, we get a homotopy meaningful notion of $\DPois_\infty$-gebra as a
Maurer--Cartan element in the properadic convolution algebra $\convDPois_A$~.
It remains to make the former one explicit.

\subsection{Explicit description of the Koszul dual coproperad \texorpdfstring{$\DPois^\antish$}{DPois}} \label{subsec::DPois_antish}

The derivation relation between the double Lie bracket and the associative product in the definition of a double Poisson gebra induces a rewriting rule in the properad $\DPois$ between the sub-properad
$\DLie$ and the sub-operad $\As$, viewed as a sub-properad concentrated in output arity one.

\begin{proposition}[{\cite[Corollary~5.10]{Ler19ii}}]\label{prop:decompo_duale_DPois}
	This rewriting rule induces a distributive law. This implies that the properad of double Poisson gebra is canonically isomorphic to
	\begin{equation} \label{prop:DPois_iso_As_box_DLie}
		\DPois\cong \As \boxtimes \DLie
	\end{equation}
	and that its Koszul dual coproperad is canonically isomorphic to
	\begin{equation} \label{prop:DPois_iso_As_box_DLieKOSZUL}
		\DPois^{\rm\antish} \cong \DLie^\antish \boxtimes \As^\antish	\ .
	\end{equation}
\end{proposition}

Let $\rmP\coloneq \scrG(E)/(R)$ be a quadratic properad.
Since it is much easier to work with properads than with coproperads, we consider the arity wise linear dual of the Koszul dual coproperad $\rmP^!\coloneq \left(\rmP^{\ac}\right)^{*}$, that we call  the \emph{Koszul dual properad} of $\rmP$.

\begin{lemma}[{\cite[Corollary~7.12]{Val07}}]\label{lemm:KDproperad}
	When $E$ is finite dimensional, the Koszul dual properad admits the following presentation
	\begin{equation}\label{Eq:KoszulDual}
		\rmP^!\cong
		\scrG\left(\susp^{-1}E^*\right)/\left(\susp^{-2}R^\perp\right)\ .
	\end{equation}
\end{lemma}

\begin{proof}
	Notice that we slightly changed the definitions from \emph{loc. cit.}: we do not use any suspension properad here.
	For more details, we refer to the complete proof of \cite[Proposition~7.2.1]{LV12}, which applies \emph{mutatis mutandis} to the present case.
\end{proof}

We start by the description of the Koszul dual of the properad $\DLie=\scrG(E)/(R)$.
The $\Sy$-bimodule  $s^{-1}E^*$ has dimension $2$ with basis depicted by ``planar boxes''
\begin{equation}\label{eq:GenDLie}
	\begin{tikzpicture}[scale=0.4,baseline=(n.base))]
		\node (n) at (0,0) {};
		\draw (0,-.5) node[below] {$\scriptstyle{1}$}
		-- (0,1) node[above] {$\scriptstyle{1}$};
		\draw[thin] (2,-.5) node[below] {$\scriptstyle{2}$}
		-- (2,1) node[above] {$\scriptstyle{2}$};
		\draw[fill=gray!25] (-0.3,0) rectangle (2.3,0.5);
	\end{tikzpicture}
	\quad \text{and} \quad
	\begin{tikzpicture}[scale=0.4,baseline=(n.base))]
		\node (n) at (0,0) {};
		\draw (0,-.5) node[below] {$\scriptstyle{2}$}
		-- (0,1) node[above] {$\scriptstyle{1}$};
		\draw[thin] (2,-.5) node[below] {$\scriptstyle{1}$}
		-- (2,1) node[above] {$\scriptstyle{2}$};
		\draw[fill=gray!25] (-0.3,0) rectangle (2.3,0.5);
	\end{tikzpicture} \ ,
\end{equation}
where the action of $(12)$ on the left-hand side of the first term gives the second term and
where the action of $(12)$ on the right-hand side of the first term gives minus the second term. The $\Sy$-bimodule
$s^{-2}R^\perp$ is generated under the action of the symmetric groups by
\begin{equation}\label{eq:RelDLie}
	\begin{tikzpicture}[scale=0.4,baseline=(n.base)]
		\node (n) at (0,.5) {};
		\draw (0,-.5) node[below] {$\scriptstyle{1}$}
		-- (0,1) node[above] {$\scriptstyle{1}$};
		\draw[thin] (2,-.5) node[below] {$\scriptstyle{2}$}
		-- (2,2) node[above] {$\scriptstyle{2}$};
		\draw[thin] (4,.5) node[below] {$\scriptstyle{3}$}
		-- (4,2) node[above] {$\scriptstyle{3}$};
		\draw[fill=gray!25] (-0.3,0) rectangle (2.3,0.5);
		\draw[fill=gray!25] (1.7,1) rectangle (4.3,1.5);
	\end{tikzpicture}
	\ - \
	\begin{tikzpicture}[scale=0.4,baseline=(n.base)]
		\node (n) at (0,.5) {};
		\draw (0,-.5) node[below] {$\scriptstyle{2}$}
		-- (0,1) node[above] {$\scriptstyle{2}$};
		\draw[thin] (2,-.5) node[below] {$\scriptstyle{3}$}
		-- (2,2) node[above] {$\scriptstyle{3}$};
		\draw[thin] (4,.5) node[below] {$\scriptstyle{1}$}
		-- (4,2) node[above] {$\scriptstyle{1}$};
		\draw[fill=gray!25] (-0.3,0) rectangle (2.3,0.5);
		\draw[fill=gray!25] (1.7,1) rectangle (4.3,1.5);
	\end{tikzpicture}
	\quad \text{and} \quad
	\begin{tikzpicture}[scale=0.4,baseline=(n.base)]
		\node (n) at (0,.5) {};
		\draw (0,-.5) node[below] {$\scriptstyle{1}$}
		-- (0,2) node[above] {$\scriptstyle{1}$};
		\draw[thin] (2,-.5) node[below] {$\scriptstyle{2}$}
		-- (2,2) node[above] {$\scriptstyle{2}$};
		\draw[fill=gray!25] (-0.3,0) rectangle (2.3,0.5)
		(-0.3,1) rectangle (2.3,1.5);
	\end{tikzpicture} \ .
\end{equation}

\begin{remark}
	We depict the elements of properads $\rmP$ and their Koszul dual coproperad $\rmP^{\ac}$ by white boxes and
	we depict the elements of the Koszul dual properad $\rmP^!$ by gray boxes. In other words, elements of $E$ are in white and elements of $E^*$ are in gray.
\end{remark}

\begin{lemma}[{\cite[Section 5.1]{Ler19ii}}]\label{lemm:DLie!}
	The properad $\DLie^!$ is concentrated in arities $(n,n)$, for $n\geqslant 1$, where
	$\DLie^!(n,n)$ is concentrated in degree $1-n$ with a basis labelled by $(\Sy_n\times\Sy_n)/\C_n$
	given by the following stairways, up to cyclic permutations:
	
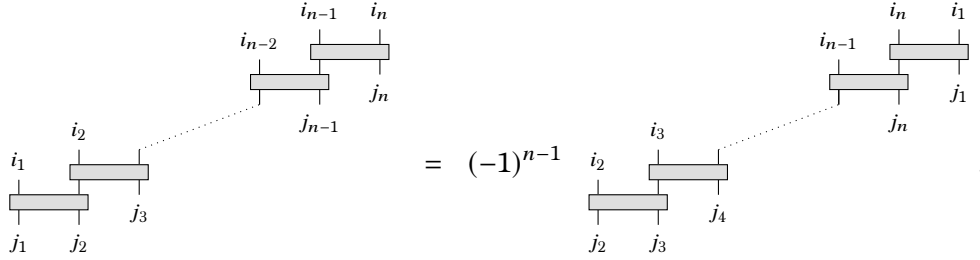
\[
		\begin{tikzpicture}[scale=0.4,baseline=1.2pc]
			\draw (0,-.5) node[below] {$\scriptstyle{j_1}$}
			-- (0,1) node[above] {$\scriptstyle{i_1}$};
			\draw[thin] (2,-.5) node[below] {$\scriptstyle{j_2}$}
			-- (2,2) node[above] {$\scriptstyle{i_2}$};
			\draw[thin] (4,.5) node[below] {$\scriptstyle{j_3}$}
			-- (4,2);
			\draw[fill=gray!25] (-0.3,0) rectangle (2.3,0.5);
			\draw[fill=gray!25] (1.7,1) rectangle (4.3,1.5);
			\draw [dotted] (4,2) -- (8,3.5);
			\draw (8,3.5)-- (8,5) node[above] {$\scriptstyle{i_{n-2}}$};
			\draw (10,3.5) node[below] {$\scriptstyle{j_{n-1}}$}
			-- (10,6) node[above] {$\scriptstyle{i_{n-1}}$};
			\draw (12,4.5)  node[below] {$\scriptstyle{j_n}$}
			-- (12,6) node[above] {$\scriptstyle{i_n}$};
			\draw (8,3.5) -- (8,4);
			\draw[fill=gray!25] (7.7,4) rectangle (10.3,4.5);
			\draw[fill=gray!25] (9.7,5) rectangle (12.3,5.5);
		\end{tikzpicture}
		\ \ = \ \ (-1)^{n-1} \ \
		\begin{tikzpicture}[scale=0.4,baseline=1.2pc]
			\draw (0,-.5) node[below] {$\scriptstyle{j_2}$}
			-- (0,1) node[above] {$\scriptstyle{i_2}$};
			\draw[thin] (2,-.5) node[below] {$\scriptstyle{j_3}$}
			-- (2,2) node[above] {$\scriptstyle{i_3}$};
			\draw[thin] (4,.5) node[below] {$\scriptstyle{j_4}$}
			-- (4,2);
			\draw[fill=gray!25] (-0.3,0) rectangle (2.3,0.5);
			\draw[fill=gray!25] (1.7,1) rectangle (4.3,1.5);
			\draw [dotted] (4,2) -- (8,3.5);
			\draw (8,3.5)-- (8,5) node[above] {$\scriptstyle{i_{n-1}}$};
			\draw (10,3.5) node[below] {$\scriptstyle{j_n}$}
			-- (10,6) node[above] {$\scriptstyle{i_n}$};
			\draw (12,4.5)  node[below] {$\scriptstyle{j_1}$}
			-- (12,6) node[above] {$\scriptstyle{i_1}$};
			\draw (8,3.5) -- (8,4);
			\draw[fill=gray!25] (7.7,4) rectangle (10.3,4.5);
			\draw[fill=gray!25] (9.7,5) rectangle (12.3,5.5);
		\end{tikzpicture}.
	\]
\end{lemma}

From this result, we deduce the following form of the Koszul dual coproperad $\DLie^{\ac}$. We denote the dual basis of the stairways by
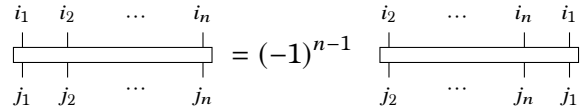
\[
	\begin{tikzpicture}[scale=0.4,baseline=(n.base)]
		\node (n) at (0,0.5) {};
		\draw[thin] (0,0) node[below] {$\scriptstyle{j_1}$}
		-- (0,1.5) node[above] {$\scriptstyle{i_1}$};
		\draw[thin] (1.5,0) node[below] {$\scriptstyle{j_2}$}
		-- (1.5,1.5) node[above] {$\scriptstyle{i_2}$};
		\draw (3.75,0) node[below] {$\scriptstyle{\cdots}$}
		(3.75,1.5) node[above] {$\scriptstyle{\cdots}$};
		\draw[thin]	(6,0) node[below] {$\scriptstyle{j_n}$}
		-- (6,1.5) node[above] {$\scriptstyle{i_n}$};
		\draw[fill=white] (-0.3,0.5) rectangle (6.3,1);
	\end{tikzpicture}
	=
	(-1)^{n-1} \ \
	\begin{tikzpicture}[scale=0.4,baseline=(n.base)]
		\node (n) at (0,0.5) {};
		\draw[thin] (0,0) node[below] {$\scriptstyle{j_2}$}
		-- (0,1.5) node[above] {$\scriptstyle{i_2}$};
		\draw (2.25,0) node[below] {$\scriptstyle{\cdots}$};
		\draw (2.25,1.5) node[above] {$\scriptstyle{\cdots}$};
		\draw[thin] (4.5,0) node[below] {$\scriptstyle{j_n}$}
		-- (4.5,1.5) node[above] {$\scriptstyle{i_n}$};
		\draw[thin]	(6,0) node[below] {$\scriptstyle{j_1}$}
		-- (6,1.5) node[above] {$\scriptstyle{i_1}$};
		\draw[fill=white] (-0.3,0.5) rectangle (6.3,1);
	\end{tikzpicture}
\]
understood to be of degree $n-1$. Their image under the infinitesimal decomposition map is equal to
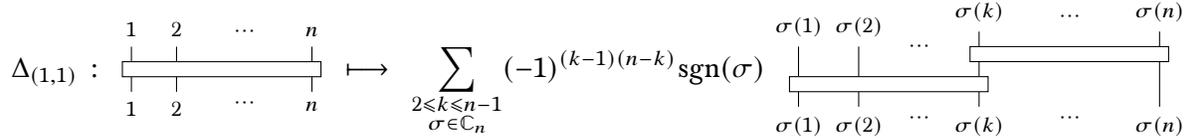
\[
	\Delta_{(1,1)}\ \colon   \
	\begin{tikzpicture}[scale=0.4,baseline=(n.base)]
		\node (n) at (0,0.5) {};
		\draw[thin] (0,0) node[below] {$\scriptstyle{1}$}
		-- ++(0,1.5) node[above] {$\scriptstyle{1}$};
		\draw[thin] (1.5,0) node[below] {$\scriptstyle{2}$}
		-- (1.5,1.5) node[above] {$\scriptstyle{2}$};
		\draw (3.75,0) node[below] {$\scriptstyle{\cdots}$}
		++(0,1.5) node[above] {$\scriptstyle{\cdots}$};
		\draw[thin]	(6,0) node[below] {$\scriptstyle{n}$}
		-- ++(0,1.5) node[above] {$\scriptstyle{n}$};
		\draw[fill=white] (-0.3,0.5) rectangle (6.3,1);
	\end{tikzpicture}
	\ \ {\longmapsto} \ \
	\sum_{{\substack{2 \leqslant k \leqslant n-1  \\ \sigma \in \C_n}}}
	(-1)^{(k-1)(n-k)}\sgn(\sigma)
	\begin{tikzpicture}[scale=0.4,baseline=(n.base)]
		\node (n) at (0,1) {};
		\draw[thin] (-2,0) node[below] {$\scriptstyle{\sigma(1)}$}
		-- (-2,2) node[above] {$\scriptstyle{\sigma(1)}$};
		\draw[thin] (0,0) node[below] {$\scriptstyle{\sigma(2)}$} --
		(0,2) node[above] {$\scriptstyle{\sigma(2)}$};
		\draw (2,0) node[below] {$\scriptstyle{\cdots}$}
		++(0,1.5) node[above] {$\scriptstyle{\cdots}$};
		\draw[thin]	(4,0) node[below] {$\scriptstyle{\sigma(k)}$}
		-- (4,2.5) node[above] {$\scriptstyle{\sigma(k)}$};
		\draw (7,0) node[below] {$\scriptstyle{\cdots}$}
		++(0,2.5) node[above] {$\scriptstyle{\cdots}$};
		\draw[thin]	(10,0) node[below] {$\scriptstyle{\sigma(n)}$}
		-- (10,2.5) node[above] {$\scriptstyle{\sigma(n)}$};
		\draw[fill=white] (-2.3,0.5) rectangle (4.3,1);
		\draw[fill=white] (3.7,1.5) rectangle (10.3,2);
	\end{tikzpicture}
\]

We are now ready to pass to the Koszul dual (co)properad of  $\DPois$.
Let us denote by $\mathrm{Part}_m(n)$ the set of ordered partitions $\lambda_1+\cdots+\lambda_m=n$ of  $n$ into $m$ positive integers.

\begin{lemma}\label{lemm:Dpois!}
	The properad $\DPois^!$ is concentrated in arities $(m,n)$, for $n\geqslant m\geqslant 1$, where
	$\DPois^!(m,n)$ is concentrated in degree $1-n$ with a basis labelled by $(\Sy_m\times\Sy_n\times \mathrm{Part}_m(n))/\C_m$
	given by the following treewise stairways
	
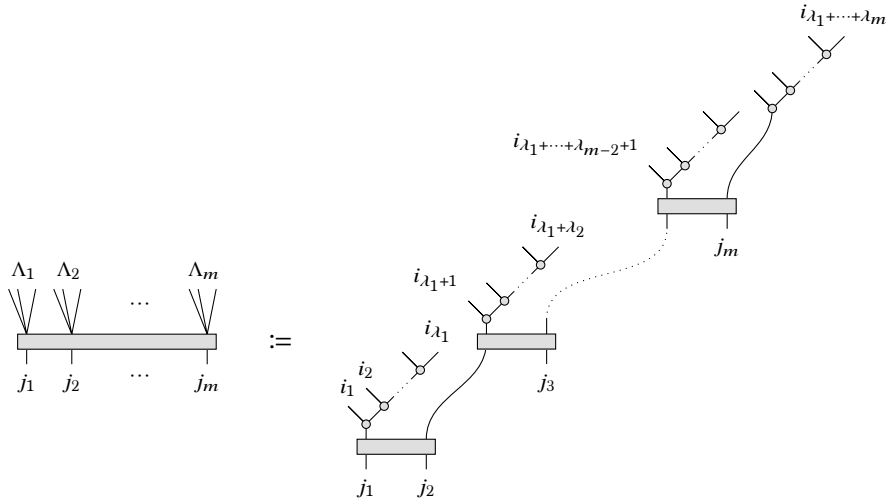
\[
		\begin{tikzpicture}[scale=0.4,baseline=(n.base)]
			\node (n) at (0,0.5) {};
			\coordinate (A) at (0,1);
			\draw[thin]
			(A) ++(0,-1) node[below] {$\scriptstyle{j_1}$} --
			(A) -- ++(-.6,1.5)
			(A) -- ++(-.3,1.5)
			(A) -- ++(.3,1.5)
			(A)++(-.1,1.5) node[above] {${\scriptstyle{\Lambda_1}}$};
			\coordinate (A) at (1.5,1);
			\draw[thin]
			(A) ++(0,-1) node[below] {$\scriptstyle{j_2}$} --
			(A) -- ++(-.6,1.5)
			(A) -- ++(-.3,1.5)
			(A) -- ++(.3,1.5)
			(A)++(-.1,1.5) node[above] {${\scriptstyle{\Lambda_2}}$};
			\draw (3.75,0) node[below] {$\scriptstyle{\cdots}$};
			\draw (3.75,1.5) node[above] {$\scriptstyle{\cdots}$};
			\coordinate (A) at (6,1);
			\draw[thin]
			(A) ++(0,-1) node[below] {$\scriptstyle{j_m}$} --
			(A) -- ++(-.6,1.5)
			(A) -- ++(-.3,1.5)
			(A) -- ++(.3,1.5)
			(A)++(-.1,1.5) node[above] {${\scriptstyle{\Lambda_m}}$};
			\draw[fill=gray!25] (-0.3,0.5) rectangle (6.3,1);
		\end{tikzpicture}
		\quad \coloneq \quad
		\begin{tikzpicture}[scale=0.4,baseline=(n.base)]
			\node (n) at (0,3.5) {};
			\def\sT{0.6} %taille des arbres
			% Premier arbre
			\coordinate (A) at (0, 1);
			\draw (0,-.5) node[below] {$\scriptstyle{j_1}$}
			-- (0, 0.5) to[out=90,in=270] (A) ;
			\draw (A)
			-- ++(-1*\sT,\sT) node[above] {$\scriptstyle{i_1}$}
			-- ++(\sT,-1*\sT)
			-- ++(\sT,\sT)
			-- ++(-1*\sT,\sT) node[above] {$\scriptstyle{i_2}$}
			-- ++(\sT,-1*\sT)
			-- ++(0.5*\sT,0.5*\sT);
			\draw[dotted] (A) ++(1.5*\sT,1.5*\sT) -- ++(\sT,\sT);
			\draw (A) ++(2.5*\sT,2.5*\sT) -- ++(.5*\sT,.5*\sT)
			-- ++(-1*\sT,\sT)
			-- ++(\sT,-1*\sT) -- ++(\sT,\sT) node[above] {$\scriptstyle{i_{\lambda_1}}$};
			\draw[fill=gray!25] (A) circle (4pt)
			++(\sT,\sT) circle (4pt)
			++(2*\sT,2*\sT) circle(4pt) ;
			% Second arbre
			\coordinate (B) at (4,4.5);
			\draw[thin] (2,-.5) node[below] {$\scriptstyle{j_2}$}
			-- (2,0.5)
			to[out=90,in=270] (4,3.75)
			to[out=90,in=270]  (B);
			\draw (B)
			-- ++(-1*\sT,\sT) node[above left] {$\scriptstyle{i_{\lambda_1+1}}$}
			-- ++(\sT,-1*\sT)
			-- ++(\sT,\sT)
			-- ++(-1*\sT,\sT)
			-- ++(\sT,-1*\sT)
			-- ++(0.5*\sT,0.5*\sT);
			\draw[dotted] (B) ++(1.5*\sT,1.5*\sT) -- ++(\sT,\sT);
			\draw (B) ++(2.5*\sT,2.5*\sT)
			-- ++(.5*\sT,.5*\sT)
			-- ++(-1*\sT,\sT)
			-- ++(\sT,-1*\sT)
			-- ++(\sT,\sT) node[above] {$\scriptstyle{i_{\lambda_1+\lambda_2}}$};
			\draw[fill=gray!25] (B) circle (4pt)
			++(\sT,\sT) circle (4pt)
			++(2*\sT,2*\sT) circle(4pt) ;
			\draw[thin] (6,3) node[below] {$\scriptstyle{j_3}$} -- (6,4.5);
			% Deux premiers rectangles
			\draw[fill=gray!25] (-0.3,0) rectangle (2.3,0.5);
			\draw[fill=gray!25] (3.7,3.5) rectangle (6.3,4);
			% Liaison
			\draw[dotted] (6,4.5) to[out=90,in=270] (10,7.5);
			\coordinate (C) at (10,9);
			\draw (10,7.5)  --  (C);
			\draw (C)
			-- ++(-1*\sT,\sT) node[above left] {$\scriptstyle{i_{\lambda_1+\cdots+\lambda_{m-2}+1}}$}
			-- ++(\sT,-1*\sT)
			-- ++(\sT,\sT)
			-- ++(-1*\sT,\sT)
			-- ++(\sT,-1*\sT)
			-- ++(0.5*\sT,0.5*\sT);
			\draw[dotted] (C) ++(1.5*\sT,1.5*\sT) -- ++(\sT,\sT);
			\draw (C) ++(2.5*\sT,2.5*\sT)
			-- ++(.5*\sT,.5*\sT)
			-- ++(-1*\sT,\sT)
			-- ++(\sT,-1*\sT)
			-- ++(\sT,\sT) ; %node[above left] {$\scriptstyle{i_{\lambda_1+\cdots+\lambda_{m-1}}}$};
			\draw[fill=gray!25] (C) circle (4pt)
			++(\sT,\sT) circle (4pt)
			++(2*\sT,2*\sT) circle(4pt) ;
			\coordinate (D) at (13.5,11.5);
			\draw (12,7.5) node[below] {$\scriptstyle{j_m}$}
			-- (12,8.5)
			to[out=90,in=270] (D);
			\draw (D)
			-- ++(-1*\sT,\sT)
			-- ++(\sT,-1*\sT)
			-- ++(\sT,\sT)
			-- ++(-1*\sT,\sT)
			-- ++(\sT,-1*\sT)
			-- ++(0.5*\sT,0.5*\sT);
			\draw[dotted] (D) ++(1.5*\sT,1.5*\sT) -- ++(\sT,\sT);
			\draw (D) ++(2.5*\sT,2.5*\sT)
			-- ++(.5*\sT,.5*\sT)
			-- ++(-1*\sT,\sT)
			-- ++(\sT,-1*\sT)
			-- ++(\sT,\sT) node[above] {$\scriptstyle{i_{\lambda_1+\cdots+\lambda_m}}$};
			\draw[fill=gray!25] (D) circle (4pt)
			++(\sT,\sT) circle (4pt)
			++(2*\sT,2*\sT) circle(4pt) ;
			% Dernier rectangle
			\draw[fill=gray!25] (9.7,8) rectangle (12.3,8.5);
		\end{tikzpicture},
	\]
	where ${\Lambda_1}\sqcup \cdots \sqcup {\Lambda_m}=\{1, \ldots, n\}$ is a partition, with $\lambda_l=|\Lambda_l|$ and $\Lambda_l\coloneq\allowbreak  (i_{\lambda_1+\cdots+\lambda_{l-1}+1}, \allowbreak \ldots, i_{\lambda_1+\cdots+\lambda_{l}})$,
	up to cyclic permutations:
	\[
		\begin{tikzpicture}[scale=0.4,baseline=(n.base)]
			\node (n) at (0,0.5) {};
			\coordinate (A) at (0,1);
			\draw[thin]
			(A) ++(0,-1) node[below] {${\scriptstyle{j_1}}$} --
			(A) -- ++(-.6,1.5)
			(A) -- ++(-.3,1.5)
			(A) -- ++(.3,1.5)
			(A)++(-.1,1.5) node[above] {${\scriptstyle{\Lambda_1}}$};
			\coordinate (A) at (1.5,1);
			\draw[thin]
			(A) ++(0,-1) node[below] {$\scriptstyle{j_2}$} --
			(A) -- ++(-.6,1.5)
			(A) -- ++(-.3,1.5)
			(A) -- ++(.3,1.5)
			(A)++(-.1,1.5) node[above] {${\scriptstyle{\Lambda_2}}$};
			\draw (3.75,0) node[below] {$\scriptstyle{\cdots}$};
			\draw (3.75,1.5) node[above] {$\scriptstyle{\cdots}$};
			\coordinate (A) at (6,1);
			\draw[thin]
			(A) ++(0,-1) node[below] {$\scriptstyle{j_m}$} --
			(A) -- ++(-.6,1.5)
			(A) -- ++(-.3,1.5)
			(A) -- ++(.3,1.5)
			(A)++(-.1,1.5) node[above] {${\scriptstyle{\Lambda_m}}$};
			\draw[fill=gray!25] (-0.3,0.5) rectangle (6.3,1);
		\end{tikzpicture}
		\ = \ (-1)^{n\lambda_1+\lambda_m} \
		\begin{tikzpicture}[scale=0.4,baseline=(n.base)]
			\node (n) at (0,0.5) {};
			\coordinate (A) at (0,1);
			\draw[thin]
			(A) ++(0,-1) node[below] {$\scriptstyle{j_2}$} --
			(A) -- ++(-.6,1.5)
			(A) -- ++(-.3,1.5)
			(A) -- ++(.3,1.5)
			(A)++(-.1,1.5) node[above] {${\scriptstyle{\Lambda_2}}$};
			\coordinate (A) at (4.5,1);
			\draw[thin]
			(A) ++(0,-1) node[below] {$\scriptstyle{j_m}$} --
			(A) -- ++(-.6,1.5)
			(A) -- ++(-.3,1.5)
			(A) -- ++(.3,1.5)
			(A)++(-.1,1.5) node[above] {${\scriptstyle{\Lambda_m}}$};
			\draw (2.25,0) node[below] {$\scriptstyle{\cdots}$}
			++(0,1.5) node[above] {$\scriptstyle{\cdots}$};
			\coordinate (A) at (6,1);
			\draw[thin]
			(A) ++(0,-1) node[below] {$\scriptstyle{j_1}$} --
			(A) -- ++(-.6,1.5)
			(A) -- ++(-.3,1.5)
			(A) -- ++(.3,1.5)
			(A)++(-.1,1.5) node[above] {${\scriptstyle{\Lambda_1}}$};
			\draw[fill=gray!25] (-0.3,0.5) rectangle (6.3,1);
		\end{tikzpicture} \ .
	\]
\end{lemma}

By convention, there are only binary trees, i.e. no box, for $m=1$~.

\begin{proof}
	This is a direct corollary of \cref{prop:decompo_duale_DPois} and \cref{lemm:DLie!}.
\end{proof}

For the Koszul dual coproperad $\DPois^{\ac}$, we represent the linear dual basis elements by
\begin{equation}\label{eq:SignCyclicDual}
	\begin{tikzpicture}[scale=0.4,baseline=(n.base)]
		\node (n) at (0,0.5) {};
		\coordinate (A) at (0,1);
		\draw[thin]
		(A) ++(0,-1) node[below] {${\scriptstyle{j_1}}$} --
		(A) -- ++(-.6,1.5)
		(A) -- ++(-.3,1.5)
		(A) -- ++(.3,1.5)
		(A)++(-.1,1.5) node[above] {${\scriptstyle{\Lambda_1}}$};
		\coordinate (A) at (1.5,1);
		\draw[thin]
		(A) ++(0,-1) node[below] {$\scriptstyle{j_2}$} --
		(A) -- ++(-.6,1.5)
		(A) -- ++(-.3,1.5)
		(A) -- ++(.3,1.5)
		(A)++(-.1,1.5) node[above] {${\scriptstyle{\Lambda_2}}$};
		\draw (3.75,0) node[below] {$\scriptstyle{\cdots}$};
		\draw (3.75,1.5) node[above] {$\scriptstyle{\cdots}$};
		\coordinate (A) at (6,1);
		\draw[thin]
		(A) ++(0,-1) node[below] {$\scriptstyle{j_m}$} --
		(A) -- ++(-.6,1.5)
		(A) -- ++(-.3,1.5)
		(A) -- ++(.3,1.5)
		(A)++(-.1,1.5) node[above] {${\scriptstyle{\Lambda_m}}$};
		\draw[fill=white] (-0.3,0.5) rectangle (6.3,1);
	\end{tikzpicture}
	\ = \
	(-1)^{n\lambda_1+\lambda_m} \
	\begin{tikzpicture}[scale=0.4,baseline=(n.base)]
		\node (n) at (0,0.5) {};
		\coordinate (A) at (0,1);
		\draw[thin]
		(A) ++(0,-1) node[below] {$\scriptstyle{j_2}$} --
		(A) -- ++(-.6,1.5)
		(A) -- ++(-.3,1.5)
		(A) -- ++(.3,1.5)
		(A)++(-.1,1.5) node[above] {${\scriptstyle{\Lambda_2}}$};
		\coordinate (A) at (4.5,1);
		\draw[thin]
		(A) ++(0,-1) node[below] {$\scriptstyle{j_m}$} --
		(A) -- ++(-.6,1.5)
		(A) -- ++(-.3,1.5)
		(A) -- ++(.3,1.5)
		(A)++(-.1,1.5) node[above] {${\scriptstyle{\Lambda_m}}$};
		\draw (2.25,0) node[below] {$\scriptstyle{\cdots}$}
		++(0,1.5) node[above] {$\scriptstyle{\cdots}$};
		\coordinate (A) at (6,1);
		\draw[thin]
		(A) ++(0,-1) node[below] {$\scriptstyle{j_1}$} --
		(A) -- ++(-.6,1.5)
		(A) -- ++(-.3,1.5)
		(A) -- ++(.3,1.5)
		(A)++(-.1,1.5) node[above] {${\scriptstyle{\Lambda_1}}$};
		\draw[fill=white] (-0.3,0.5) rectangle (6.3,1);
	\end{tikzpicture}
\end{equation}
and we consider the following notation for the most canonical ones:
\[
	\nu_{\lambda_1, \ldots, \lambda_m}\coloneq
	\begin{tikzpicture}[scale=0.9,baseline=(n.base)]
		\node (n) at (0,0.5) {};
		\coordinate (A) at (0,1);
		\draw[thin]
		(A) ++(0,-1) node[below] {${\scriptstyle{1}}$} --
		(A) -- ++(-.6,1)
		(A) -- ++(-.3,1)
		(A) -- ++(.3,1)
		(A)++(-.1,1) node[above right] {${\scriptstyle{\lambda_1}}$};
		\coordinate (A) at (1.5,1);
		\draw[thin]
		(A) ++(0,-1) node[below] {$\scriptstyle{2}$} --
		(A) -- ++(-.6,1)
		(A) -- ++(-.3,1)
		(A) -- ++(.3,1)
		(A)++(-.1,1);
		\draw (3.75,0) node[below] {$\scriptstyle{\cdots}$};
		\draw (3.75,1.4) node[above] {$\scriptstyle{\cdots}$};
		\draw (-0.22,2.05) node[above] {$\scriptstyle{\cdots}$};
		\draw (-.6,2.03) node[above] {${\scriptstyle{1}}$};
		\draw (0.86,2) node[above] {${\scriptstyle{\lambda_1+1}}$};
		\draw (1.45,2.05) node[above] {$\scriptstyle{\cdots}$};
		\draw (2.1,2) node[above] {${\scriptstyle{\lambda_1+\lambda_2}}$};
		\draw (5.35,2.05) node[above] {$\scriptstyle{\cdots}$};
		\coordinate (A) at (6,1);
		\draw[thin]
		(A) ++(0,-1) node[below] {$\scriptstyle{m}$} --
		(A) -- ++(-.6,1)
		(A) -- ++(-.3,1)
		(A) -- ++(.3,1)
		(A)++(-.2,1) node[above right] {${\scriptstyle{\lambda_1+\cdots+\lambda_m}}$};
		\draw[fill=white] (-0.3,0.5) rectangle (6.3,1);
	\end{tikzpicture}~.
\]

\begin{proposition}\label{prop:InfDecDPoisAC}
	For any partition ordered partition
	$\lambda_1+\cdots+\lambda_m=n$ of $n$ into $m$ positive integers, the image of the basis element
	$\nu_{\lambda_1, \ldots, \lambda_m}$
	under the infinitesimal decomposition map is equal to
	\begin{multline*} \label{eq:decompo_DPois_DPois}
		\Delta_{(1,1)}\left(\nu_{\lambda_1, \ldots, \lambda_m}\right)=
		\sum_{k=1}^{m}
		\sum_{\sigma \in \C_m}
		\sum_{0\leqslant p <  \lambda_{\sigma(m)}\atop 0\leqslant q <  \lambda_{\sigma(k)}}
		(-1)^{\theta} \
		\begin{tikzpicture}[scale=0.5,baseline=(n.base)]
			\node (n) at (0,2) {};
			%partie basse
			\coordinate (A) at (0,1);
			\draw[thin]
			(A) ++(0,-1) node[below] {$\scriptstyle{\sigma(1)}$} --
			(A) -- ++(-.6,1.5)
			(A) -- ++(-.3,1.5)
			(A) -- ++(.3,1.5)
			(A)++(-.1,1.5) node[above] {$\scriptstyle{\overline{\lambda}_{\sigma(1)}}$};
			\coordinate (A) at (1.5,1);
			\draw[thin]
			(A) ++(0,-1) node[below] {$\scriptstyle{\sigma(2)}$} --
			(A) -- ++(-.6,1.5)
			(A) -- ++(-.3,1.5)
			(A) -- ++(.3,1.5)
			(A)++(-.1,1.5) node[above] {$\scriptstyle{\overline{\lambda}_{\sigma(2)}}$};
			\draw (3.5,0) node[below] {$\scriptstyle{\cdots}$}
			++ (0,1.5) node[above] {$\scriptstyle{\cdots}$};
			\coordinate (A) at (6,1);
			\draw[thin]
			(A) ++(0,-1) node[below] {$\scriptstyle{\sigma(k)}$} --
			(A) -- ++(-1.4,1.5)
			(A) -- ++(-1,1.5) node[above] {$\scriptstyle{\overline{p}}$}
			(A) -- ++(-.6,1.5)
			(A) -- ++(.6,1.5)
			(A) -- ++(1,1.5) node[above] {$\scriptstyle{\overline{q}}$}
			(A) -- ++(1.4,1.5) ;
			\coordinate (shift) at (0,3);
			\draw (A) -- ++(shift);
			\draw[fill=white] (-0.3,0.5) rectangle (6.3,1);
			%partie haute
			\coordinate (A) at (6,4.5);
			\draw[thin]
			(A) -- ++(-1.2,2)
			(A) -- ++(-.8,2)
			(A) ++(.5,2) node[above] {$\scriptstyle{\overline{q}}$};
			\draw[dotted]
			(A) -- ++(.1,2)
			(A) -- ++(.5,2)
			(A) -- ++(1,2) ;
			\draw[decoration={brace}, decorate]
			(A) ++(-1.3,3) -- ++ (2.4,0)
			(A) ++(0,3)
			node[above] {$\scriptstyle{\overline{\lambda}_{\sigma(k)}}$};
			\coordinate (A) at (9,4.5);
			\draw[thin]
			(A) ++(0,-2) node[below] {$\scriptstyle{\sigma(k+1)}$} --
			(A) -- ++(-.6,1.5)
			(A) -- ++(-.3,1.5)
			(A) -- ++(.3,1.5)
			(A)++(-.1,1.5) node[above] {$\scriptstyle{\overline{\lambda}_{\sigma(k+1)}}$};
			\draw (11,3.5) node[below] {$\scriptstyle{\cdots}$}
			++ (-.3,1.5) node[above] {$\scriptstyle{\cdots}$};
			\coordinate (A) at (13,4.5);
			\draw[thin]
			(A) ++(0,-2) node[below] {$\scriptstyle{\sigma(m)}$} --
			(A) -- ++(.8,2)
			(A) -- ++(.5,2)
			(A) -- ++(1.2,2) ;
			\draw[dotted]
			(A) -- ++(-1,2)
			(A) -- ++(-.6,2)
			(A) -- ++(-.2,2)
			(A) ++ (-.5,2) node[above] {$\scriptstyle{\overline{p}}$};
			\draw[decoration={brace}, decorate]
			(A) ++(-1.1,3) -- ++ (2.4,0)
			(A) ++(0,3)
			node[above] {$\scriptstyle{\overline{\lambda}_{\sigma(m)}}$};
			\draw[fill=white] (5.7,4) rectangle (13.3,4.5); \ .
		\end{tikzpicture}~,
	\end{multline*}
	where
	$\overline{\lambda_l}\coloneq \{\lambda_1+\cdots+\lambda_{l-1}+1, \allowbreak \ldots, \lambda_1+\cdots+\lambda_{l}\}$~,
	where $\overline{p}$ is the set made up of the first $p$ terms of $\overline{\lambda}_{\sigma(m)}$~,
	where $\overline{q}$ is the set made up of the last $q$ terms of $\overline{\lambda}_{\sigma(k)}$~, and where
	\begin{align*}
		\theta= &
		q\left(\lambda_{\sigma(k)}-1\right)
		+p\left(\lambda_{\sigma(k)}+\cdots+\lambda_{\sigma(m-1)}\right)
		+\lambda_m
		+(n+1)\left(\lambda_1+\cdots+\lambda_{\sigma(1)-2}\right)+n\lambda_{\sigma(1)-1} \\
		        & ~~+\left(\lambda_{\sigma(1)}+\cdots+\lambda_{\sigma(k-1)}+p+q\right)
		\left(\lambda_{\sigma(k)}+\cdots+\lambda_{\sigma(m)}-p-q-1\right)
		~.
	\end{align*}
\end{proposition}

The summand $k=1$ is understood without a white box on the bottom left side, that is just a corolla with $p+1+q$ leaves, with $p+1+q\geqslant 2$~. The summand $k=m$ is understood without a white box on the top right side, that is just a corolla whose leaves are labelled by $\bar\lambda_{\sigma(m)} \backslash (\bar{p}\cup \bar{q})$, which is required to have at least two elements.

\begin{proof}
	\cref{lemm:KDproperad} shows that
	the Koszul dual properad $\DPois^!$ admits the same generators \eqref{eq:GenDLie} as the Koszul dual properad $\DLie^!$ plus the ``planar'' binary product
	
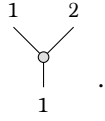
\[
		\begin{tikzpicture}[baseline=0ex,scale=0.2]
			\draw (0,4) node[above] {$\scriptstyle{1}$}
			-- (2,2);
			\draw (4,4) node[above] {$\scriptstyle{2}$}
			-- (2,2) -- (2,0) node[below] {$\scriptstyle{1}$};
			\draw[fill=gray!25] (2,2) circle (10pt);
		\end{tikzpicture}~.
	\]
	Its relations are that of $\DLie^!$, depicted above in \eqref{eq:RelDLie}, plus the anti-associativity relation
	
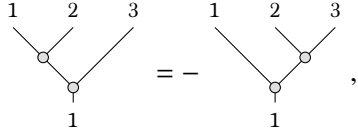
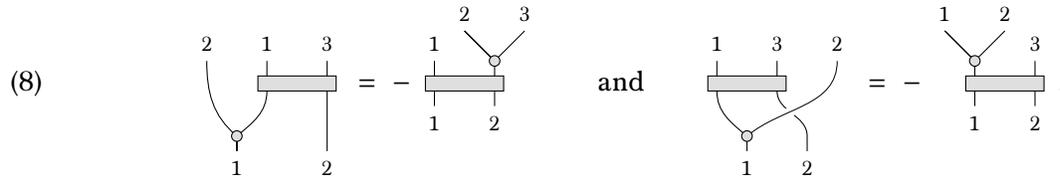
\[
		\begin{tikzpicture}[baseline=0.5ex,scale=0.2]
			\draw (0,4) node[above] {$\scriptstyle{1}$} --(4,0)
			-- (4,-1) node[below] {$\scriptstyle{1}$};
			\draw (4,4) node[above] {$\scriptstyle{2}$} -- (2,2);
			\draw (8,4) node[above] {$\scriptstyle{3}$}-- (4,0);
			\draw[fill=gray!25] (2,2) circle (10pt);
			\draw[fill=gray!25] (4,0) circle (10pt);
		\end{tikzpicture}
		=-
		\begin{tikzpicture}[baseline=0.5ex,scale=0.2]
			\draw (0,4) node[above] {$\scriptstyle{1}$}
			-- (4,0) -- (4,-1) node[below] {$\scriptstyle{1}$};
			\draw (4,4) node[above] {$\scriptstyle{2}$} -- (6,2);
			\draw (8,4) node[above] {$\scriptstyle{3}$} -- (4,0);
			\draw[fill=gray!25] (6,2) circle (10pt);
			\draw[fill=gray!25] (4,0) circle (10pt);
		\end{tikzpicture},
	\]
	the rewriting rules
	\begin{equation}\label{eqn:ReWrDPois!}
		\begin{tikzpicture}[scale=0.4,baseline=(n.base))]
			\node (n) at (0,0) {};
			\draw (-2,1) node[above] {$\scriptstyle{2}$}
			to[out=270, in=140] (-1,-1.5)
			-- (-1,-2) node[below] {$\scriptstyle{1}$}
			-- (-1,-1.5) to[out=40, in=270] (0,0)
			-- (0,1) node[above] {$\scriptstyle{1}$};
			\draw (2,-2) node[below] {$\scriptstyle{2}$}
			-- (2,1) node[above] {$\scriptstyle{3}$};
			\draw[fill=gray!25] (-0.3,0) rectangle (2.3,0.5)
			(-1,-1.5) circle (5pt);
		\end{tikzpicture}
		\ =\ - \
		\begin{tikzpicture}[scale=0.4,baseline=(n.base))]
			\node (n) at (0,0) {};
			\draw (0,-.5) node[below] {$\scriptstyle{1}$}
			-- (0,1) node[above] {$\scriptstyle{1}$};
			\draw (2,-.5) node[below] {$\scriptstyle{2}$}
			-- (2,1) -- (1,2) node[above] {$\scriptstyle{2}$}
			-- (2,1) -- (3,2)  node[above] {$\scriptstyle{3}$};
			\draw[fill=gray!25] (-0.3,0) rectangle (2.3,0.5) (2,1) circle (5pt);
		\end{tikzpicture}
		\qquad \text{and} \qquad
		\begin{tikzpicture}[scale=0.4,baseline=(n.base))]
			\node (n) at (0,0) {};
			\draw  (1,-2) node[below] {$\scriptstyle{1}$} -- (1,-1.5) to[out=140, in=270] (0,0) -- (0,1) node[above] {$\scriptstyle{1}$};
			\draw  (3,-2) node[below] {$\scriptstyle{2}$}
			-- (3,-1.5) to[out=90, in=270] (2,0) --
			(2,1) node[above] {$\scriptstyle{3}$};
			\draw[white, double=black , double distance =0.3pt, ultra thick]
			(1,-1.5) to[out=40, in=270] (4,1) ;
			\draw (4,1) node[above] {$\scriptstyle{2}$};
			\draw[fill=gray!25] (-0.3,0) rectangle (2.3,0.5)
			(1,-1.5) circle (5pt);
		\end{tikzpicture}
		\ =\ - \
		\begin{tikzpicture}[scale=0.4,baseline=(n.base))]
			\node (n) at (0,0) {};
			\draw (0,-.5) node[below] {$\scriptstyle{1}$}
			-- (0,1) -- (-1,2)  node[above] {$\scriptstyle{1}$}
			--(0,1) -- (1,2) node[above] {$\scriptstyle{2}$};
			\draw (2,-.5) node[below] {$\scriptstyle{2}$}
			-- (2,1) node[above] {$\scriptstyle{3}$} ;
			\draw[fill=gray!25] (-0.3,0) rectangle (2.3,0.5)
			(0,1) circle (5pt);
		\end{tikzpicture}\ ,
	\end{equation}
	and the genus vanishing relation
	\[
		\begin{tikzpicture}[scale=0.4,baseline=(n.base))]
			\node (n) at (0,0) {};
			\draw (-2,1) node[above] {$\scriptstyle{1}$}
			to[out=270, in=140] (-1,-1.5)
			-- (-1,-2) node[below] {$\scriptstyle{1}$}
			-- (-1,-1.5) to[out=40, in=270] (0,0)
			-- (0,1) node[above] {$\scriptstyle{2}$};
			\draw[fill=gray!25] (-2.3,0) rectangle (0.3,0.5)
			(-1,-1.5) circle (5pt);
		\end{tikzpicture}\ .
	\]

	It was already shown in \cite[Lemma~5.4]{Ler19ii} that the composite of any elements in the properad $\DLie^!$ along a graph of positive genus vanishes. Using the aforementioned relations, the same property holds true in the Koszul dual properad $\DPois^!$. Indeed, given any composite along a graph of positive genus, one can pull up the binary  product, unless one encounters a box connected to the two inputs of such a product, in which case it vanishes. Once all the binary products stand above, one finds below a composite of boxes along a graph of positive genus, which  vanishes.

	\medskip

	So the infinitesimal decomposition map of the Koszul dual coproperad $\DPois^{\ac}$ splits any element into two along one edge.
	In order to understand which ones, we work in the Koszul dual properad  $\DPois^!$~.
	Using the rewriting rule~\eqref{eqn:ReWrDPois!}, one can see that only the partial composite
	of two basis elements given in \cref{lemm:Dpois!} along the 2-vertices graphs given on the right-hand side of the above formula can produce a given basis element.
	Here, we have been using the cyclic symmetric to place the composed output at the first place and the composed input at the last place.

	\medskip

	We conclude with the computation of the sign by systematically applying the Koszul sign rule and convention as follows.
	The partial composite of these two basis elements in the Koszul dual properad  $\DPois^!$ is given by first permuting the right-most $q$ binary operations of the bottom element above the first binary box of the top element using $q$ iterations of the right-most rewriting rule \eqref{eqn:ReWrDPois!}: this produces the sign $(-1)^q$~.
	At the first input of the left bottom binary box of the top element, one has now the composite of a right comb with $\lambda_{\sigma(k)}-q-1$ vertices with a right comb made up of  $q$ vertices.
	Rewriting it a total right comb produces the sign $(-1)^{q(\lambda_{\sigma(k)}-q-1)}$.
	All together, they form the first term of $\theta$.

	Then, one permutes the remaining right-most $p$ binary operations of the bottom element above
	all the binary boxes of the top element, using $p$ iterations of the left-most rewriting rule \eqref{eqn:ReWrDPois!}, and one also permutes them with all the right-combs of the top element except for the last one: this produces the second term of $\theta$.
	This way, one gets a basis type element of $\DPois^!$, that is a stairway with right-combs on each step, that one has to rotate cyclically in order to have one with first input and output label by $1$: this produces the next three terms of $\theta$.

	Finally, one has to dualise linearly this partial composition product in the Koszul dual properad $\DPois^!$ in order to obtain the infinitesimal decomposition map in the Koszul dual coproperad $\DPois^{\ac}$~. This produces an extra sign given by the product of the degrees of the two elements which is the last term of $\theta$.
\end{proof}

 \cref{prop:InfDecDPoisAC} shows that the Koszul dual coproperad $\DPois^{\ac}$ is actually a codioperad \cite{Gan03}, that is its  (infinitesimal) decomposition map produces only genus $0$ graphs. This is salient feature of the properad $\DPois$: even if a category  of gebras can be encoded by a dioperad, which is the case for double Poisson gebras, nothing ensures \emph{a priori} that its Koszul dual codioperad is big enough to contain all the higher homotopies (operadic syzygies) to resolve the associated properad, see \cref{subsec:TypesofAlg} and \cite[Section~5.6]{MV09I} for  counterexamples. 
This property is actually equivalent to the vanishing of the homology groups of some involved graph complexes concentrated in positive genera \cite{Ler19ii}. 

\subsection{Homotopy curved double Poisson gebras}\label{subsec:CurvedHDPois}
Now we have  everything at hand to make explicit the notion of a homotopy double Poisson gebra 
and to generalise it to the curved level.

\begin{theorem}\label{thm:DoublePoisson}
	A homotopy double Poisson gebra is a  dg vector space $A$ equipped with a collection of operations
	\[
		\calm_{\lambda_1, \ldots, \lambda_m} \colon A^{\otimes n} \too A^{\otimes m}
	\]
	of degree $n-2$, for any ordered partition $\lambda_1+\cdots+\lambda_m=n$ of $n\geqslant 1$ into positive integers, without the trivial partition of $1$~, satisfying the following relations.
	\begin{description}
		\item[\sc cyclic skew symmetry]
		      \begin{equation}\label{eq:CyclicSkewSymOP}
			      \calm_{\lambda_2, \ldots, \lambda_m, \lambda_1}=
			      (-1)^{n\lambda_1+\lambda_m}
			      \ \tau_m^{-1}\cdot\calm_{\lambda_1, \lambda_2, \ldots, \lambda_m}\cdot \tau_{\lambda_1, \ldots, \lambda_m}
			      ~,
		      \end{equation}
		      where $\tau_{\lambda_1, \ldots, \lambda_m} \in \Sy_n$ permutes cyclically the blocks of  size $\lambda_1, \ldots, \lambda_m$~.

		\item[\sc homotopy double Poisson relations]
		      \begin{multline}\label{eq:RelHoDoublePoiss}
			      \partial \left(\calm_{\lambda_1, \ldots, \lambda_m}\right)=\\
			      \sum_{k=1}^{m}
			      \sum_{\sigma \in \C_m}
			      \sum_{0\leqslant p <  \lambda_{\sigma(m)}\atop 0\leqslant q <  \lambda_{\sigma(k)}}
			      (-1)^{\xi}\,
			      \sigma^{-1}\cdot\left(\calm_{\lambda_{\sigma(1)}, \ldots, \lambda_{\sigma(k-1)}, p+1+q}
			      \circ^1_{i}
			      \calm_{\lambda_{\sigma(k)-q}, \lambda_{\sigma(k+1)}, \ldots, \lambda_{\sigma(m-1)},
				      \lambda_{\sigma(m)-p}}\right)\cdot \omega~,
		      \end{multline}
		      where
		      $i\coloneq \lambda_{\sigma(1)}+\cdots +\lambda_{\sigma(k-1)}+ p+1$, where
		      \begin{align*}
			      \xi= &
			      q\left(\lambda_{\sigma(k)}-1\right)
			      +p\left(\lambda_{\sigma(k)}+\cdots+\lambda_{\sigma(m-1)}\right)
			      +\lambda_m
			      +(n+1)\left(\lambda_1+\cdots+\lambda_{\sigma(1)-2}\right)+n\lambda_{\sigma(1)-1}+ \\
			           & \left(\lambda_{\sigma(1)}+\cdots+\lambda_{\sigma(k-1)}+p+q\right)
			      \left(\lambda_{\sigma(k)}+\cdots+\lambda_{\sigma(m)}-p-q\right)+1
			      ~,
		      \end{align*}
		      and
		      where $\omega$ is the permutation sending $(1, \ldots, n)$ to
		      \begin{align*}
			      \Big(
			       & \cev{\lambda}_{\sigma(1)}+1, \ldots, \cev{\lambda}_{\sigma(k)},
			      \cev{\lambda}_{\sigma(m)}+1, \cdots, \cev{\lambda}_{\sigma(m)}+p,
			      \cev{\lambda}_{\sigma(k)}+1, \cdots, \cev{\lambda}_{\sigma(k+1)}-q, \\
			       &
			      \cev{\lambda}_{\sigma(k+1)}+1, \ldots, \cev{\lambda}_{\sigma(m)},
			      \cev{\lambda}_{\sigma(m)}+p+1, \ldots, \cev{\lambda}_{\sigma(m+1)},
			      \cev{\lambda}_{\sigma(k+1)}-q+1, \cdots, \cev{\lambda}_{\sigma(k+1)}
			      \Big)~.
		      \end{align*}
		      under the convention $\cev{\lambda}_l \coloneq \lambda_1+\cdots+\lambda_{l-1}$~, for
		      $1<l \leqslant m$~, and $\cev{\lambda}_1\coloneq 0$~.
	\end{description}
\end{theorem}

Graphically, the partial composite appearing on the right-hand side of the homotopy double Poisson relations is represented in \cref{PartCompoHoDP}.

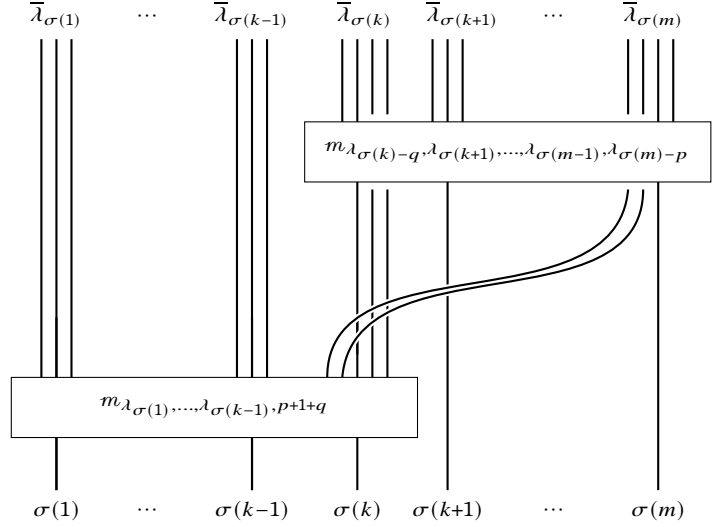
\begin{figure*}[h!]
	\begin{tikzpicture}[scale=1,baseline =(n.base)]
		\node (n) at (1,2) {};
		%%% trucs
		\draw[thick]
		(2.2,6.5) -- (2.2,5.5)
		(2.4,6.5) -- (2.4,5.5)
		(2.2,4.5) to[out=270,in=90] (2.2,2)
		(2.4,4.5)  to[out=270,in=90] (2.4,2);
		%% Sorties de G
		\draw[thick]
		(2,5) --  (2,2.3) %node[near end, right] {\small $k$}
		(3.2,5) --  (3.2,0.5)	node[below] {$\scriptstyle{\sigma(k+1)}$}
		(6,5) -- (6,0.5) node[below] {$\scriptstyle{\sigma(m)}$};
		\draw[thick] %entrees
		%% lambda_1
		(-2.2,6.5)  to (-2.2,2)
		(-2,6.5) node[above] {$\scriptstyle{\overline{\lambda}_{\sigma(1)}}$} to[out=270,in=90] (-2,0.5)
		(-1.8,6.5) to[out=270,in=90] (-1.8,2)
		%% lambda_k-1
		(0.4,6.5)  to[out=270,in=90] (0.4,2)
		(0.6,6.5) node[above] {$\scriptstyle{\overline{\lambda}_{\sigma(k-1)}}$} to[out=270,in=90] (0.6,0.5)
		(0.8,6.5) to[out=270,in=90] (0.8,2)
		(1.8,6.5) to[out=270,in=90] (1.8,5)
		(2,6.5)  to  (2,5)
		(2.1, 6.5) node[above] {$\scriptstyle{\overline{\lambda}_{\sigma(k)}}$}
		(3,6.5) to[out=270,in=90] (3,5)
		(3.2,6.5) to[out=270,in=90] (3.2,5)
		(3.4,6.5) node[above] {$\scriptstyle{\overline{\lambda}_{\sigma(k+1)}}$} to[out=270,in=90] (3.4,5)
		(6,6.5) to[out=270,in=90] (6,5)
		(6.2,6.5) to[out=270,in=90] (6.2,5)
		(5.4,6.5) node[above right] {$\scriptstyle{\overline{\lambda}_{\sigma(m)}}$}
		(5.6,6.5) -- (5.6,5.5)
		(5.8,6.5) -- (5.8, 5.5)
		% cdots
		(-0.8,6.6) node[above] {$\scriptstyle{\cdots}$}
		(-0.8,0.4) node[below] {$\scriptstyle{\cdots}$}
		(4.6,6.6) node[above] {$\scriptstyle{\cdots}$}
		(4.6,0.4) node[below] {$\scriptstyle{\cdots}$};
		%%% SORTIES
		\draw[thick]
		(-2,2.8) -- (-2,0.5) node[below] {$\scriptstyle{\sigma(1)}$}
		(0.6,2.8) -- (0.6,0.5) node[below] {$\scriptstyle{\sigma(k-1)}$}
		(2,2.8) -- (2,0.5) node[below] {$\scriptstyle{\sigma(k)}$};
		%%%% supers croisements
		\draw[draw=white,double=black,double distance=2*\pgflinewidth,thick] %% A MODIF
		% (5.6,6.5) -- (5.6,4.8) to[out=270,in=90] (2.1,1.5)
		% (5.8,6.5) -- (5.8,4.8) to[out=270,in=90] (2.3,1.5);
		(5.6,4.5) to[out=265,in=90] (1.6,2)
		(5.8,4.5) to[out=270,in=85] (1.8,2);
		%%% RECTANGLES
		\draw[fill=white]
		(-2.6,1.2) rectangle (2.8,2)
		node[midway]
			{$\scriptstyle{\calm_{\lambda_{\sigma(1)}, \ldots, \lambda_{\sigma(k-1)}, p+1+q}}$}
		(1.3,4.6) rectangle (6.7,5.4)
		node[midway]
			{$\scriptstyle{\calm_{\lambda_{\sigma(k)-q}, \lambda_{\sigma(k+1)}, \ldots, \lambda_{\sigma(m-1)}, 	\lambda_{\sigma(m)-p}}}$};
	\end{tikzpicture}
	\caption{Partial composite involved in the homotopy double Poisson relations.}
	\label{PartCompoHoDP}
\end{figure*}

\begin{proof}
	Recall from~\cref{prop:twisting} that a structure of a homotopy double Poisson gebra on $A$ corresponds to a twisting morphism $\alpha \colon \DPois^\antish \to \End_A$~,
	that is a  Maurer--Cartan element of the properadic convolution Lie-admissible algebra
	\begin{align*}
		\widehat{\Hom}\left(\DPois^\antish , \End_A\right)
		\cong
		\left(
		\prod_{n,m\geqslant 1}
		\Hom_{\Sy_m^\op\times \Sy_n}\left(\overline{\DPois}^\antish(m,n) , \Hom\left(A^{\otimes n}, A^{\otimes m}\right)\right), \partial,  \star\right)  \ .
	\end{align*}
	Let us denote by
	\[\calm_{\lambda_1, \ldots, \lambda_m}\coloneq \alpha\left(\nu_{\lambda_1, \ldots, \lambda_m}\right)\]
	the image of the  basis elements.
	Since these later ones have degree $n-1$ and since a twisting morphism $\alpha$ has degree $-1$, the operations
	$\calm_{\lambda_1, \ldots, \lambda_m}$ have degree $n-2$~.
	The cyclic skew symmetry \eqref{eq:SignCyclicDual} of the basis elements of $ \DPois^\antish$
	induces the cyclic skew symmetry
	\eqref{eq:CyclicSkewSymOP}
	of the structural operations $\calm_{\lambda_1, \ldots, \lambda_m}$~.
	The Maurer--Cartan equation \eqref{eq:MCProperad} satisfied by $\alpha$~, once evaluated on the  basis elements
	$\nu_{\lambda_1, \ldots, \lambda_m}$~, produces the homotopy double Poisson relations, after the form of the infinitesimal decomposition map of the Koszul dual coproperad
	$\DPois^\ac$
	given in  \cref{prop:InfDecDPoisAC}. The only extra sign created here comes from the Koszul convention: one has to permute  $\alpha$ with the bottom basis element in order to apply it to the top basis element.
\end{proof}

\begin{remark}\label{rem::Schedler}\leavevmode
	\begin{enumerate}

		\item Since the Koszul dual cooperad $\As^\antish$ of the operad encoding associative algebras, embeds into the coproperad
		      $\DPois^\ac$, any homotopy double Poisson gebra contains an $\rmA_\infty$-algebra structure made up of the operations $\calm_{n}$~, for $n\geqslant 2$~.

		\item In \cite[Definition 4.1]{Sch06}, T. Schedler defined the notion of a \emph{double Poisson-infinity (al)gebra}, which is a dg associative algebra made up of a collection of brackets
		      \[
			      \{-,\ldots,- \}_n \colon A^{\otimes n} \too A^{\otimes n}
		      \]
		      of degree $n-2$, for all $n\in \NN^*$, satisfying some identities. Such a structure is a special case of homotopy double Poisson gebra, where the operations $\calm_{\lambda_1, \ldots, \lambda_m}=0$ vanish, for every partition  with at least one $\lambda_i>1$, and under the identification
		      \[
			      \calm_{1, \ldots, 1} \coloneq \{-,\ldots,- \}_n
		      \]
		      adding extra symmetries. 
		      
		\item  In \cite{Yeu18,Pri20}, W. Yeung and J.P. Pridham define the notion of an \emph{$n$-shifted double Poisson algebra}.
		      We will only consider here  the case  $n=-1$ and we will refer to this structure as  \emph{semi homotopy double Poisson gebra}. Such a notion is modelled by the properad $\mathrm{shDPois}$ given by
		      \[
			      \dfrac{\scrG\left(
				      %ASS
				      \begin{tikzpicture}[baseline=(n.base),scale=0.10]
					      \node (n) at (0,1) {};
					      \draw (0,4) node[above] {$\scriptscriptstyle{1}$}
					      -- (2,2);
					      \draw (4,4) node[above] {$\scriptscriptstyle{2}$}
					      -- (2,2) -- (2,0) node[below] {$\scriptscriptstyle{1}$};
					      \draw[fill=white] (2,2) circle (12pt);
				      \end{tikzpicture}, \
				      %DLIEinfty
				      \begin{tikzpicture}[scale=0.3,baseline=(n.base)]
					      \node (n) at (0,0.5) {};
					      \draw[thin] (0,0) node[below] {$\scriptscriptstyle{1}$}
					      -- ++(0,1.5) node[above] {$\scriptscriptstyle{1}$};
					      \draw[thin] (1.5,0) node[below] {$\scriptscriptstyle{2}$}
					      -- (1.5,1.5) node[above] {$\scriptscriptstyle{2}$};
					      \draw (3.75,0) node[below] {$\scriptscriptstyle{\cdots}$}
					      ++(0,1.5) node[above] {$\scriptscriptstyle{\cdots}$};
					      \draw[thin]	(6,0) node[below] {$\scriptscriptstyle{n}$}
					      -- ++(0,1.5) node[above] {$\scriptscriptstyle{n}$};
					      \draw[fill=white] (-0.3,0.5) rectangle (6.3,1);
				      \end{tikzpicture}
				      , \  n\geqslant 2\right)}
			      {\left(
				      %ASS
				      \begin{tikzpicture}[baseline=0.5ex,scale=0.1]
					      \draw (0,4) node[above] {$\scriptscriptstyle{1}$} --(4,0)
					      -- (4,-1) node[below] {$\scriptscriptstyle{1}$};
					      \draw (4,4) node[above] {$\scriptscriptstyle{2}$} -- (2,2);
					      \draw (8,4) node[above] {$\scriptscriptstyle{3}$}-- (4,0);
					      \draw[fill=white] (2,2) circle (10pt);
					      \draw[fill=white] (4,0) circle (10pt);
				      \end{tikzpicture}
				      -
				      \begin{tikzpicture}[baseline=0.5ex,scale=0.1]
					      \draw (0,4) node[above] {$\scriptscriptstyle{1}$}
					      -- (4,0) -- (4,-1) node[below] {$\scriptscriptstyle{1}$};
					      \draw (4,4) node[above] {$\scriptscriptstyle{2}$} -- (6,2);
					      \draw (8,4) node[above] {$\scriptscriptstyle{3}$} -- (4,0);
					      \draw[fill=white] (6,2) circle (10pt);
					      \draw[fill=white] (4,0) circle (10pt);
				      \end{tikzpicture},
				      %DERV
				      \begin{tikzpicture}[scale=0.3,baseline=(n.base)]
					      \node (n) at (0,0.5) {};
					      \draw[thin]
					      (0,0) node[below] {$\scriptscriptstyle{1}$}
					      -- (0,1.5) node[above] {$\scriptscriptstyle{1}$}
					      (1.5,0) node[below] {$\scriptscriptstyle{2}$}
					      -- (1.5,1.5) node[above] {$\scriptscriptstyle{2}$}
					      (6,0) node[below] {$\scriptscriptstyle{n}$}
					      -- (6,2) -- (5,3) node[above] {$\scriptscriptstyle{n}$}
					      (6,2) -- (7,3) node[above] {$\scriptscriptstyle{n+1}$};
					      \draw[fill=white] (6,2) circle (4pt);
					      \draw[fill=white] (-0.3,0.5) rectangle (6.3,1);
				      \end{tikzpicture}
				      -\
				      \begin{tikzpicture}[scale=0.3,baseline=(n.base)]
					      \node (n) at (0,.5) {};
					      \draw[thin]
					      (0,0) -- (0,1.5) node[above] {$\scriptscriptstyle{1}$}
					      (1.5,0) node[below] {$\scriptscriptstyle{2}$}
					      -- (1.5,1.5) node[above] {$\scriptscriptstyle{2}$}
					      (4.5,0) node[below] {$\scriptscriptstyle{n-1}$}
					      -- (4.5,1.5) node[above] {$\scriptscriptstyle{n-1}$}
					      (6,0) node[below] {$\scriptscriptstyle{n}$}
					      -- (6,1.5)  node[above]{$\scriptscriptstyle{n+1}$};
					      \draw[thin] (0,0) -- (-1,-1) --
					      (-2,0) node[above]{$\scriptscriptstyle{n}$}
					      (-1,-1) -- (-1,-1.5) node[below] {$\scriptscriptstyle{1}$};
					      \draw[fill=white] (-1,-1) circle (4pt);
					      \draw[fill=white] (-0.3,0.5) rectangle (6.3,1);
				      \end{tikzpicture}
				      -\
				      \begin{tikzpicture}[scale=0.3,baseline=(n.base)]
					      \node (n) at (0,.5) {};
					      \draw[thin]
					      (0,0) node[below] {$\scriptscriptstyle{1}$}
					      -- (0,1.5) node[above] {$\scriptscriptstyle{1}$}
					      (1.5,0) node[below] {$\scriptscriptstyle{2}$}
					      -- (1.5,1.5) node[above] {$\scriptscriptstyle{2}$}
					      (6,1.5) node[above] {$\scriptscriptstyle{n}$} -- (6,0)
					      -- (7,-1) -- (8,0) node[above] {$\scriptscriptstyle{n+1}$}
					      (7,-1) -- (7,-1.5) node[below] {$\scriptscriptstyle{n}$};
					      \draw[fill=white] (7,-1) circle (4pt);
					      \draw[fill=white] (-0.3,0.5) rectangle (6.3,1);
				      \end{tikzpicture}
				      \right)}~,
		      \]
		      equipped with the differential of $\DLie_\infty$.
		      Sending the generators of $\DPois_\infty$, which are not present here, to zero, defines a canonical surjection of properads
		      \[
			      \DPois_\infty \twoheadrightarrow \mathrm{shDP} \ .
		      \]
		      This  induces a canonical  functor from  semi homotopy double Poisson gebras to homotopy double Poisson gebras.
		      Notice that the definition given in \emph{loc. cit.} of a semi homotopy double Poisson gebra is in terms of a Maurer--Cartan element in a  Lie algebra of noncommutative vector fields. In this direction, we expect the following  noncommutative version of the main result of \cite{Mel16}: for a cofibrant dg associative algebra $A$, the space of $\DPois_\infty$-gebra structures on $A$ equivalent to $A$ as $A_\infty$-algebras is homotopy equivalent to the space of $\mathrm{shDP}$-gebra structures on $A$ with the same associative product. The most natural first step would be to settle a partial rectification of the $A_\infty$-structure of a $\DPois_\infty$-gebra. 
			\end{enumerate}
\end{remark}

The notion a homotopy double Poisson gebra admits a curved generalisation in a way similar to homotopy associative algebras \cite{CT01, FOOO09I} (that it contains) and homotopy Lie algebras \cite{ZWIEBACH199333}; for more details, we refer the reader to  \cite[Chapter~4]{DSV21}.
To settle it, we use a properadic extension of the method presented in \emph{loc. cit.} which relies on the principle that ``curvature is Koszul dual to unit''.

\begin{definition}[The properad $\mathrm{u}\DPois^!$]
	The \emph{unital extension of the properad $\DPois^!$} is defined by

	\[
		\mathrm{u}\DPois^!\coloneq
		\frac{\DPois^!\vee u}{\left(\nu\circ_1 u=-\id; \nu \circ_2 u=\id\right)}
		=
		\dfrac{
			\scrG\left(
			\begin{tikzpicture}[baseline=0ex,scale=0.10]
				\draw (0,4) node[above] {$\scriptscriptstyle{1}$}
				-- (2,2);
				\draw (4,4) node[above] {$\scriptscriptstyle{2}$}
				-- (2,2) -- (2,0) node[below] {$\scriptscriptstyle{1}$};
				\draw[fill=gray!25] (2,2) circle (10pt);
			\end{tikzpicture}
			~ ; ~
			\begin{tikzpicture}[baseline=0ex,scale=0.2]
				\draw (0,0)  node[below] {$\scriptscriptstyle{1}$}
				-- (0,1.5)node[above] {$\scriptscriptstyle{1}$} ;
				\draw (2,0) node[below] {$\scriptscriptstyle{2}$}
				-- (2,1.5) node[above] {$\scriptscriptstyle{2}$};
				\draw[fill=gray!25] (-0.3,0.5) rectangle (2.3,1);
			\end{tikzpicture}
			\ = \ - \
			\begin{tikzpicture}[baseline=0ex,scale=0.2]
				\draw (0,0)  node[below] {$\scriptscriptstyle{2}$}
				-- (0,1.5)node[above] {$\scriptscriptstyle{2}$} ;
				\draw (2,0) node[below] {$\scriptscriptstyle{1}$}
				-- (2,1.5) node[above] {$\scriptscriptstyle{1}$};
				\draw[fill=gray!25] (-0.3,0.5) rectangle (2.3,1);
			\end{tikzpicture}
			~;~
			\begin{tikzpicture}[baseline=0ex,scale=0.10]
				\draw (2,0) node[below] {$\scriptscriptstyle{1}$}
				-- (2,3);
				\draw[fill=gray!25] (2,3) circle (10pt);
			\end{tikzpicture}
			\right)
		}{
			\left(
			\begin{tikzpicture}[baseline=0.5ex,scale=0.1]
				\draw (0,4) node[above] {$\scriptscriptstyle{1}$} --(4,0)
				-- (4,-1) node[below] {$\scriptscriptstyle{1}$};
				\draw (4,4) node[above] {$\scriptscriptstyle{2}$} -- (2,2);
				\draw (8,4) node[above] {$\scriptscriptstyle{3}$}-- (4,0);
				\draw[fill=gray!25] (2,2) circle (10pt);
				\draw[fill=gray!25] (4,0) circle (10pt);
			\end{tikzpicture}
			+
			\begin{tikzpicture}[baseline=0.5ex,scale=0.1]
				\draw (0,4) node[above] {$\scriptscriptstyle{1}$}
				-- (4,0) -- (4,-1) node[below] {$\scriptscriptstyle{1}$};
				\draw (4,4) node[above] {$\scriptscriptstyle{2}$} -- (6,2);
				\draw (8,4) node[above] {$\scriptscriptstyle{3}$} -- (4,0);
				\draw[fill=gray!25] (6,2) circle (10pt);
				\draw[fill=gray!25] (4,0) circle (10pt);
			\end{tikzpicture}
			~~;~~
			\begin{tikzpicture}[baseline=0.75ex,scale=0.10]
				\draw (0,6) -- (0,4) -- (2,2);
				\draw (4,4) node[above] {$\scriptscriptstyle{1}$}
				-- (2,2) -- (2,0) node[below] {$\scriptscriptstyle{1}$};
				\draw[fill=gray!25] (2,2) circle (10pt);
				\draw[fill=gray!25] (0,6) circle (10pt);
			\end{tikzpicture}
			+
			\begin{tikzpicture}[baseline=0.75ex,scale=0.10]
				\draw (0,4) node[above] {$\scriptscriptstyle{1}$}
				-- (0,0) node[below] {$\scriptscriptstyle{1}$};
			\end{tikzpicture}
			~~;~~
			\begin{tikzpicture}[baseline=0.75ex,scale=0.10]
				\draw (4,6) -- (4,4) -- (2,2);
				\draw (0,4) node[above] {$\scriptscriptstyle{1}$}
				-- (2,2) -- (2,0) node[below] {$\scriptscriptstyle{1}$};
				\draw[fill=gray!25] (2,2) circle (10pt);
				\draw[fill=gray!25] (4,6) circle (10pt);
			\end{tikzpicture}
			-
			\begin{tikzpicture}[baseline=0.75ex,scale=0.10]
				\draw (0,4) node[above] {$\scriptscriptstyle{1}$}
				-- (0,0) node[below] {$\scriptscriptstyle{1}$};
			\end{tikzpicture}
			~~;~~
			\begin{tikzpicture}[scale=0.2,baseline=-1ex]
				\node (n) at (0,0) {};
				\draw (-2,1) node[above] {$\scriptscriptstyle{1}$} --(-2,0)
				to[out=270, in=140] (-1,-1.5)
				-- (-1,-2) node[below] {$\scriptscriptstyle{1}$}
				-- (-1,-1.5) to[out=40, in=270] (0,0)
				-- (0,1) node[above] {$\scriptscriptstyle{2}$};
				\draw[fill=gray!25] (-2.3,0) rectangle (0.3,0.5)
				(-1,-1.5) circle (5pt);
			\end{tikzpicture}
			~~;~~
			\begin{tikzpicture}[scale=0.2,baseline=0.1ex]
				\node (n) at (0,.5) {};
				\draw (0,-.5) node[below] {$\scriptscriptstyle{1}$}
				-- (0,2) node[above] {$\scriptscriptstyle{1}$};
				\draw[thin] (2,-.5) node[below] {$\scriptscriptstyle{2}$}
				-- (2,2) node[above] {$\scriptscriptstyle{2}$};
				\draw[fill=gray!25] (-0.3,0) rectangle (2.3,0.5)
				(-0.3,1) rectangle (2.3,1.5);
			\end{tikzpicture}
			\atop
			\begin{tikzpicture}[scale=0.2,baseline=1ex]
				\draw (0,0) node[below] {$\scriptscriptstyle{1}$} -- (0,1)
				to[out=90,in=270] (-1,2.5) node[above] {$\scriptscriptstyle{1}$};
				\draw (2,0) node[below] {$\scriptscriptstyle{2}$}
				-- (2,1.5) -- (1,2.5) node[above] {$\scriptscriptstyle{2}$};
				\draw (2,1.5) -- (3,2.5) node[above] {$\scriptscriptstyle{3}$};
				\draw[fill=gray!25] (2,1.5) circle (6pt);
				\draw[fill=gray!25] (-0.3,0.5) rectangle (2.3,1);
			\end{tikzpicture}
			+
			\begin{tikzpicture}[scale=0.2,baseline=1ex]
				\draw (1,0.5) -- (2,1.5) -- (1,0.5) -- (0,1.5) -- (1,0.5)
				-- (1,-0.5) node[below] {$\scriptscriptstyle{1}$};
				\draw (0,2.5) node[above] {$\scriptscriptstyle{2}$}
				-- (0,1.5);
				\draw (2,2.5) node[above] {$\scriptscriptstyle{1}$}
				-- (2,1.5);
				\draw (4,2) -- (4,2.5) node[above] {$\scriptscriptstyle{3}$};
				\draw (3,-0.5) node[below] {$\scriptscriptstyle{2}$}
				to[out=90,in=270] (4,1.5) ;
				\draw[fill=gray!25] (1,0.5) circle (6pt);
				\draw[fill=gray!25] (1.7,1.5) rectangle (4.3,2);
			\end{tikzpicture}
			~~;~~
			\begin{tikzpicture}[scale=0.2,baseline=(n.base)]
				\node (n) at (0,-0.6) {};
				\draw  (1,-2) node[below] {$\scriptscriptstyle{1}$} -- (1,-1.5) to[out=140, in=270] (0,0) -- (0,1) node[above] {$\scriptscriptstyle{1}$};
				\draw  (3,-2) node[below] {$\scriptscriptstyle{2}$}
				-- (3,-1.5) to[out=90, in=270] (2,0) --
				(2,1) node[above] {$\scriptscriptstyle{3}$};
				\draw[white, double=black , double distance =0.3pt, ultra thick]
				(1,-1.5) to[out=40, in=270] (4,1) ;
				\draw (4,1) node[above] {$\scriptscriptstyle{2}$};
				\draw[fill=gray!25] (-0.3,0) rectangle (2.3,0.5)
				(1,-1.5) circle (5pt);
			\end{tikzpicture}
			+
			\begin{tikzpicture}[scale=0.2,baseline=1ex]
				\draw (2,0) node[below] {$\scriptscriptstyle{2}$} -- (2,1)
				to[out=90,in=270] (3,2.5) node[above] {$\scriptscriptstyle{3}$};
				\draw (0,0) node[below] {$\scriptscriptstyle{1}$}
				-- (0,1.5) -- (-1,2.5) node[above] {$\scriptscriptstyle{1}$};
				\draw (0,1.5) -- (1,2.5) node[above] {$\scriptscriptstyle{2}$};
				\draw[fill=gray!25] (0,1.5) circle (6pt);
				\draw[fill=gray!25] (-0.3,0.5) rectangle (2.3,1);
			\end{tikzpicture}
			~~;~~
			\begin{tikzpicture}[scale=0.2,baseline=-1]
				\draw[thin] (0,-1) node[below] {$\scriptscriptstyle{1}$}
				-- (0,1.5) node[above] {$\scriptscriptstyle{1}$};
				\draw[thin] (2,-1) node[below] {$\scriptscriptstyle{2}$}
				-- (2,1.5) node[above] {$\scriptscriptstyle{2}$};
				\draw[thin] (4,-1) node[below] {$\scriptscriptstyle{3}$}
				-- (4,1.5) node[above] {$\scriptscriptstyle{3}$};
				\draw[fill=gray!25] (1.7,0.5) rectangle (4.3,1);
				\draw[fill=gray!25] (-0.3,-0.5) rectangle (2.3,0);
			\end{tikzpicture}
			-
			\begin{tikzpicture}[scale=0.2,baseline=-1]
				\draw[thin] (0,-1) node[below] {$\scriptscriptstyle{2}$}
				-- (0,1.5) node[above] {$\scriptscriptstyle{2}$};
				\draw[thin] (2,-1) node[below] {$\scriptscriptstyle{3}$}
				-- (2,1.5) node[above] {$\scriptscriptstyle{3}$};
				\draw[thin] (4,-1) node[below] {$\scriptscriptstyle{1}$}
				-- (4,1.5) node[above] {$\scriptscriptstyle{1}$};
				\draw[fill=gray!25] (1.7,0.5) rectangle (4.3,1);
				\draw[fill=gray!25] (-0.3,-0.5) rectangle (2.3,0);
			\end{tikzpicture}
			\right)
		}~,
	\]

	\[
		\begin{aligned}
			\mbox{where }			\nu=\begin{tikzpicture}[baseline=0ex,scale=0.10]
				\draw (0,4) node[above] {$\scriptscriptstyle{1}$}
				-- (2,2);
				\draw (4,4) node[above] {$\scriptscriptstyle{2}$}
				-- (2,2) -- (2,0) node[below] {$\scriptscriptstyle{1}$};
				\draw[fill=gray!25] (2,2) circle (10pt);
			\end{tikzpicture}
			\mbox{ has degree } -1\ ,
			\begin{tikzpicture}[baseline=0ex,scale=0.2]
				\draw (0,0)  node[below] {$\scriptscriptstyle{1}$}
				-- (0,1.5)node[above] {$\scriptscriptstyle{1}$} ;
				\draw (2,0) node[below] {$\scriptscriptstyle{2}$}
				-- (2,1.5) node[above] {$\scriptscriptstyle{2}$};
				\draw[fill=gray!25] (-0.3,0.5) rectangle (2.3,1);
			\end{tikzpicture}
			\mbox{ has degree } -1\ , \text{and where} \
			u=			\begin{tikzpicture}[baseline=0ex,scale=0.10]
				\draw (2,0) node[below] {$\scriptscriptstyle{1}$}
				-- (2,3);
				\draw[fill=gray!25] (2,3) circle (10pt);
			\end{tikzpicture}~~
			\mbox{ has degree } 1\ .
		\end{aligned}
	\]
\end{definition}

\begin{remark}
	A $\mathrm{u}\DPois^!$-algebra structure on  the suspension $\susp A$ of a dg vector space
	amounts to a unital associative algebra structure on $A$ and a degree $-1$ skew-symmetric operation $\varkappa \colon A^{\otimes 2} \to A^{\otimes 2}$ satisfying Relation~\eqref{eq:RelDLie}, the vanishing of the genus $1$ composite of the associative product with $\varkappa$, and
	Relations \eqref{eqn:ReWrDPois!} without signs.
\end{remark}

\begin{lemma}\label{Lem:uDPois!} The
	unital extension of the properad $\DPois^!$ is given by a distributive law, so it is  canonically isomorphic to
	\[
		\mathrm{u}\DPois^!\cong \DLie^! \boxtimes \rmS\mathrm{u}\As~,
	\]
	where $\rmS=\End_{\susp \field}$ is the suspension operad, cf. \cite[Section~7.2.2]{LV12}.
\end{lemma}

\begin{proof}
	This is a direct consequence of \cref{prop:decompo_duale_DPois}.
\end{proof}

This result and \cref{prop:InfDecDPoisAC} show that the unital extension properad $\mathrm{u}\DPois^!$ is a ``dioperad", that is its composition products along graphs of positive genera vanish. 
By \cref{prop:PropDualPartCoprop}, the arity-wise linear dual
$\mathrm{c}\DPois^\antish \coloneq \left(\mathrm{u}\DPois^!\right)^*$
of this properad forms an infinitesimal  coproperad. (It does not form a coproperad neither a comonadic coproperad due to the infinite series that appear in the iterations of the infinitesimal decomposition maps; it carries a counit which fails to be coaugmented.) \cref{Lem:uDPois!} implies that its underlying $\Sy$-bimodule is isomorphic to
\[
	\mathrm{c}\DPois^\antish \cong \DLie^\antish \boxtimes \left(\field u^*\oplus \As^\antish\right)~ .
\]

\begin{definition}[Homotopy curved double Poisson gebra]\label{def:HoCuDPois}
	A \emph{homotopy curved double Poisson gebra} structure of a graded vector space $A$ is a Maurer--Cartan element in the convolution algebra
	\[\mathfrak{c}\convDPois_A \coloneq \widehat{\Hom}\left(\mathrm{c}\DPois^\antish , \End_A\right)~.\]
\end{definition}

Notice that in the definition of this convolution algebra, we use the full infinitesimal coproperad $\mathrm{c}\DPois^\antish$, and not any coaugmentation coideal, since this latter one fails to be coaugmented. This explains why one does not start from a dg vector space but just from a graded vector space. 

\begin{proposition}\label{prop:chDPois}
	A  homotopy curved double Poisson gebra is a  graded vector space $A$ equipped with a collection of operations
	\[
		\calm_{\lambda_1, \ldots, \lambda_m} \colon A^{\otimes n} \too A^{\otimes m}
	\]
	of degree $n-2$, for any ordered partition $\lambda_1+\cdots+\lambda_m=n$ of $n\geqslant 0$ into non-negative integers, satisfying the following relations.
	\begin{description}
		\item[\sc cyclic skew symmetry]
		      \[\calm_{\lambda_2, \ldots, \lambda_m, \lambda_1}=
			      (-1)^{\lambda_1(n-m-1)-n}
			      \ \tau_m^{-1}\cdot\calm_{\lambda_1, \lambda_2, \ldots, \lambda_m}\cdot \tau_{\lambda_1, \ldots, \lambda_m}
			      ~, \]
		      where $\tau_{\lambda_1, \ldots, \lambda_m} \in \Sy_n$ permutes cyclically the blocks of  size $\lambda_1, \ldots, \lambda_m$~.

		\item[\sc homotopy double Poisson relations]
		      \begin{multline*}
			      \sum_{k=1}^{m}
			      \sum_{\sigma \in \C_m}
			      \sum_{0\leqslant p \leqslant  \lambda_{\sigma(m)}\atop 0\leqslant q \leqslant  \lambda_{\sigma(k)}}
			      (-1)^{\xi}
			      \sigma^{-1}\cdot\left(\calm_{\lambda_{\sigma(1)}, \ldots, \lambda_{\sigma(k-1)}, p+1+q}
			      \circ^1_{i}
			      \calm_{\lambda_{\sigma(k)-q}, \lambda_{\sigma(k+1)}, \ldots, \lambda_{\sigma(m-1)},
				      \lambda_{\sigma(m)-p}}\right)\cdot \omega=0~,
		      \end{multline*}
		      with the same notations as in \eqref{eq:RelHoDoublePoiss} for homotopy double Poisson gebras.
	\end{description}
\end{proposition}

In the summand $k=1$, the left-hand operation is  $\calm_{p+1+q}$, with $p+1+q\geqslant 1$~.
In the summand $k=m$, the right-hand operations is
$\calm_{ \lambda_{\sigma(m)-p-q}}$, with $p+q\leqslant \sigma(m)$~.

\begin{remark}
A homotopy curved double Poisson gebra whose structure maps vanish except possibly for 
$\calm_0$, $\calm_{1}$, $\calm_{2}$, and $\calm_{1,1}$ is a 
double Poisson gebra ($\calm_{2}, \calm_{1,1}$)
equipped with a derivation ($\calm_{1}$),  that does not square to zero \emph{a priori}, and a \emph{curvature} 
($\calm_0\in A_{-2}$) satisfying 
\[\calm_1 \circ_1^1 \calm_0=0 \quad ,\quad 
\calm_1^2=\calm_2\circ_1^1 \calm_0-\calm_2\circ_2^1 \calm_0
\quad ,\quad \calm_{1,1}\circ_1^1 \calm_0=0
~.\]
Such a notion deserves the name of a \emph{curved double Poisson gebra}. 
\end{remark}

\cref{prop:chDPois} shows that a homotopy double Poisson gebra is a homotopy curved double Poisson gebra such that the operations $\calm_{\lambda_1, \ldots, \lambda_m}$ are trivial when at least one of the $\lambda_i$  is trivial. (In this case, the operation $\calm_{1}$ squares to zero and is considered as the differential of $A$.)
This result is also a direct consequence of the following general arguments.
\cref{Lem:uDPois!} shows that the canonical map
\[
	\DPois^! \hookrightarrow \mathrm{u}\DPois^!
\]
is an embedding of properads. So the the linear dual morphism of infinitesimal coproperads
\[
	\mathrm{c}\DPois^\antish \twoheadrightarrow \DPois^\antish
\]
is surjective. This induces an embedding of (dg Lie-admissible) dg Lie algebras
\[\convDPois_A= \widehat{\Hom}\left(\DPois^\antish , \End_A\right)
	\hookrightarrow
	\widehat{\Hom}\left(\mathrm{c}\DPois^\antish , \End_A\right)=\mathfrak{c}\convDPois_A
	~.\]
So the deformation theory of the homotopy double Poisson gebras is included in the deformation theory of
 homotopy curved double Poisson gebras.

\begin{remark}
The presentation given in this section tries to go straight to the point  but it is  surely not the most conceptual one. 
In order to encode faithfully the notion of (homotopy) curved double Poisson gebras and to settle their properties, one should follow the method introduced by V. Roca i Lucio in \cite{RL22} who introduced a new notion of a \emph{curved properad} and develop \emph{mutatis mutandis} the associated curved properadic calculus. 
\end{remark}

\section{Curved pre-Calabi--Yau algebras}\label{sec:PCY}

In this section, we introduce the notion of a \emph{curved pre-Calabi--Yau algebra} via the deformation theory of morphisms of cyclic non-symmetric operads. 
We treat this in details in order to  settle carefully all the signs appearing. 
This  leads naturally to the (curved) necklace Lie-admissible algebra, which extends to the even more general higher cyclic complex. A curved pre-Calabi--Yau algebra structure is a Maurer--Cartan element there. The higher cyclic Lie-admissible algebra is stated to be isomorphic to the Lie-admissible algebra encoding curved homotopy double Poisson gebras at the  end of this section. 

\subsection{Cyclic non-symmetric operads}\label{subsec:cyclic_nonsymmetric_operads}
Let $\Cyc$ be the groupoid of finite cyclic sets $\langle x_1, \ldots, \allowbreak x_n \rangle$ of cardinality $n\geqslant 1$,  i.e. sets equipped with an identification with the edges of an oriented $n$-gon, modulo rotations.
The morphisms are the bijections which respect the respective cyclic orders.  We will often represent cyclically ordered sets like gears.

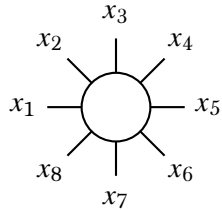
\[
	\begin{tikzpicture}[scale=1.3, baseline = (n.base)]
		\node (n) at (0,0) {};
		\foreach \i in {1,...,8}{
				\draw[thick] (0,0) -- ({\i * 360/8}:20pt);
				\draw[thick] ({(\i + 3) * (-360)/8}:27pt) node {{$x_{\i}$}};
			}
		\draw[fill=white, thick] (0,0) circle [radius=10pt];
	\end{tikzpicture}
\]

\begin{definition}[Cyclic module]
	A \emph{cyclic module} is a module $\M \colon \mathsf{Cyc}^{\mathsf{op}}\to \dgVect$ over the groupoid $\mathsf{Cyc}$~.
	The associated category is  denoted by $\CycMod$.
\end{definition}

The groupoid $\Cyc$ admits for skeletal category the one made up of
one cyclically ordered set  $\langle n\rangle \coloneq \langle 1,  \ldots,  n\rangle$, for any $n\geqslant 1$, equipped with the cyclic groups $\C_n\subset \Sy_n$ for automorphisms. So  the data of a cyclic module $\M$ is equivalent to a collection $\{\M(\la n\ra )\}_{n\geqslant 1}$ of dg vector spaces equipped with actions of the cyclic groups $\C_n$ under the
formula
\[
	\M(X)\coloneq
	\left(\prod_{f\in \Cyc(\langle n\rangle,\, X)} \M(\la n\ra)\right)\big/\sim\ ,
\]
where $|X|=n$
and where
$
	(f, \mu)\sim
	\big(
	g, {g}^{-1}f\cdot\mu
	\big)
$ .
From now on, we will identify both notions, coordinate-free and skeletal, and use the more appropriate description each time.

\bigskip

For any cyclic set $X$, we consider the set $\PT(X)$ of (non-necessarily rooted) planar trees with leaves labeled bijectively and clockwise cyclically by the elements of $X$~.
They induce  the endofunctor
$
	\bbPT \colon \CycMod \to \CycMod
$
defined by
\[
	\bbPT(\M)(X)\coloneq\coprod_{\rmt \in \PT(X)} \rmt(\M)\ ,
\]
where $\rmt(\M)\coloneq\bigotimes_{v\in \text{vert}(\rmt)} \M(\mathrm{in}(v))$~, where $\mathrm{in}(v)$ stands for cyclic set of leaves
of the vertex $v$. The operation of forgetting the nesting of planar trees in $\bbPT\left(\bbPT(\M)\right)$, produces elements of $\bbPT(\M)$ and thus induces a monad structure on $\bbPT$.  We call this monad the \emph{monad of planar trees}. We refer the reader to \cite[Section~1]{DV17}.

\begin{definition}[Cyclic non-symmetric operad \cite{Markl99}]
	A \emph{cyclic non-symmetric (ns) operad} is an algebra over the monad $\bbPT$ of planar trees.
\end{definition}

The monad of planar trees admits a homogenous quadratic presentation, with $2$-vertices planar trees
for generators,
which simplifies the definition of a cyclic non-symmetric operad.
There is a one-to-one correspondence between
$2$-vertices planar trees
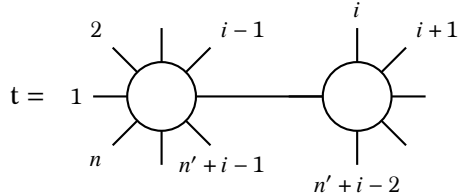
\[\rmt=\vcenter{\hbox{
			\begin{tikzpicture}[scale=1.3]
				\draw[thick] (-1, 0) to (1,0)
				(-1, 0) -- (-1, -0.7)
				(-1, 0) -- (-0.5, -0.5) node[below]  {\scalebox{0.8}{$\quad n'+i-1$}}
				(-1, 0) -- (-1, 0.7)
				(-1, 0) -- (-0.5, 0.5) node[above right]  {\scalebox{0.8}{$i-1$}}
				(-1, 0) -- (-1.7, 0) node[left]  {\scalebox{0.8}{$1$}}
				(-1, 0) -- (-1.5, 0.5) node[above left]  {\scalebox{0.8}{$2$}}
				(-1, 0) -- (-1.5, -0.5) node[below left]  {\scalebox{0.8}{$n$}};
				\draw[thick] (1, 0) -- (1, 0.7) node[above]  {\scalebox{0.8}{$i$}}
				(1, 0) -- (1, -0.7) node[below]  {\scalebox{0.8}{$n'+i-2$}}
				(1,0) -- (1.5, -0.5)
				(1, 0) -- (0.3, 0)
				(1, 0) -- (1.7, 0)
				(1, 0) -- (1.5, 0.5) node[above right]  {\scalebox{0.8}{$i+1$}};
				(1, 0) -- (0.5, 0.5) ;
				(1, 0) -- (1.5, -0.5)  ;
				(1, 0) -- (0.5, -0.5) ;
				\draw[fill=white, thick] (1,0) circle [radius=10pt];
				\draw[fill=white, thick] (-1,0) circle [radius=10pt];
			\end{tikzpicture}}}
\]
and triples $(n, n',i)$, with $n \geqslant 2$, $n'\geqslant 1$,  and $2\leqslant i\leqslant n$~.

So the action of the planar tree monad on a cyclic non-symmetric operad $\P$ corresponding to 2-vertices planar trees is equivalent to \emph{partial composition maps}:

\begin{equation*}
	\circ_i \ \colon \ \P(\la n\ra)\otimes \P(\la n'\ra) \to  \P(\la n+n'-2\ra)\ , \ \text{for} \
	n \geqslant 2~, \ n'\geqslant 1~,
	\ \text{and}\  2\leqslant i\leqslant n~.
\end{equation*}

\begin{proposition}\label{prop:DefEquivCycOp}
	A cyclic non-symmetric operad structure on a cyclic module is equivalent to the data of partial composition maps  satisfying the parallel-sequential relation
	\begin{equation}\label{Rel1}
		\big(\mu \circ_i \nu\big)\circ_j\omega = \left\{
		\begin{array}{ll}
			\rule{0pt}{12pt} (-1)^{|\nu||\omega|} \left(\mu \circ_j \omega\right) \circ_{i+n''-2}\nu\ , & \text{for}\quad  2\leqslant j\leqslant i-1\ ,   \\
			\rule{0pt}{12pt} \mu \circ_i \left(\nu\circ_{j-i+2} \omega\right)\ ,                        & \text{for}\quad i\leqslant j\leqslant i+n'-2\ , \\
			\rule{0pt}{12pt} (-1)^{|\nu||\omega|}\left(\mu \circ_{j-n'+2}\omega\right)\circ_{i}\nu\ ,   & \text{for}\quad
			i+n'-1\leqslant j\leqslant n+n'-2\ ,                                                                                                          \\
		\end{array}\right. \
	\end{equation}
	and the equivariance property
	\begin{equation}\label{eq:equiv}
		\left\{\begin{array}{l}
			\left(\mu \circ_i \nu\right)^{\tau_{n+n'-2}}=\left(\mu^{\tau_n}\right) \circ_{i-1} \nu\ , \quad \text{for}\quad i\geqslant 3~, \\
			\rule{0pt}{12pt} \left(\mu \circ_2 \nu\right)^{\tau_{n+n'-2}}=(-1)^{|\mu||\nu|}\left(\nu^{\tau_{n'}}\right) \circ_{n'}
			\left(\mu^{\tau_n}\right)\ ,
		\end{array}\right.
	\end{equation}
	for any $\mu \in \P(\la n\ra), \nu \in \P(\la n'\ra), \omega \in \P(\la n''\ra)$~.
\end{proposition}

\begin{proof}
	This is a non-symmetric version of \cite[Theorem~2.2]{GetzlerKapranov95}. In one direction, any cyclic non-symmetric operad structure carries partial composition maps satisfying the parallel-sequential axioms and the equivariance property. In the other way round, the data of such partial composition maps allows one to define the action of any planar tree as it can be written as iterating graftings with 2-vertices planar trees.
\end{proof}

\begin{remark}
	Any cyclic non-symmetric operad induces a structure of a non-symmetric operad on $\P_n\coloneq \P(\la n+1\ra)$ by forgetting the cyclic group actions. In the other way round, a cyclic non-symmetric operad structure amounts to a non-symmetric operad structure on a cyclic module satisfying the equivariant property~\eqref{eq:equiv}, see \cite[Page~257]{MSS02}.
\end{remark}

\begin{example}\label{ex:cyc_ns_operads}
	\leavevmode
	\begin{enumerate}

		\item We consider the one-dimensional cyclic module $\AAs(\langle n\rangle)\coloneq \K \mu_n$ with trivial $\mathbb{C}_n$ action and where
		      $|\mu_n|=0$~, for $n\geqslant 3$~, and $\AAs(\langle 2\rangle) = \AAs(\langle  1\rangle) = 0$~.
		      It forms a cyclic non-symmetric operad once equipped with the following partial composition maps
		      \begin{align*}
			      \mu_n\circ_i \mu_{n'}\coloneq \mu_{n+n'-2}\ .
		      \end{align*}

		\item Let $(V, d_V, \langle\ ,\, \rangle)$ be a differential graded vector space equipped with a symmetric bilinear form of degree $0$~.
		      For instance, this can be given by $V\coloneq A\oplus A^*$ equipped with the usual linearity paring $\langle f, x\rangle =f(x)$~.
		      Its \emph{endomorphism cyclic non-symmetric operad}  $\EEnd_V$ is defined on the underlying cyclic module
		      \[
			      \EEnd_V(\la n\ra)\coloneq V^{\otimes n}
		      \]
		      by the partial composition map
		      \begin{multline*}
			      \circ_i(a_1\otimes \cdots \otimes a_n, b_1\otimes \cdots \otimes b_{n'})\coloneq \\
			      (-1)^{
					      \left(|b_1|+\cdots+|b_{n'}|\right)\left(|a_{i+1}|+\cdots+|a_{n}|\right)} \,\langle a_i, b_1\rangle \,
			      a_1 \otimes \cdots \otimes a_{i-1}\otimes b_{2} \otimes \cdots \otimes b_{n'} \otimes
			      a_{i+1}\otimes\cdots \otimes a_n \ .
		      \end{multline*}

		\item The genus $0$ part of any differential graded modular operad carries a canonical cyclic non-symmetric operad structure, where one retains only the actions of the cyclic groups and the partial composition maps. For instance, the homology $\left\{H_\bullet\left(\overline{\mathcal{M}}_{0,\langle n\rangle}\right)\right\}_{n\geqslant 3}$
		      of the genus $0$ part of the the Deligne--Mumford--Knudsen moduli spaces  of stable curves with marked points form a cyclic non-symmetric operad, with trivial summand in arities $1$ and $2$, where the structure maps are given by gluing
		      curves at marked points, see \cite[Section~6]{GetzlerKapranov98}. Similarly, the homology $\left\{H_\bullet\left({\mathcal{M}}^\delta_{0,\langle n\rangle}\right)\right\}_{n\geqslant 3}$ of the Brown's partial compactification \cite{Bro09, DV17} also  forms a cyclic non-symmetric operad.
	\end{enumerate}
\end{example}

\begin{remark}
	In order to avoid any confusion, we use the Roman font for properads (\cref{subsec:Properads}) and the calligraphic font for cyclic non-symmetric operads. For instance, the endomorphism properad of a  dg vector space $A$ is denoted by $\End_A$ and  the endomorphism cyclic non-symmetric operad of a dg vector space $V$ equipped with a scalar product is denoted by $\EEnd_V$~.
\end{remark}

A morphism of cyclic non-symmetric operads is a map of cyclic modules which commutes with the respective partial composition maps. We denote the associated category by
$\mathsf{cyclic} \ \mathsf{ns} \ \mathsf{operads}$.

\begin{definition}[Algebra over a cyclic non-symmetric operad]
	An \emph{algebra structure over a cyclic non-symmetric operad $\P$} on a differential graded vector space $(V, d_V, \langle\ ,\, \rangle)$ equipped with a symmetric bilinear form  is given by the data of morphism of cyclic non-symmetric operads $\P \to \EEnd_V$~.
\end{definition}

\begin{example}
	The category of algebras over the cyclic non-symmetric operad  $\AAs$ is the category of \emph{cyclic associative algebras}, which are
	differential graded associative algebras $(V, d_V, \cdot)$ equipped with a
	symmetric bilinear form $\langle\ ,\, \rangle$ satisfying
	$\langle a\cdot b, c \rangle= \langle a , b\cdot c  \rangle$~, for any $a,b,c\in V$~.
\end{example}

\begin{remark}
	One can also consider the notion of a \emph{unital} cyclic non-symmetric operad defined with the extra data of an arity $2$ element which is a unit for the partial composition maps.
	The cyclic module $\uAAs(\langle n\rangle)\coloneq \K \mu_n$ with trivial $\mathbb{C}_n$ action, for $n\geqslant 1$, equipped with the partial composition maps $\mu_n\circ_i \mu_{n'}\coloneq \mu_{n+n'-2}$ forms such a unital cyclic non-symmetric operad.
	The endomorphism cyclic non-symmetric operad $\EEnd_V$ associated to a vector space  $(V, \langle\ ,\, \rangle)$ equipped with a \emph{non-degenerate} symmetric bilinear form, admits a unital structure given by $\sum_{i=1}^k v_i\otimes v_i^*$~, where $\{v_1, \ldots, v_k\}$ is a basis of $V$ and where $\{v_1^*, \ldots, v_k^*\}$ is the induced dual basis.
	A morphism of unital cyclic non-symmetric operads is required to preserve the respective units. In the present case, an algebra over the unital cyclic non-symmetric operad $\uAAs$ is a \emph{unital} associative algebra
	satisfying $\langle a\cdot b, c \rangle= \langle a , b\cdot c  \rangle$~.
\end{remark}

We will need the following variations of the notion of a cyclic non-symmetric operad.
One can first consider the monad $\bbPT_-$ of \emph{signed planar trees}, which is given by the monad of planar trees equipped with an extra sign coming from the permutation of vertices.

\begin{definition}[Anti-cyclic non-symmetric operad]
	An \emph{anti-cyclic non-symmetric operad} is an algebra over the monad $\bbPT_-$ of signed planar trees.
	Such a structure is equivalent to a cyclic module $\P$ endowed with partial composition maps  satisfying the parallel--sequential relation \eqref{Rel1}
	and the equivariance property up to the following sign
	\begin{equation}\label{eq:equivanti}
		\left(\mu \circ_2 \nu\right)^{\tau_{n+n'-2}}=
		-(-1)^{|\mu||\nu|}\left(\nu^{\tau_{n'}}\right) \circ_{n'}
		\left(\mu^{\tau_n}\right)\ ,
	\end{equation}
	for any $\mu \in \P(\la n\ra), \nu \in \P(\la n'\ra)$, with $n\geqslant 2$ and $n'\geqslant 1$~.
\end{definition}

\begin{examples}\leavevmode
	\begin{enumerate}
		\item
		      The toy model of anti-cyclic non-symmetric operads
		      is the endomorphism operad $\EEnd_V$ of a vector space $V$ equipped with a \emph{skew-symmetric} bilinear form
		      $\la a, b\ra=-(-1)^{|a||b|} \la b, a \ra$~.

		\item A conceptually important anti-cyclic non-symmetric operad is  the \emph{suspension operad} $\susp^2 \EEnd_{\field \asusp}$~.
	\end{enumerate}
\end{examples}

The endofunctor of signed planar trees can also be equipped with a comonad structure, denoted by
$\bbPT^c_- \to \bbPT^c_-(\bbPT^c_-)$, which amounts to sending a planar tree into the sum
of all its partitions  into planar sub-trees.

\begin{definition}[Anti-cyclic non-symmetric cooperad]
	A \emph{anti-cyclic non-symmetric cooperad} is a coalgebra over the comonad $\bbPT^c_-$ of signed planar trees.
\end{definition}

Similarly to the monad case, the comonad $\bbPT^c_-$  is cogenerated by planar trees with two vertices.

\begin{proposition}
	An anti-cyclic non-symmetric cooperad structure on a cyclic module $\CC$ is equivalent to the data of
	\emph{infinitesimal decomposition maps}
	\begin{eqnarray*}
		&\delta_i \ \colon  \
		\CC(\la n+n'-2\ra) \to   \CC(\la n\ra)\otimes \CC(\la n'\ra)   ~, \quad \text{for}\ n\geqslant 2~, n' \geqslant 1~, \  \text{and}\ 2\leqslant i\leqslant n~, \\
		&
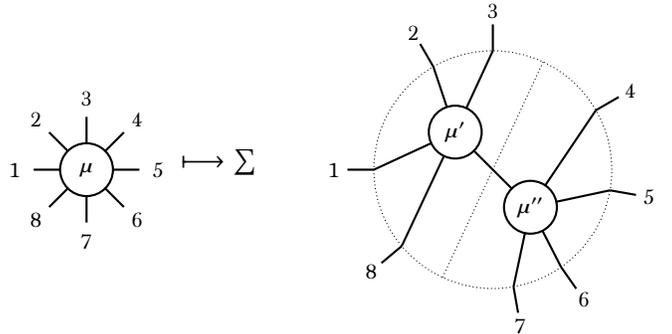
\begin{tikzpicture}[baseline = (n.base)]
			\node (n) at (0,0) {};
			\foreach \i in {1,...,8}{
					\draw[thick] (0,0) -- ({\i * 360/8}:20pt);
					\draw[thick] ({(\i + 3) * (-360)/8}:27pt) node {\scriptsize{${\i}$}};
				}
			\draw[fill=white, thick] (n) circle [radius=10pt] node {\scriptsize{$\mu$}};
		\end{tikzpicture}
		\longmapsto
		\sum
		\qquad
		\begin{tikzpicture}[baseline = (n.base)]
			\node (n) at (0,0) {};
			\node (A) at (0.5,-0.5) {};
			\node (B) at (-0.5,0.5) {};
			\draw[fill=white, densely dotted] (0,0) circle [radius=45pt];
			\coordinate (A1) at (0:45pt);
			\draw[thick] (B) -- (A);
			\newcounter{compte}
			\setcounter{compte}{4}
			\foreach \angle in {30, -10, -55, -80}{
					\draw[thick] (A) -- (\angle:45pt);
					\draw[thick] (\angle:45pt) -- (\angle:55pt) ;
					\draw[thick] (\angle:60pt) node {\scriptsize{${\thecompte}$}};
					\addtocounter{compte}{1}
				}
			\setcounter{compte}{1}
			\foreach \angle in {180, 120, 90}{
					\draw[thick] (B) -- (\angle:45pt);
					\draw[thick] (\angle:45pt) -- (\angle:55pt) ;
					\draw[thick] (\angle:60pt) node {\scriptsize{${\thecompte}$}};
					\addtocounter{compte}{1}
				}
			\setcounter{compte}{8}
			\foreach \angle in {220}{
					\draw[thick] (B) -- (\angle:45pt);
					\draw[thick] (\angle:45pt) -- (\angle:55pt) ;
					\draw[thick] (\angle:60pt) node {\scriptsize{${\thecompte}$}};
				}
			\draw[densely dotted] (65:45pt) -- ({65+180}:45pt);
			\draw[fill=white,thick] (A) circle [radius=10pt] node {\scriptsize{$\mu''$}};
			\draw[fill=white,thick] (B) circle [radius=10pt] node {\scriptsize{$\mu'$}};
		\end{tikzpicture}~,
	\end{eqnarray*}
	satisfying the parallel--sequential relation
	\begin{equation}\label{Rel1CO}
		(\delta_i\otimes \id)\delta_j=
		\left\{
		\begin{array}{ll}
			\rule{0pt}{12pt}
			(23)\cdot \left((\delta_j\otimes \id)\delta_{i+n''-2}\right)\ ,
			 & \text{for}\quad  2\leqslant j\leqslant i-1\ ,   \\
			\rule{0pt}{12pt}
			\, (\id \otimes \delta_{j-i+2})\delta_i\ ,
			 & \text{for}\quad i\leqslant j\leqslant i+n'-2\ , \\
			\rule{0pt}{12pt}
			(23)\cdot \left((\delta_{j-n'+2}\otimes \id)\delta_{i}\right)\ ,
			 & \text{for}\quad
			i+n'-1\leqslant j\leqslant n+n'-2\ ,               \\
		\end{array}\right. \
	\end{equation}
	and the equivariance property
	\begin{equation}\label{eq:equivCO}
		\left\{\begin{array}{l}
			\rule{0pt}{12pt}
			\delta_i \tau_{n+n'-2}^{-1}= \left(\tau_n^{-1}\otimes \id\right)\delta_i\ , \quad \text{for}\quad i\geqslant 3~, \\
			\rule{0pt}{12pt}
			\delta_2 \tau_{n+n'-2}^{-1}=-
			\left(\tau_{n'}^{-1}\otimes \tau_{n}^{-1}\right){}^{(23)}\delta_{n'}
			\ .
		\end{array}\right.
	\end{equation}
	where $\tau_n^{-1}(\mu)\coloneq \mu^{\tau_n}$, for any $\mu\in \CC(\la n\ra))$~.
\end{proposition}

\begin{example}\label{ex:coop_As_anticyclic}
	We consider the arity-wise linear dual of the suspension anti-cyclic non-symmetric operad:
	\[\AAs^{\ac}(\langle n\rangle)\coloneq \left(\susp^2 \EEnd_{\field \asusp}   \right)^\vee(\langle n\rangle)\cong \field \,\susp^{n-2}~, \ \text{for}\ n\geqslant 3~,\]
	$\AAs^{\ac}(\langle 1\rangle)=0$, and $\AAs^{\ac}(\langle 2\rangle)=0$~.
	Since $\EEnd_{\field \asusp}$ forms an arity-wise finite dimensional anti-cyclic non-symmetric operad, its linear dual  is a well defined anti-cyclic non-symmetric cooperad. It is one-dimensional $\AAs^{\ac}(\langle n\rangle)\coloneq \K \upsilon_n$ in each arity $n\geqslant 3$,  with the signature action of $\C_n$ and the degree $|\upsilon_n|=n-2$~. Its infinitesimal decomposition map is equal to
	\begin{align*}
		\delta_i(\upsilon_{n+n'-2})\coloneq (-1)^{in'}\,\upsilon_{n}\otimes \upsilon_{n'}~.
	\end{align*}
\end{example}

\begin{example}\label{ex:coop_uAs_anticyclic}
	In a similar way, we consider the following anti-cyclic non-symmetric cooperad
	\[\cAAs^{\ac}(\langle n\rangle)\coloneq \left(\susp^2 \EEnd_{\field \asusp}   \right)^\vee(\langle n\rangle)\cong \field\, \susp^{n-2}~, \ \text{for}\ n\geqslant 1~.\]
\end{example}

\begin{remark}\leavevmode
	Like the case of modular operads \cite{Ward19, DSVV20}, the category of cyclic non-symmetric operads can be encoded by a groupoid-colored operad which is Koszul and whose Koszul dual (co)operad encodes anti-cyclic non-symmetric (co)operads.
\end{remark}

Another range of generalisations can be obtained by shifting the underlying cyclic module.

\begin{definition}[Shifted cyclic non-symmetric operad]
	A \emph{shifted cyclic non-symmetric operad}  structure on a cyclic module $\P$ is a cyclic non-symmetric operad structure on the desuspended cyclic module $\asusp \P$. Such a data is equivalent to degree $-1$ partial composition maps
	satisfying the parallel--sequential relation up to the following sign
	\begin{equation}\label{Rel1Shifted}
		\big(\mu \circ_i \nu\big)\circ_j\omega = \left\{
		\begin{array}{ll}
			\rule{0pt}{12pt} -(-1)^{|\nu||\omega|} \left(\mu \circ_j \omega\right) \circ_{i+n''-2}\nu\ , & \text{for}\quad  2\leqslant j\leqslant i-1\ ,   \\
			\rule{0pt}{12pt} -\mu \circ_i \left(\nu\circ_{j-i+2} \omega\right)\ ,                        & \text{for}\quad i\leqslant j\leqslant i+n'-2\ , \\
			\rule{0pt}{12pt} -(-1)^{|\nu||\omega|}\left(\mu \circ_{j-n'+2}\omega\right)\circ_{i}\nu\ ,   & \text{for}\quad
			i+n'-1\leqslant j\leqslant n+n'-2\ ,                                                                                                           \\
		\end{array}\right. \
	\end{equation}
	and the equivariance property up to the following sign
	\begin{equation}\label{eq:equivShifted}
		\left\{\begin{array}{l}
			\left(\mu \circ_i \nu\right)^{\tau_{n+n'-2}}=\left(\mu^{\tau_n}\right) \circ_{i-1} \nu\ , \quad \text{for}\quad i\geqslant 3~, \\
			\rule{0pt}{12pt} \left(\mu \circ_2 \nu\right)^{\tau_{n+n'-2}}=-(-1)^{|\mu||\nu|}\left(\nu^{\tau_{n'}}\right) \circ_{n'}
			\left(\mu^{\tau_n}\right)\ ,
		\end{array}\right.
	\end{equation}
	for every $\mu \in \mathcal{P}(X)$, $\nu\in \mathcal{P}(Y)$, and $\omega\in \mathcal{P}(Z)$~.
\end{definition}

\begin{example}\label{ex:End_shifted_cyc_ns_operad}
	The paradigm of shifted cyclic non-symmetric operad is the endormorphism operad
	$\EEnd_{\susp A\oplus  A^*}$ associated to the dg vector space equipped with the degree $-1$ graded-symmetric paring $\langle  f, \susp x\rangle \coloneq (-1)^{|f|} f(x)$ and  $\langle  \susp x, f \rangle \coloneq (-1)^{|f||x|} f(x)$~.
\end{example}

In the same straightforward way, one can define notions of shifted (anti)-cyclic non-symmetric (co)operads. The details are left to the reader.

%%%%%%%%%%%%%%%%%%%%

\subsection{Deformation theory of morphisms of cyclic non-symmetric operads}\label{subsec:subDef}
In this section, we develop the deformation theory of morphisms of cyclic non-symmetric operads
following a  presentation similar to that of \cref{subsec::def_theo_properad}. Let us first recall the shifted version of the classical notion of a Lie algebra.

\begin{definition}[Shifted Lie algebra]
	A \emph{(differential graded) shifted Lie algebra} is a triple $(\g, \d, \{\ , \,\})$  made up of a differential graded vector spaces $(\g,\d)$ and a degree $-1$ symmetric product $\{\ , \,\} \colon \g^{\odot 2} \to \g$
	for which $\d$ is a derivation and satisfying the Jacobi relation
	\[
		\{\{x,y\},z\}+(-1)^{(|x|-1)(|y|+|z|)} \{\{y,z\},x\}+	(-1)^{(|z|-1)(|x|+|y|)}	\{\{z,x\},y\}=0 \ .
	\]
\end{definition}

\begin{remark}
	A shifted Lie algebra structure on $\g$ is equivalent to a Lie algebra structure on the desuspension $\asusp \g$~.
\end{remark}

\begin{definition}[Maurer--Cartan equation]
	The \emph{Maurer--Cartan equation} of a (shifted) Lie algebra is the equation
	\begin{equation}\label{eq:MaEq}
		\d \alpha + {\textstyle \frac12} \{\alpha, \alpha\}=0~.
	\end{equation}
	We only consider solutions of degree $|\alpha|=-1$ (respectively degree $|\alpha|=0$); their set is denoted by $\MC(\g)$~.
\end{definition}

Let us now introduce our main example.

\begin{definition}[Totalisation]
	The \emph{totalisation} of a cyclic module $\P$ is the graded vector space defined by
	\[
		\whP\coloneq \prod_{n \geqslant 1} \P(\la n\ra)^{\C_n}\ .
	\]
\end{definition}

\begin{remark}
	Since we are working here in characteristic $0$, we could have equivalently considered coinvariants instead of invariants in the definition of the totalisation.
\end{remark}

\begin{lemma}[{\cite[Proposition~2.18]{KWZ15}}]
	\label{lem:ModtoDeltaLie}
	The following assignment defines a functor from (shifted) anti-cyclic non-symmetric operads to complete
	(shifted) Lie algebras
	\begin{align*}
		\mathsf{(shifted)}\ \mathsf{anti}\textsf{-}\mathsf{cyclic} \ \mathsf{ns} \ \mathsf{operads}
		 & \to
		\mathsf{complete} \ \mathsf{(shifted)}\ \mathsf{Lie}\ \mathsf{algebras} \\
		\left(\P, d_\P, \circ_i\right)
		 & \mapsto
		\left(\whP, \d, \{\ , \,\}, \F
		\right)\ ,
	\end{align*}
	where the differential $\d$ is induced by the differential $d_\P$ and where the Lie bracket is given by
	\[
		\{\mu,\nu\}\coloneq \sum_{i=2}^n   \mu \circ_i \nu - (-1)^{|\mu||\nu|}\sum_{j=2}^{n'}   \nu \circ_j \mu~,
	\]
	(with the sign $+(-1)^{|\mu||\nu|}$ in front of the second term  in the shifted case),
	for any $\mu \in \P(\la n\ra)^{\C_n}$ and $\nu \in \P(\la n'\ra)^{\C_{n'}}$~,
	and the decreasing filtration
	\[
		\whP=\F_{-1}\supset \F_0 \supset \F_1\supset \cdots \supset \F_N \supset \F_{N+1}\supset \cdots \ ,
	\]
	where $\F_N$ is made up of  series such that the terms of arity $n<N+2$ vanish.
\end{lemma}

\begin{proof}
	The proof amounts to straightforward computations from the axioms \eqref{Rel1}, \eqref{eq:equiv}, and \eqref{eq:equivanti}.
\end{proof}

\begin{remark}
	The above Lie bracket is made up of the skew-symmetrisation of the  product
	which is the usual pre-Lie product $\sum_{i=2}^n   \mu \circ_i \nu$ on the totalisation $\prod_{n \geqslant 1} \P(\la n\ra)$ of the underlying non-symmetric operad \cite{KapranovManin01}. Notice that this pre-Lie product is not stable on the sub-space of invariants with respect to the cyclic groups actions and so the skew-symmetrisation is mandatory here.
\end{remark}

\begin{lemma}\label{lemm:ConvolutionLie}
	The assignment
	\begin{align*}
		 & (\mathsf{anti}\textsf{-}\mathsf{cyclic} \ \mathsf{ns} \ \mathsf{cooperads})^{\mathsf{op}}
		\times
		\mathsf{(shifted)}\ \mathsf{cyclic} \ \mathsf{ns} \ \mathsf{operads}
		\to
		\mathsf{(shifted)}\ \mathsf{anti}\textsf{-}\mathsf{cyclic} \ \mathsf{ns} \ \mathsf{operads}  \\
		 & \left(\big(\CC, d_\CC, \delta_i, \big), \big(\P, d_\P, \circ_i\big)\right) \mapsto
		\Hom \left(\CC, \P\right)\coloneq
		\left(
		\left\{\Hom\big(
		\CC(\langle n\rangle), \P(\langle n\rangle)\big)\right\}_{n\in \NN^*}, \partial,\bigcirc_i
		\right)
	\end{align*}
	defines a functor, where  the cyclic groups  act  by conjugaison,
	where
	the  differential is given by
	\[\partial(f)\coloneq d_\P \circ f - (-1)^{|f|}f \circ d_\CC\ ,\]
	and  where the partial compositions maps are given by
	\[
		\bigcirc_{i}(f\otimes g)\coloneq \circ_i (f\otimes g)\delta_i\ .
	\]
\end{lemma}

\begin{proof}
	This follows in a straightforward way from the defining relations.
\end{proof}

\begin{definition}[Convolution algebra]\label{def:ConDeltaLie}
	The composite of the above two functors produces functorially the \emph{convolution dg Lie algebra}:
	\[
		\widehat{\Hom}(\CC, \P)\coloneq
		\left(
		\prod_{n\geqslant 1} \Hom_{\C_n}\left(\CC(\langle n\rangle), \P(\langle n\rangle)\right),
		\partial, \{-,-\}
		\right)\ .
	\]
\end{definition}

\begin{definition}[Twisting morphism]
	The solutions  to the Maurer--Cartan equation~\eqref{eq:MaEq} in the convolution algebra $\widehat{\Hom}(\CC, \P)$  are called  \emph{twisting morphisms} from $\CC$ to $\P$; their set is denoted by $\Tw(\CC, \P)$~.
\end{definition}

As first example, we consider the convolution dg Lie algebra associated to the
anti-cyclic non-symmetric cooperad $\AAs^\antish$ of \Cref{ex:coop_As_anticyclic} and the
cyclic non-symmetric endomorphism
operad $\EEnd_{V}$ of \Cref{ex:cyc_ns_operads}.

\begin{definition}[Cyclic $\rmA_\infty$-algebra]
	The algebraic structure defined by a twisting morphism in $\Tw\left(\AAs^{\ac}, \EEnd_V\right)$
	is called a \emph{cyclic $\rmA_\infty$-algebra} structure on $V$~.
\end{definition}

\begin{proposition}\label{prop:strucure_on_V}
	A cyclic $\rmA_\infty$-algebra structure on a dg module $V$ induces products $m_n \colon V^{\otimes n} \to V$ of degree $n-2$, for $n\geqslant 2$, satisfying
	\begin{equation}\label{eqn:Ainfini}
		\partial(m_n)=\sum_{p+q+r=n} (-1)^{pq+r+1} m_{p+1+r} \circ \left(\id^{\otimes p}\otimes m_q\otimes\id^{\otimes r}    \right) ~,
	\end{equation}
	for any $n\geqslant 2$~, and equipped with a symmetric bilinear  form $\langle\ ,\, \rangle$ satisfying
	\begin{equation}\label{eqn:AinfiniSymm}
		\langle v_1, m_n(v_2, \ldots, v_{n+1}) \rangle=(-1)^{|v_1|(|v_2| +\cdots +|v_{n+1}|)+n}\langle v_{n+1}, m_n(v_1, \ldots, v_{n}) \rangle~,
	\end{equation}
	for any $n\geqslant 2$~.
	Both structures are equivalent when the pairing $\langle\ ,\, \rangle$ is non-degenerate.
\end{proposition}

\begin{proof}
	We consider the morphism of chain complexes $\Theta \colon V^{\otimes (n+1)} \to \Hom(V^{\otimes n}, \allowbreak V)$ defined by
	\[\Theta(v_1, \ldots, v_{n+1})\coloneq v_1\la v_2, -\ra \cdots \la v_{n+1}, -\ra~.\]
	Let $\alpha\in \Tw\left(\AAs^{\ac}, \EEnd_V\right)$ be a twisting morphism. Under the notation
	$m_n\coloneq \Theta(\alpha(\upsilon_{n+1}))$, the image of the Maurer--Cartan equation satisfied by $\alpha$ under $\Theta$ gives the abovementioned equation~\eqref{eqn:Ainfini} of $\rmA_\infty$-algebras.
	The equation~\eqref{eqn:AinfiniSymm} comes from the signature representation on $\AAs^{\ac}(\la n\ra)$ and the equivariance of the map $\alpha$. The map $\Theta$ is an isomorphism if and only if the symmetric bilinear form $\langle\ ,\, \rangle$ is non-degenerate. In this case, the two algebraic structures are equivalent.
\end{proof}

More generally, one can consider the anti-cyclic non-symmetric cooperad
$\cAAs^\antish$ of \Cref{ex:coop_uAs_anticyclic}. The canonical surjection
$\cAAs^\antish\twoheadrightarrow \AAs^\antish$ of anti-cyclic non-symmetric cooperads induces an embedding of dg Lie algebras
\[\widehat{\Hom}\left(\AAs^{\ac}, \EEnd_V\right) \hookrightarrow \widehat{\Hom}\left(\cAAs^{\ac}, \EEnd_V\right)~.\]

\begin{definition}[Cyclic curved $\rmA_\infty$-algebra]
	The algebraic structure defined by a twisting morphism in $\Tw\left(\cAAs^{\ac}, \EEnd_V\right)$
	is called a \emph{cyclic curved $\rmA_\infty$-algebra} structure on $V$~.
\end{definition}

\begin{proposition}\label{prop:strucure_on_V}
	A cyclic curved $\rmA_\infty$-algebra structure on a graded module $V$ is  made up of products $m_n \colon V^{n} \to V$ of degree $n-2$, for $n\geqslant 0$, satisfying
	\begin{equation}\label{eqn:AinfiniCurved}
		\sum_{p+q+r=n} (-1)^{pq+r+1} m_{p+1+r} \circ \left(\id^{\otimes p}\otimes m_q\otimes\id^{\otimes r}\right)=0 ~,
	\end{equation}
	for any $n\geqslant 0$~, and equipped with a symmetric bilinear  form $\langle\ ,\, \rangle$ satisfying
	\begin{equation}\label{eqn:AinfiniSymmCurved}
		\langle v_1, m_n(v_2, \ldots, v_{n+1}) \rangle=(-1)^{|v_1|(|v_2| +\cdots +|v_{n+1}|)+n}\langle v_{n+1}, m_n(v_1, \ldots, v_{n}) \rangle~,
	\end{equation}
	for any $n\geqslant 2$~. Both structures are equivalent when the pairing $\langle\ ,\, \rangle$ is non-degenerate.
\end{proposition}

\begin{proof}
	This proof is similar to the one given above and uses the same identifications.
\end{proof}

Therefore  a cyclic  $\rmA_\infty$-algebra is a cyclic curved $\rmA_\infty$-algebra  such that the curvature operation $m_0$ is trivial. (In this case, the operation $m_1$ squares to zero and is considered as the differential of $V$.) The above-mentioned embedding of dg Lie algebras shows that the deformation theory of the former is included inside the deformation theory of the latter.

\subsection{Generalised necklace algebra}\label{subsec::pCY_alg}
In this section, we consider the special case  $V=\susp A\oplus A^*$ equipped with its canonical degree $-1$ skew-symmetric pairing: $\langle  f, \susp x\rangle \coloneq (-1)^{|f|} f(x)$~.

\begin{proposition}
	The convolution algebra $\widehat{\Hom}\left(\AAs^{\ac}, \allowbreak \EEnd_{\susp A\oplus A^*}\right)$
	is a shifted Lie algebra with  underlying graded vector space isomorphic to
	\begin{align}\label{eq::Lie_adm_of_cyclicAinfty}
		\susp^{2}
		\prod_{N \geqslant 3}\left(
		\bigoplus_{1\leqslant m < N}
		\left(
		\bigoplus_{\lambda_1+\cdots+\lambda_m=n}
		A\otimes
		\left((\susp A)^*\right)^{\otimes \lambda_1}
		\otimes  A \otimes
		\left((\susp A)^*\right)^{\otimes \lambda_2}
		\cdots\otimes  A \otimes
		\left((\susp A)^*\right)^{\otimes \lambda_m}
		\right)^{\C_m}
		\oplus	\left(\left( (\susp A)^*\right)^{\otimes N}\right)^{\C_N}\right)~,
	\end{align}
	where the sum runs over
	ordered partitions $\lambda_1+\cdots+\lambda_m=n$ of $n= N-m$
	into non-negative integers,
	and with the bracket given by
	\begin{align*}
		 & \left\{	\susp^{2}  a_1\otimes \cdots \otimes a_N, 	\susp^{2}
		b_1\otimes \cdots \otimes b_{N'}\right\}=                                                                                                  \\
		 & \susp^{2}
		\sum_{i=2}^N
		(-1)^{|a_1|+\cdots+|a_{i-1}|+|{b}|\left(|a_{i+1}|+\cdots |a_{N}| \right)}
		\la a_i, b_1\ra\,   a_1 \otimes \cdots \otimes a_{i-1} \otimes b_2\otimes \cdots  \otimes b_{N'}\otimes a_{i+1} \otimes \cdots \otimes a_N \\
		 & +
		\susp^{2}
		\sum_{j=2}^{N'}
		(-1)^{|b_1|+\cdots+|b_{j-1}|+|{a}|\left(|b_{1}|+\cdots |b_{j}| \right)}
		\la b_j, a_1\ra\,   b_1 \otimes \cdots \otimes b_{j-1} \otimes a_2\otimes \cdots  \otimes a_{N}\otimes b_{j+1} \otimes \cdots \otimes b_{N'} ~,
	\end{align*}
	with $a_k, b_l \in  A \oplus (\susp A)^*$,
	$a\coloneq a_1\otimes \cdots \otimes a_N$, and
	$b \coloneq b_1\otimes \cdots \otimes b_{N'}$~.
\end{proposition}

\begin{proof}
	Since the cooperad $\AAs^{\ac}$ is anti-cyclic non-symmetric and since the operad $\EEnd_{\susp A\oplus  A^*}$ is
	shifted cyclic non-symmetric, the convolution algebra
	$\widehat{\Hom}\left(\AAs^{\ac}, \allowbreak \EEnd_{\susp A\oplus  A^*}\right)$ is a shifted Lie algebra by Lemmata~\ref{lem:ModtoDeltaLie} and \ref{lemm:ConvolutionLie}.
	Using now capital $N$ for the arity in the cyclic operad context, the cyclic module $\AAs^{\ac}(\langle N \rangle)$ is isomorphic to $\susp^{N-2}$, for any $N\geqslant 3$. So the underlying graded space of the convolution algebra is isomorphic to
	\[
		\prod_{N\geqslant 3} \Hom_{\C_N}\left(s^{N-2}, \left(\susp A\oplus  A^*\right)^{\otimes N}\right)\cong
		\susp^{2}\prod_{N\geqslant 3}\left(\susp^{-N}\otimes \left(\susp A\oplus  A^*\right)^{\otimes N}  \right)^{\C_N} \cong
		\susp^{2}\prod_{N\geqslant 3}\left( \left(A\oplus  (\susp A)^*\right)^{\otimes N}  \right)^{\C_N}    ~.\]
	Spitting according to which  part of the direct sum comes from either $A$ or $\susp A^*$, one obtains the  graded space   displayed above in \eqref{eq::Lie_adm_of_cyclicAinfty}.
	The abovementioned isomorphism is given explicitly by
	\[\Phi\left(\susp^{N-2}\mapsto \susp a_1 \otimes \cdots \otimes \susp a_N \right)
		\coloneq
		(-1)^{\frac{N(N+1)}{2}+N|a_N|+\cdots +|a_1|}\,
		\susp^{2}\,a_1 \otimes \cdots \otimes a_N~,
	\]
	where $a_1,\ldots, a_N\in A\oplus  (\susp A)^*$~.
	Transporting the shifted Lie bracket of the convolution algebra (\cref{def:ConDeltaLie}) under this isomorphism gives  the shifted Lie bracket of the above statement.
\end{proof}

\begin{definition}[Generalised necklace Lie algebra]\label{def::necklaceLiealg}
	The \emph{generalised necklace Lie algebra} associated to the dg vector space $A$ is
	the desuspension of the above shifted Lie sub-algebra:
	\[
		\nec_A\coloneq \left(
		\susp
		\prod_{N \geqslant 3}	\left(
		\bigoplus_{1\leqslant m < N}
		\left(
		\bigoplus_{\lambda_1+\cdots+\lambda_m=n}
		A\otimes
		\left((\susp A)^*\right)^{\otimes \lambda_1} \otimes
		\cdots\otimes  A \otimes
		\left((\susp A)^*\right)^{\otimes \lambda_m}
		\right)^{\C_m}\right)	,
		d, \{-,-\}\right) \ .
	\]
\end{definition}

There is a crucial point for us in the present case: when $V=\susp A\oplus  A^*$, the Lie bracket on
the generalised necklace Lie algebra splits into two, depending on whether one applies the linearity pairing to $f\otimes \susp x$ or to $\susp x\otimes f$~, where $f\in A^*$ and $x\in A$.
More precisely, we denote by $X\ast Y$ the summand of $\{X,Y\}$ made up of the terms where one applies the linear pairing $\langle f,\susp x\rangle$, where $f\in A^*$ comes from $X$ and $x\in A$ comes from $Y$~.
So the  Lie bracket is equal to the skew-symmetrisation of the product $\ast$:
\[
	\{X, Y\}=X\ast Y - (-1)^{|X||Y|}Y\ast X\ .
\]

\begin{lemma}
	The binary product $\ast$ satisfies the relation of a Lie-admissible algebra:
	\[
		\sum_{\sigma \in \Sy_3} \mathrm{sgn}(\sigma) \mathrm{assoc}(-,-,-)^{\sigma} = 0 \ ,
	\]
	where $\mathrm{assoc}$ stands for the associator:
	$\mathrm{assoc}(x,y,z) \coloneq 	(x\ast y) \ast z - x\ast(y\ast z)$~, for every $x,y,z$~.
\end{lemma}

\begin{proof}
	This is actually the definition of a Lie-admissible bracket: its  skew-symmetrised bracket satisfies the  Jacobi relation.
\end{proof}

\begin{remark}
Contrary to what is claimed in \cite[Section~2]{IKV19}, the operation $\ast$ does not satisfy any stronger relation, like the pre-Lie relation, in general.
\end{remark}

Working \emph{mutatis mutandis} with the anti-cyclic non-symmetric cooperad $\cAAs^\antish$
encoding cyclic curved $\rmA_\infty$-algebras, one gets the following more general context.

\begin{definition}[Curvature necklace Lie-admissible algebra]\label{def::necklaceLiealg}
	The \emph{curvature necklace Lie-admissible algebra} associated to the dg vector space $A$ is
	defined by
	\[
		\cnec_A\coloneq \left(
		\susp
		\prod_{N \geqslant 1}	\left(
		\bigoplus_{1\leqslant m < N}
		\left(
		\bigoplus_{\lambda_1+\cdots+\lambda_m=n}
		A\otimes
		\left((\susp A)^*\right)^{\otimes \lambda_1}
		\otimes
		\cdots\otimes  A \otimes
		\left((\susp A)^*\right)^{\otimes \lambda_m}
		\right)^{\C_m}\right)	,
		d, \ast\right) \ .
	\]
\end{definition}

There is a canonical embedding of dg Lie-admissible algebras
\[\nec_A \hookrightarrow \cnec_A~.\]
The skew-symmetrisation of the Lie-admissible product $\ast$ produces a Lie bracket and thus the curvature necklace Lie algebra.

\begin{definition}[Tensorial curved pre-Calabi--Yau algebra]
	A structure of an \emph{tensorial curved pre-Calabi--Yau algebra} on a graded vector space $A$ is a Maurer--Cartan element in the
	curved necklace  Lie-admissible algebra, that is a degree $-1$ element $\alpha$ satisfying:
	\[\tfrac12\left\{\alpha, \alpha\right\}=\alpha\ast \alpha= 0~.\]

\end{definition}

Otherwise stated, a tensorial curved pre-Calabi--Yau structure on a graded vector space $A$ corresponds to a shifted  cyclic curved $\rmA_\infty$-algebra structure on $\susp A \oplus A^*$, i.e. a  cyclic curved $\rmA_\infty$-algebra structure on $A \oplus (\susp A)^*$, with trivial terms on $\left((\susp A)^*\right)^{\otimes N}$~. In particular, it carries a curved $\rmA_\infty$-algebra structure on $A$, which is a curved $\rmA_\infty$-sub-algebra of $A \oplus (\susp A)^*$ by this latter condition.

\subsection{Higher cyclic complex and the main theorem}
Recall that there is a canonical inclusion
\begin{equation}\label{eq:Inclusion}
	\begin{array}{rcl}
		W\otimes V^* & \hookrightarrow & \Hom\left(V, W\right)            \\
		x\otimes f   & \mapsto         & \big( v  \mapsto  x f(v) \big)~,
	\end{array}
\end{equation}
which is an isomorphism if and only if $V$ is finite dimensional in each degree.
It leads to considering
$\Hom\left((\susp A)^{\otimes \lambda_1}\otimes \cdots \otimes (\susp A)^{\otimes \lambda_m},  A^{\otimes m}\right)$
instead of
$A\otimes \left((\susp A)^*\right)^{\otimes \lambda_m} \otimes  \cdots\otimes  A \otimes \left((\susp A)^*\right)^{\otimes \lambda_1}
$~.
Given two maps
\[\rmF\in \Hom _{\C_m}
	\left(	\bigoplus_{\lambda_1+\cdots+\lambda_m=n} \bigotimes_{j=1}^m
	(\susp A)^{\otimes \lambda_j},  A^{\otimes m}
	\right)\quad   \text{and} \quad
	\rmG\in \Hom _{\C_{m'}}
	\left(	\bigoplus_{\lambda'_1+\cdots+\lambda'_{m'}=n'} \bigotimes_{j=1}^{m'}
	(\susp A)^{\otimes \lambda'_j},  A^{\otimes m'}
	\right)~, \]
we identify the suspension $\susp \rmF$ with the induced map
$\bigotimes_{j=1}^{m}
	(\susp A)^{\otimes \lambda_j} \to \susp A \otimes A^{\otimes m-1}$~, and similarly for $\susp \rmG$~.
For any $1\leqslant i \leqslant n$, we denote by $\susp\rmF \circledast^1_i \susp\rmG$:
\begin{multline*}
	(\susp A)^{\otimes \lambda_1} \otimes \cdots \otimes (\susp A)^{\otimes \lambda_{k-1}}
	\otimes (\susp A)^{\otimes \lambda'_1+\lambda_k-l}\otimes (\susp A)^{\otimes \lambda'_2}\otimes \cdots \otimes
	(\susp A)^{\otimes k-1+\lambda'_{m'}} \otimes
	(\susp A)^{\otimes \lambda_{k+1}}\otimes \cdots \otimes (\susp A)^{\otimes \lambda_m} \\
	\to \susp A^{\otimes m+m'-1}~,
\end{multline*}
with $i=\lambda_1+\cdots+\lambda_{k-1}+l$ and $1\leqslant l \leqslant \lambda_k$,
the partial composite of the two multilinear operations $\susp \rmF$ and $\susp \rmG$
including the permutations of inputs and outputs depicted in  \cref{Fcirc1iG}, for it to leave in the space of cyclically invariant maps.

\begin{figure*}[h!]
	\begin{tikzpicture}[scale=1,baseline =(n.base)]
		\node (n) at (1,2) {};
		\draw[thick] %% Sorties de G
		(2,4.8) --  (2,1.5) %node[near end, right] {\small $k$}
		(3.2,4.8) to[out=270,in=90] (5.7,1.5) --
		(5.7,0.8) to[out=270,in=90] (3,-2)
		node[below] {\small $k+1$}
		(6,4.8) to[out=270,in=90] (8.2,1.5) --
		(8.2,0.8) to[out=270,in=90] (5.8,-2) node[below] {\small $k+m'-1$};
		\draw[thick] %entrees
		%% lambda_1
		(-1.4,6.5)  to[out=270,in=90] (-1,1.5)
		(-1.2,6.5)  to[out=270,in=90] (-0.8,1.5)
		(-1,6.5) to[out=270,in=90] (-0.6,1.5)
		(-1.2,6.56) node[above] {\small$\lambda_1$}
		%% lambda_k-1
		(0.4,6.5)  to[out=270,in=90] (0.8,1.5)
		(0.6,6.5)  to[out=270,in=90] (1,1.5)
		(0.8,6.5) to[out=270,in=90] (1.2,1.5)
		(0.6,6.56) node[above] {\small$\lambda_{k-1}$}
		(1.8,6.5) to[out=270,in=90] (1.8,5.5)
		(2,6.5) to[out=270,in=90] (2,5.5)
		(2.1,6.5) node[above] {\small $\lambda'_1+\lambda_k-l\quad $}
		(3,6.5) to[out=270,in=90] (3,5.5)
		(3.2,6.5) node[above] {\small $\lambda'_2$}
		to[out=270,in=90] (3.2,5.5)
		(3.4,6.5) to[out=270,in=90] (3.4,5.5)
		(6,6.5) to[out=270,in=90] (6,5.5)
		(6.2,6.5) to[out=270,in=90] (6.2,5.5)
		(5.9,6.5) node[above] {\small $l-1+\lambda'_{m'}$}
		;
		\draw[draw=white,double=black,double distance=2*\pgflinewidth,thick]
		(7,6.5) -- (7,4.8) to[out=270,in=90] (2.8,1.5)
		(7.2,6.5)
		-- (7.2,4.8) to[out=270,in=90] (3,1.5)
		(7.4,6.5) -- (7.4,4.8) to[out=270,in=90] (3.2,1.5)
		(7.3,6.56)  node[above] {\small $\lambda_{k+1}$}
		(9,6.5) -- (9,4.8) to[out=270,in=90] (4.8,1.5)
		(9.2,6.5)
		-- (9.2,4.8) to[out=270,in=90] (5,1.5)
		(9.4,6.5) -- (9.4,4.8) to[out=270,in=90] (5.2,1.5)
		(9.2,6.56) node[above] {\small $\lambda_{m}$};
		%
		%%% SORTIES
		\draw[thick]
		(-0.8,0.8) -- (-0.8,-2) node[below] {\small $1$}
		(1,0.8) -- (1,-2) node[below] {\small $k-1$}
		(2,0.8) -- (2,-2) node[below] {\small $k$}
		(0,-2.1) node[below] {\small $\cdots$}
		(4.3,-2.1) node[below] {\small $\cdots$}
		(8.27,-2.1) node[below] {\small $\cdots$}
		(-0.3,6.6) node[above] {\small $\cdots$}
		(4.3,6.6) node[above] {\small $\cdots$}
		(8.3,6.6) node[above] {\small $\cdots$}	;
		\draw[draw=white,double=black,double distance=2*\pgflinewidth,thick]
		(3,0.8) to[out=270,in=90] (7.4,-2) node[below] {\small $k+m'$}
		(5,0.8) to[out=270,in=90] (9.4,-2) node[below] {\small $m+m'-1$};
		\draw[thick]
		(2.2,4.7) to[out=270,in=90] (2.2,1.5)
		(2.4,4.7)  to[out=270,in=90] (2.4,1.5)
		(2.2,6.5) to[out=270,in=90] (2.2,5.6)
		(2.4,6.5)  to[out=270,in=90] (2.4,5.6);
		%	(2.2,6.5) to[out=270,in=90] (2.2,5.8) to[out=270,in=90] (2.1,5.15) to[out=270,in=90] (2.2,4.5)
		%	(2.4,6.5) to[out=270,in=90] (2.4,5.8) to[out=270,in=90] (2.3,5.15) to[out=270,in=90] (2.4,4.5)
		%%% supers croisements
		\draw[draw=white,double=black,double distance=2*\pgflinewidth,thick] %% A MODIF
		% (5.6,6.5) -- (5.6,4.8) to[out=270,in=90] (2.1,1.5)
		% (5.8,6.5) -- (5.8,4.8) to[out=270,in=90] (2.3,1.5);
		(5.6,6.5) -- (5.6,5.6)
		(5.6,4.7)  to[out=270,in=90] (1.6,1.5)
		(5.8,6.5) -- (5.8,5.6)
		(5.8,4.7)  to[out=270,in=90] (1.8,1.5);
		%	(5.8,6.5) -- (5.8,5.6) to[out=270,in=90] (5.7,5.15)  to[out=270,in=90] (5.8,4.5) to[out=270,in=90] (1.8,1.5);
		%
		%%% RECTANGLES
		\draw[fill=white]
		(-1.1,0.8) rectangle (5.3,1.5) node[midway] {\small $\susp\rmF$}
		(1.7,4.8) rectangle (6.3,5.5) node[midway] {\small $\susp\rmG$};
	\end{tikzpicture}
	\caption{Illustration of the partial composite $\susp\rmF \circledast^1_i \susp\rmG$~.}
	\label{Fcirc1iG}
\end{figure*}

\begin{definition}[Higher cyclic complex]\label{def:HighHochCx}
	The \emph{higher cyclic complex} associated to a dg vector space $A$
	is the Lie-admissible algebra
	\[
		\hhc_A \coloneq
		\left(\susp
		\prod_{N \geqslant 1}
		\left(
		\bigoplus_{1\leqslant m < N}
		\Hom _{\C_m}
		\left(	\bigoplus_{\lambda_1+\cdots+\lambda_m=n} \bigotimes_{j=1}^m
		(\susp A)^{\otimes \lambda_j},  A^{\otimes m}
		\right)
		\right) \ , \partial \ , \circledast \right) \ ,
	\]
	where
	\[\susp \rmF\circledast \susp \rmG \coloneq \sum_{i=1}^n \susp \rmF \circledast^1_i \susp\rmG ~.\]
\end{definition}

\begin{proposition}
	The inclusion \eqref{eq:Inclusion}  induces a monomorphism of dg Lie-admissible algebras
	\[\mathfrak{cneck}_A \hookrightarrow \hhc_A~, \]
	which is an isomorphism if and only if $A$ is finite dimensional in each degree.
\end{proposition}

\begin{proof}
	It is straightforward to check that the product $\circledast$ is well-defined, that is it lands in cyclically invariant maps, and that it is Lie-admissible. (This is also a direct corollary of the proof of \cref{thm:main} below.)
	The map $\Psi : \mathfrak{cneck}_A \to \hhc_A$ is explicitly given by
	\[\Psi\left(\susp \, a_1 \otimes f_1 \otimes \cdots \otimes a_m\otimes f_m\right)
		\coloneq
		(-1)^\zeta\, \susp \,
		a_1 \otimes \cdots \otimes a_m \otimes f_m \otimes \cdots \otimes f_1~,
	\]
	with $a_j\in A$ and $f_j\in \left((\susp A)^*\right)^{\otimes \lambda_{j}}$, for $1\leqslant j \leqslant m$,
	where
	\[
		\zeta=|f_1|(|a_2|+|f_2|+\cdots+|a_m|+|f_m|)+|f_2|(|a_3|+|f_3|+\cdots+|a_m|+|f_m|)+\cdots +
		|f_{m-1}|(|a_m|+|f_m|)~.
	\]
	and
	where the right-hand term is interpreted as a map in
	$\susp \Hom\left((\susp A)^{\otimes \lambda_1}\otimes \cdots \otimes (\susp A)^{\otimes \lambda_m},  A^{\otimes m}\right)$
	under the inclusion \eqref{eq:Inclusion}.
	Using the Koszul sign rule and the Koszul sign convention, it is straightforward to check that
	\begin{multline*}
		\Psi\left(\susp \,  a_1 \otimes f_1 \otimes \cdots \otimes a_m\otimes f_m\right)
		\circledast
		\Psi\left(\susp \, b_1 \otimes g_1 \otimes \cdots \otimes b_{m'}\otimes g_{m'}\right)\\=
		\Psi\big((\susp \, a_1 \otimes f_1 \otimes \cdots \otimes a_m\otimes f_m)
		\ast
		(\susp \, b_1 \otimes g_1 \otimes \cdots \otimes b_{m'}\otimes g_{m'})\big)~.
	\end{multline*}
\end{proof}

\begin{definition}[Curved pre-Calabi--Yau algebra]
	A structure of \emph{a curved pre-Calabi--Yau algebra} on a graded vector space $A$ is a Maurer--Cartan element in the
	higher cyclic Lie-admissible algebra $\hhc_A$,  that is a degree $-1$ element $\alpha$ satisfying:
	\[\tfrac12\left\{\alpha, \alpha\right\}=\alpha\ast \alpha= 0~.\]
\end{definition}

The main result of this paper is the following theorem.

\begin{theorem}\label{thm:main}
	For any dg vector space $A$,
	there is a canonical and functorial isomorphism of dg Lie-admissible algebras
	\[
		\mathfrak{c}\convDPois_{A}  \cong \hhc_A \ .
	\]
\end{theorem}

\begin{proof}
	Since the proof is a straightforward, long, and not really enligthening computation, we postpone it to  \cref{appendice}.
\end{proof}

In the end, we get the following monomorphism of dg Lie-admissible algebras
\[
	\cnec_A \hookrightarrow \hhc_A\cong	\mathfrak{c}\convDPois_A~,
\]
which is an isomorphism if and only if $A$ is degree-wise finite dimensional.

\begin{corollary}\label{cor::hDPois_pCY}\leavevmode

\begin{enumerate}
\item Curved pre-Calabi--Yau algebra structures  are equivalent to curved homotopy double Poisson gebra structures.
\item Any tensorial curved pre-Calabi--Yau algebra carries a canonical curved pre-Calabi--Yau algebra structure. These two notions are equivalent when $A$ is degree-wise finite dimensional. 
\end{enumerate}
\end{corollary}

\begin{proof}
	It is a direct corollary of \cref{thm:main}.
\end{proof}

%%%%%%%%%%%%%%%%%%%%%%%%  SECTION 3

\section{Pre-Calabi--Yau algebras, $\rmV_\infty$-gebras, and homotopy double Poisson gebras}\label{sec:pCYViDP}

In this section, we recall the more commun notion of a \emph{pre-Calabi--Yau algebra} \cite{KTV21}, obtained after killing the curvature  in the abovementioned equivalent notions, and we compare it to the notion of a $\rmV_\infty$-gebras \cite{PT19, TZ07Bis}. 
Homotopy double Poisson gebras are particular examples of pre-Calabi-Yau algebras which, in turn, form the genus $0$ part of $\rmV_\infty$-gebras. These three algebraic categories are actually encoded by different but related coproperads. So we make the various properadic decomposition maps explicit using a certain combinatorics prompted naturally by the slicing of planar corollas of the previous section. In the end, the first combinatorial objects that we get are equivalent to the ones used in \cite{KTV21}: this shows their properadic origin and this will allow us to settle all the homotopical properties of pre-Calabi--Yau algebras, homotopy double Poisson gebras, and $\rmV_\infty$-gebras, in the next section. 

\subsection{Pre-Calabi--Yau algebras and $\rmV_\infty$-gebras}\label{subsec:TypesofAlg}

The notion of a \emph{curved} pre-Calabi--Yau algebra is not the one which appears in the literature so far: the most commun notion has no curvature. Since the data of a curved pre-Calabi--Yau algebra is equivalent to the data of a homotopy curved double Poisson gebra by \cref{cor::hDPois_pCY}, this notion is made up of structure maps 
$\calm_{\lambda_1, \ldots, \lambda_m} \colon A^{\otimes n} \too A^{\otimes m}$~, for any 
partition $\lambda_1+\cdots+\lambda_m=n$ of $n\geqslant 0$~, 
see  \cref{prop:chDPois}.

\begin{definition}[{Pre-Calabi--Yau algebra \cite{KTV21}}]\label{def:pCY}
	A \emph{pre-Calabi--Yau algebra} is
	a curved pre-Calabi--Yau algebra whose curvature vanishes, i.e. $\calm_0 =0$. In other words, 
		a pre-Calabi--Yau algebra is a  dg vector space $(A, d)$ equipped with a collection of operations
	\[
		\calm_{\lambda_1, \ldots, \lambda_m} \colon A^{\otimes n} \too A^{\otimes m}
	\]
	of degree $n-2$, for any ordered partition $\lambda_1+\cdots+\lambda_m=n$ of $n\geqslant 0$ into non-negative integers, for $m\geqslant 2$, and $\lambda_1\geqslant 2$, for $m=1$, satisfying the following relations.
	\begin{description}
		\item[\sc cyclic skew symmetry]
		      \[\calm_{\lambda_2, \ldots, \lambda_m, \lambda_1}=
			      (-1)^{\lambda_1(n-m-1)-n}
			      \ \tau_m^{-1}\cdot\calm_{\lambda_1, \lambda_2, \ldots, \lambda_m}\cdot \tau_{\lambda_1, \ldots, \lambda_m}
			      ~, \]
		      where $\tau_{\lambda_1, \ldots, \lambda_m} \in \Sy_n$ permutes cyclically the blocks of  size $\lambda_1, \ldots, \lambda_m$~.

		\item[\sc pre-Calabi--Yau relations]
  \begin{multline}\label{eq:pCYrelations}
			      \partial \left(\calm_{\lambda_1, \ldots, \lambda_m}\right)=\\
			      \sum_{k=1}^{m}
			      \sum_{\sigma \in \C_m}
			      \sum_{0\leqslant p \leqslant  \lambda_{\sigma(m)}\atop 0\leqslant q \leqslant  \lambda_{\sigma(k)}}
			      (-1)^{\xi}\,
			      \sigma^{-1}\cdot\left(\calm_{\lambda_{\sigma(1)}, \ldots, \lambda_{\sigma(k-1)}, p+1+q}
			      \circ^1_{i}
			      \calm_{\lambda_{\sigma(k)-q}, \lambda_{\sigma(k+1)}, \ldots, \lambda_{\sigma(m-1)},
				      \lambda_{\sigma(m)-p}}\right)\cdot \omega~,
		      \end{multline}		      with the same notations as in \eqref{eq:RelHoDoublePoiss}.
	\end{description}
	\end{definition}

\begin{proposition}	
The category of pre-Calabi--Yau algebras is the category of gebras over the cobar construction of the infinitesimal codioperad 
\[\rmC_{\rm pCY}\coloneq 
\frac{\mathrm{c}\DPois^\antish}{\field u^*}
\cong \left(\frac{\mathrm{u}\DPois^!}{\field u}\right)^*~.\]
\end{proposition}

\begin{proof}		
More precisely, we identify $\frac{\mathrm{u}\DPois^!}{\field u}$ with its canonical section, that is the sub-$\Sy$-bimodule of $\mathrm{u}\DPois^!$ with  trivial component of arity $(1,0)$. 
It is straightforward to check that the properad structure on $\mathrm{u}\DPois^!$ induces a canonical properad structure on it.
As proved in \cref{prop:InfDecDPoisAC}, any composite along a positive genus graph vanishes; so the linear dual 
infinitesimal coproperad forms a codioperad. The rest of the statement follows from \cref{cor::hDPois_pCY} and \cref{def:HoCuDPois}.
\end{proof}

\begin{remark}
An infinitesimal codioperad of \emph{multi-corollas}, defined by an infinitesimal decomposition map, and its cobar construction were  introduced in \cite[Section~6.5.3]{KTV21}, see also \cite{Yeung22}. 
It is shown in \cite[Proposition~47]{KTV21} to encode pre-Calabi--Yau algebras. 
So this infinitesimal codioperad of multi-corollas is isomorphic to the present infinitesimal codioperad $\rmC_{\rm pCY}$, see \cref{prop:CuttingInfDec} for a precise isomorphism. 
\end{remark}

The  infinitesimal codioperad $\rmC_{\rm pCY}$ is actually conilpotent since no infinite series can appear in the iterations of the infinitesimal decomposition maps. This property is mandatory to proceed further with the homotopical properties of pre-Calabi--Yau algebras: for instance, the notion of an $\infty$-morphism of \cref{sec:InftyMor} requires the \emph{full} decomposition map $\Delta \colon \rmC_{\rm pCY} \to \rmC_{\rm pCY} \boxtimes\rmC_{\rm pCY}$, that we  will make explicit in the next section~\ref{subsec:CoproduitsCpCY}. 

\begin{proposition}
The category of homotopy double Poisson gebras is equivalent to the sub-category of pre-Calabi--Yau algebras 
with vanishing structure maps 
$\calm_{\lambda_1, \ldots, \lambda_m}$, when at least one number of indices $\lambda_i$ is equal to $0$. 
\end{proposition}

\begin{proof}
This is a straightforward corollary of the morphism of codioperads 
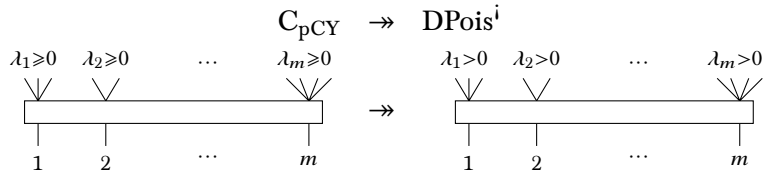
\[
\begin{array}{rcl} 
\rmC_{\rm pCY} &\twoheadrightarrow &\DPois^{\ac} \\ 
\vcenter{\hbox{
	\begin{tikzpicture}[scale=0.6]
		\coordinate (A) at (0,1);
		\draw[thin]
		(A) ++(0,-1) node[below] {${\scriptstyle{1}}$} --
		(A) -- ++(0,0.5)
		(A) -- ++(-.3,0.5)
		(A) -- ++(.3,0.5)
		(A)++(-.1,0.5) node[above] {${\scriptstyle{\lambda_1\geqslant 0}}$};
		\coordinate (A) at (1.5,1);
		\draw[thin]
		(A) ++(0,-1) node[below] {$\scriptstyle{2}$} --
		(A) -- ++(-.3,0.5)
		(A) -- ++(.3,0.5)
		(A)++(0,0.5) node[above] {${\scriptstyle{\lambda_2\geqslant 0}}$};
		\draw (3.75,0) node[below] {$\scriptstyle{\cdots}$};
		\draw (3.75,1.5) node[above] {$\scriptstyle{\cdots}$};
		\coordinate (A) at (6,1);
		\draw[thin]
		(A) ++(0,-1) node[below] {$\scriptstyle{m}$} --
		(A) -- ++(-.5,0.5)
		(A) -- ++(-.2,0.5)
		(A) -- ++(.2,0.5)
		(A) -- ++(.5,0.5)		
		(A)++(-.1,0.5) node[above] {${\scriptstyle{\lambda_m\geqslant 0}}$};
		\draw[fill=white] (-0.3,0.5) rectangle (6.3,1);
	\end{tikzpicture}}}
&\twoheadrightarrow&
\vcenter{\hbox{
	\begin{tikzpicture}[scale=0.6]
		\coordinate (A) at (0,1);
		\draw[thin]
		(A) ++(0,-1) node[below] {${\scriptstyle{1}}$} --
		(A) -- ++(0,0.5)
		(A) -- ++(-.3,0.5)
		(A) -- ++(.3,0.5)
		(A)++(-.1,0.5) node[above] {${\scriptstyle{\lambda_1> 0}}$};
		\coordinate (A) at (1.5,1);
		\draw[thin]
		(A) ++(0,-1) node[below] {$\scriptstyle{2}$} --
		(A) -- ++(-.3,0.5)
		(A) -- ++(.3,0.5)
		(A)++(0,0.5) node[above] {${\scriptstyle{\lambda_2> 0}}$};
		\draw (3.75,0) node[below] {$\scriptstyle{\cdots}$};
		\draw (3.75,1.5) node[above] {$\scriptstyle{\cdots}$};
		\coordinate (A) at (6,1);
		\draw[thin]
		(A) ++(0,-1) node[below] {$\scriptstyle{m}$} --
		(A) -- ++(-.5,0.5)
		(A) -- ++(-.2,0.5)
		(A) -- ++(.2,0.5)
		(A) -- ++(.5,0.5)		
		(A)++(-.1,0.5) node[above] {${\scriptstyle{\lambda_m> 0}}$};
		\draw[fill=white] (-0.3,0.5) rectangle (6.3,1);
	\end{tikzpicture}}}
\end{array}	
\]
which projects the cyclic stairway basis elements of \cref{lemm:Dpois!} and \cref{Lem:uDPois!} onto the ones with positive numbers of sub-sets of inputs.
\end{proof}

From curved pre-Calabi--Yau algebras to pre-Calabi--Yau algebras, we have been modifying by hands the governing  quasi-free dg properad. The natural question that arises then is: "the notion of a pre-Calabi--Yau algebra is the homotopy version of which algebraic structure?". In this direction, a first question to answer is: the codioperad $\rmC_{\rm pCY}$ is the Koszul dual of which properad, or equivalently, does the properad
$\frac{\mathrm{u}\DPois^!}{\field u}=\rmC_{\rm pCY}^*$ admit a quadratic presentation? 
Regarding the stairway basis elements of the properad $\mathrm{u}\DPois^!$ given in \cref{lemm:Dpois!} and \cref{Lem:uDPois!}, one can see that the two elements 
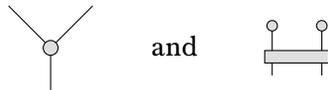
\[			
\vcenter{\hbox{\begin{tikzpicture}[scale=0.28]
				\draw (0,4) 
				-- (2,2);
				\draw (4,4) 
				-- (2,2) -- (2,0);
				\draw[fill=gray!25] (2,2) circle (10pt);
			\end{tikzpicture}}}
 \qquad \text{and} \qquad
\vcenter{\hbox{ \begin{tikzpicture}[scale=0.33]
				\draw (2,0)  -- (2,2);
				\draw (0,0) -- (0,2);
				\draw[fill=gray!25] (0,2) circle (6pt);
				\draw[fill=gray!25] (2,2) circle (6pt);				
				\draw[fill=gray!25] (-0.3,0.5) rectangle (2.3,1);
			\end{tikzpicture}}}
\]
of respective degrees $-1$ and $+1$ form a set of generators. They lead to the following notion.

%%%%%%%%%%%%%%%%

\begin{definition}[{$\rmV$-gebra \cite{TZ07Bis}}]\label{def:Valg}
A \emph{$\rmV$-gebra} is an associative algebra $(A, \mu)$ equipped with a degree $-2$ symmetric bi-tensor $b\in \left(A^{\otimes 2}\right)^{\Sy_2}$ satisfying 
\[\sum \mu(a, b')\otimes b''=\sum b'\otimes \mu(b'', a)~,\]
for any $a\in A$, where $b=\sum b'\otimes b''$~.
\end{definition}	
	
\begin{remark}
This notion appears in the works of P. Seidel in symplectic geometry \cite[Section~3.3]{Seidel12} under the name \emph{associative algebra with boundary}.
\end{remark}

The category of $\rmV$-gebras is encoded by the quadratic properad 
	\[
		\rmV \coloneq
		\dfrac{
			\scrG\left(
			\begin{tikzpicture}[baseline=0ex,scale=0.10]
				\draw (0,4) node[above] {$\scriptscriptstyle{1}$}
				-- (2,2);
				\draw (4,4) node[above] {$\scriptscriptstyle{2}$}
				-- (2,2) -- (2,0) node[below] {$\scriptscriptstyle{1}$};
				\draw[fill=white] (2,2) circle (10pt);
			\end{tikzpicture}
			~ ; ~
			\begin{tikzpicture}[baseline=0ex,scale=0.2]
				\draw (0,0)  node[below] {$\scriptscriptstyle{1}$} -- (0,0.5);
				\draw (2,0) node[below] {$\scriptscriptstyle{2}$} -- (2,0.5);
				\draw[fill=white] (-0.3,0.5) rectangle (2.3,1);
			\end{tikzpicture}
			\ = \ 
			\begin{tikzpicture}[baseline=0ex,scale=0.2]
				\draw (0,0)  node[below] {$\scriptscriptstyle{2}$} -- (0,0.5);
				\draw (2,0) node[below] {$\scriptscriptstyle{1}$} -- (2,0.5);
				\draw[fill=white] (-0.3,0.5) rectangle (2.3,1);
			\end{tikzpicture}
			\right)
		}{
			\left(
			\begin{tikzpicture}[baseline=0.5ex,scale=0.1]
				\draw (0,4) node[above] {$\scriptscriptstyle{1}$} --(4,0)
				-- (4,-1) node[below] {$\scriptscriptstyle{1}$};
				\draw (4,4) node[above] {$\scriptscriptstyle{2}$} -- (2,2);
				\draw (8,4) node[above] {$\scriptscriptstyle{3}$}-- (4,0);
				\draw[fill=white] (2,2) circle (10pt);
				\draw[fill=white] (4,0) circle (10pt);
			\end{tikzpicture}
			-
			\begin{tikzpicture}[baseline=0.5ex,scale=0.1]
				\draw (0,4) node[above] {$\scriptscriptstyle{1}$}
				-- (4,0) -- (4,-1) node[below] {$\scriptscriptstyle{1}$};
				\draw (4,4) node[above] {$\scriptscriptstyle{2}$} -- (6,2);
				\draw (8,4) node[above] {$\scriptscriptstyle{3}$} -- (4,0);
				\draw[fill=white] (6,2) circle (10pt);
				\draw[fill=white] (4,0) circle (10pt);
			\end{tikzpicture}
			~~;~~
			\begin{tikzpicture}[scale=0.2,baseline=1ex]
				\draw (1,0.5) -- (2,1.5) -- (1,0.5) -- (0,1.5) -- (0, 2.5) -- (0, 1.5)-- (1,0.5)
				-- (1,-0.5) node[below] {$\scriptscriptstyle{1}$};
				\draw (0,2.5) node[above] {$\scriptscriptstyle{1}$};
				\draw (4,-0.5) node[below] {$\scriptscriptstyle{2}$} -- (4,1.5) ;
				\draw[fill=white] (1,0.5) circle (6pt);
				\draw[fill=white] (1.7,1.5) rectangle (4.3,2);
			\end{tikzpicture}
			-
			\begin{tikzpicture}[scale=0.2,baseline=1ex]
				\draw (1,0.5) -- (2,1.5) -- (1,0.5) -- (0,1.5) -- (1,0.5)
				-- (1,-0.5) node[below] {$\scriptscriptstyle{2}$};
				\draw (2,2.5) node[above] {$\scriptscriptstyle{1}$} -- (2,1.5)-- (1,0.5);
				\draw (-2,-0.5)  node[below] {$\scriptscriptstyle{1}$}--  (-2,1.5);
				\draw[fill=white] (1,0.5) circle (6pt);
				\draw[fill=white] (-2.3,1.5) rectangle (0.3,2);
			\end{tikzpicture}
			\right)
		}~,
	\]
	where the first generator has degree $0$ and the second one has degree $-2$~.

\begin{proposition}[\cite{PT19}]\label{prop:VinftypCY}
The genus 0 part of the Koszul dual coproperad of $\rmV$ is isomorphic to the codioperad encoding pre-Calabi--Yau algebras:
\[\rmV^\antish_{\rm{g}=0} \cong \rmC_{\rm pCY}~. \]
\end{proposition}

\begin{proof}
Though this statement can essential be found in \cite[Proposition~3.4]{PT19}, we give it a new proof here 
since the methods applied in \cref{subsec::DPois_antish} for the quadratic properad $\DPois$ apply \emph{mutatis mutandis} to the quadratic properad $\rmV$, 
since it establishes a relationship with homotopy (curved) double Poisson gebras, and 
since it
 unravels the new higher genera part. \cref{lemm:KDproperad} shows that the Koszul dual properad $\rmV^!$ is given by 
	\[
		\rmV^! \cong
		\dfrac{
			\scrG\left(
			\begin{tikzpicture}[baseline=0ex,scale=0.10]
				\draw (0,4) node[above] {$\scriptscriptstyle{1}$}
				-- (2,2);
				\draw (4,4) node[above] {$\scriptscriptstyle{2}$}
				-- (2,2) -- (2,0) node[below] {$\scriptscriptstyle{1}$};
				\draw[fill=gray!25] (2,2) circle (10pt);
			\end{tikzpicture}
			~ ; ~
			\begin{tikzpicture}[baseline=0ex,scale=0.2]
				\draw (0,0)  node[below] {$\scriptscriptstyle{1}$} -- (0,0.5);
				\draw (2,0) node[below] {$\scriptscriptstyle{2}$} -- (2,0.5);
				\draw[fill=gray!25] (-0.3,0.5) rectangle (2.3,1);
			\end{tikzpicture}
			\ = \ 
			\begin{tikzpicture}[baseline=0ex,scale=0.2]
				\draw (0,0)  node[below] {$\scriptscriptstyle{2}$} -- (0,0.5);
				\draw (2,0) node[below] {$\scriptscriptstyle{1}$} -- (2,0.5);
				\draw[fill=gray!25] (-0.3,0.5) rectangle (2.3,1);
			\end{tikzpicture}
			\right)
		}{
			\left(
			\begin{tikzpicture}[baseline=0.5ex,scale=0.1]
				\draw (0,4) node[above] {$\scriptscriptstyle{1}$} --(4,0)
				-- (4,-1) node[below] {$\scriptscriptstyle{1}$};
				\draw (4,4) node[above] {$\scriptscriptstyle{2}$} -- (2,2);
				\draw (8,4) node[above] {$\scriptscriptstyle{3}$}-- (4,0);
				\draw[fill=gray!25] (2,2) circle (10pt);
				\draw[fill=gray!25] (4,0) circle (10pt);
			\end{tikzpicture}
			+
			\begin{tikzpicture}[baseline=0.5ex,scale=0.1]
				\draw (0,4) node[above] {$\scriptscriptstyle{1}$}
				-- (4,0) -- (4,-1) node[below] {$\scriptscriptstyle{1}$};
				\draw (4,4) node[above] {$\scriptscriptstyle{2}$} -- (6,2);
				\draw (8,4) node[above] {$\scriptscriptstyle{3}$} -- (4,0);
				\draw[fill=gray!25] (6,2) circle (10pt);
				\draw[fill=gray!25] (4,0) circle (10pt);
			\end{tikzpicture}
			~~;~~
			\begin{tikzpicture}[scale=0.2,baseline=1ex]
				\draw (1,0.5) -- (2,1.5) -- (1,0.5) -- (0,1.5) -- (0, 2.5) -- (0, 1.5)-- (1,0.5)
				-- (1,-0.5) node[below] {$\scriptscriptstyle{1}$};
				\draw (0,2.5) node[above] {$\scriptscriptstyle{1}$};
				\draw (4,-0.5) node[below] {$\scriptscriptstyle{2}$} -- (4,1.5) ;
				\draw[fill=gray!25] (1,0.5) circle (6pt);
				\draw[fill=gray!25] (1.7,1.5) rectangle (4.3,2);
			\end{tikzpicture}
			+
			\begin{tikzpicture}[scale=0.2,baseline=1ex]
				\draw (1,0.5) -- (2,1.5) -- (1,0.5) -- (0,1.5) -- (1,0.5)
				-- (1,-0.5) node[below] {$\scriptscriptstyle{2}$};
				\draw (2,2.5) node[above] {$\scriptscriptstyle{1}$} -- (2,1.5)-- (1,0.5);
				\draw (-2,-0.5)  node[below] {$\scriptscriptstyle{1}$}--  (-2,1.5);
				\draw[fill=gray!25] (1,0.5) circle (6pt);
				\draw[fill=gray!25] (-2.3,1.5) rectangle (0.3,2);
			\end{tikzpicture}
			~~;~~
					\begin{tikzpicture}[scale=0.2,baseline=-0.8ex)]
			\draw (-2,0.5) 			to[out=270, in=140] (-1,-1.5)
			-- (-1,-2) node[below] {$\scriptscriptstyle{1}$}
			-- (-1,-1.5) to[out=40, in=270] (0,0)
			-- (0,0.5);
			\draw[fill=gray!25] (-2.3,0) rectangle (0.3,0.5)
			(-1,-1.5) circle (5pt);
		\end{tikzpicture}
			\right)
		}~,
	\]
where the first generator has degree $-1$ and the second one has degree $1$~.
Any element of the free properad on these two generators is a linear combination of forests of planar binary rooted trees composed at the top with some symmetric bi-tensors. When a symmetric bi-tensor connects two planar binary rooted trees, one can use the anti-associativity relation to pull up top vertices attached to it, and then use  the second relation as follows to pass all the binary vertices from one tree to the other. 
\[
\vcenter{\hbox{
\begin{tikzpicture}[scale=0.4]
	\draw (1,0.5) -- (2,1.5) -- (1,0.5) -- (0,1.5) -- (1,0.5) -- (1,-0.5) ;
	\draw (4,1.5) -- (5,0.5) -- (6, 1.5) -- (5,0.5) -- (5, -0.5);	
	\draw[fill=gray!25] (1,0.5) circle (6pt);
	\draw[fill=gray!25] (5,0.5) circle (6pt);	
	\draw[fill=gray!25] (1.7,1.5) rectangle (4.3,2);
\end{tikzpicture}}}
			~=~-~
\vcenter{\hbox{\begin{tikzpicture}[scale=0.4]
	\draw (1,0.5) -- (2,1.5) -- (1,0.5) -- (0,1.5) -- (1,0.5) -- (1,-0.5) ;
	\draw (2,1.5) -- (1,2.5) ;	
	\draw (2,1.5) -- (3,2.5) ;	
	\draw (5,2.5) -- (5,1.5) ;			
	\draw[fill=gray!25] (1,0.5) circle (6pt);
	\draw[fill=gray!25] (2,1.5) circle (6pt);	
	\draw[fill=gray!25] (2.7,2.5) rectangle (5.3,3);
\end{tikzpicture}}}
\]
In the end, one gets just one planar binary rooted tree modulo the anti-associativity relation with some symmetric bi-tensors plugged at its inputs. 
Unlike in \cref{lemm:Dpois!}, we choose here left combs of planar binary trees as normal forms. \emph{A priori}, there are three possibilities: a symmetric bi-tensor can be plugged to such a planar tree at one output, at two consecutive inputs, or at two non-consecutive inputs. Having all symmetric bi-tensors plugged at only one output produces the basis of the genus 0 part $\rmV^\antish_{\rm{g}=0}$ of the Koszul dual properad, i.e. the Koszul dual dioperad, see \cref{eqn:IsoBases}. 
The second case turns out to be impossible, i.e. is equal to zero: using the anti-associativity relation, one can pull up one binary product with one symmetric bi-tensor grafted at its two inputs, but this is the third relation of $\rmV^!$~. The third case produces non-trivial higher genus composites whose combinatorics will be addressed elsewhere.

\medskip

We consider the morphism of properads $\Phi \colon \rmV^! \to \frac{\mathrm{u}\DPois^!}{\field u}$ defined by 
\[
\vcenter{\hbox{\begin{tikzpicture}[scale=0.2]
				\draw (0,4) 
				-- (2,2);
				\draw (4,4) 
				-- (2,2) -- (2,0);
				\draw[fill=gray!25] (2,2) circle (10pt);
			\end{tikzpicture}}}
\ \mapsto \ 
\vcenter{\hbox{\begin{tikzpicture}[scale=0.2]
				\draw (0,4) 
				-- (2,2);
				\draw (4,4) 
				-- (2,2) -- (2,0);
				\draw[fill=gray!25] (2,2) circle (10pt);
			\end{tikzpicture}}} 
\qquad \text{and} \qquad 
\vcenter{\hbox{ \begin{tikzpicture}[scale=0.33]
				\draw (2,0)  -- (2,0.5);
				\draw (0,0) -- (0,0.5);
				\draw[fill=gray!25] (-0.3,0.5) rectangle (2.3,1);
			\end{tikzpicture}}}
\ \mapsto \ 
\vcenter{\hbox{ \begin{tikzpicture}[scale=0.33]
				\draw (2,0)  -- (2,2);
				\draw (0,0) -- (0,2);
				\draw[fill=gray!25] (0,2) circle (6pt);
				\draw[fill=gray!25] (2,2) circle (6pt);				
				\draw[fill=gray!25] (-0.3,0.5) rectangle (2.3,1);
			\end{tikzpicture}}}
			\ ,
\]
where it is straightforward to check that this  assignment sends the relations of the properad $\rmV^!$ to $0$ in 
$\frac{\mathrm{u}\DPois^!}{\field u}$~. 
This morphism sends the abovementioned basis of the genus $0$ part of $\rmV^!$ to the cyclic stairways basis of 
$\frac{\mathrm{u}\DPois^!}{\field u}$, see \cref{lemm:Dpois!} and \cref{Lem:uDPois!}:
\begin{equation}\label{eqn:IsoBases}
\Phi \ \colon \ 
\vcenter{\hbox{
			\begin{tikzpicture}[scale=0.5]
			\coordinate (A) at (0,0);
			\draw (A) --++ (0,-0.5) node[below] {\tiny{$j_1$}};
			\draw (A) --++ (45:1) ++ (45:0.4) ++ (0:0.2) node {\tiny{$\ i_{\lambda_1}$}} ;
			\draw (A) --++ (135:1) coordinate (B);
			\draw[fill=gray!25] (A) circle (4pt);
			\draw (B) --++ (45:1) ++ (45:0.4) ++(0:0.4) node {\tiny{$\ i_{\lambda_1-1}$}} ;
			\draw (B) --++ (135:0.5);	 
			\draw[fill=gray!25] (B) circle (4pt);	 
			\draw[dotted] (B) ++ (135:0.5) --++(135:1);
			\draw (B) ++(135:1.5) --++ (135:0.5) coordinate (C);
			\draw (C) --++ (45:1) ++ (45:0.4) ++ (0:0.1) node {\tiny{$i_{1}$}} ;
			\draw (C)  --++ (135:1.5) coordinate (D);	 
			\draw[fill=gray!25] (C) circle (4pt);	 
			\draw (D) --++ (45:1) --++(0,0.3) coordinate (B1); 
			\draw (D) --++ (135:1.5) coordinate (E);
			\draw (E) --++ (135:0.5);
			\draw[fill=gray!25] (D) circle (4pt);	 
			\draw (B1) ++ (-0.2,0.3) coordinate (B12);
			\draw (B1) ++ (+0.8,0) --++(0,-0.3);
			\draw (B1) ++ (0.8,0) ++(0,-0.1) node[below] {\tiny{$j_2$}};			
			\draw[fill=gray!25] (B1) ++(1,0) rectangle (B12);	
			\draw (E) --++ (45:1) ++ (45:0.4) ++ (0:0.6) node {\tiny{$\ i_{\lambda_1+\lambda_2}$}}; 
			\draw[fill=gray!25] (E) circle (4pt);	 
			\draw[dotted] (E) ++ (135:.5) --++ (135:1);	 
			\draw (E) ++ (135:1.5) --++ (135:0.5) coordinate (F);	 
			\draw (F) --++ (45:1) ++ (45:0.4) ++ (0:0.1) node 
			{\tiny{$\qquad \qquad \qquad i_{\lambda_1+\cdots+\lambda_{m-2}+1}$}};
			\draw  (F) --++ (135:1.5) coordinate (G);	 
			\draw[fill=gray!25] (F) circle (4pt);	 
			\draw (G) --++ (45:1) --++ (0,0.3) coordinate (B2);
			\draw (G) --++ (135:1.5) coordinate (H);
			\draw[fill=gray!25] (G) circle (4pt);
			\draw (B2) ++ (-0.2,0.3) coordinate (B22);
			\draw (B2) ++ (0.8,0) --++(0,-0.3);
			\draw (B2) ++ (0.8,0) --++(0,-0.1) node[below] {\tiny{$j_m$}};			
			\draw[fill=gray!25] (B2) ++(1,0) rectangle (B22);	 
			\draw (H) --++ (45:1) ++ (45:0.4) ++ (0:1) node
			{\tiny{$i_{\lambda_1+\cdots + \lambda_{m}}$}} ;
			\draw (H) --++(135:0.5);
			\draw[fill=gray!25] (H) circle (4pt);	
			\draw[dotted] (H) ++ (135:.5) --++ (135:1);	 
			\draw (H) ++ (135:1.5) --++ (135:0.5) coordinate (I);	  
			\draw (I) --++(45:1); 
			\draw (I) --++(135:1); 			
			\draw (I) ++ (135:1.5) ++ (180:0.5) node {\tiny{$i_{\lambda_1+\cdots + \lambda_{m-1}+1}$}}; 
			\draw (I) ++(45:1.4) ++(0:1.5) node {\tiny{$i_{\lambda_1+\cdots + \lambda_{m-1}+2}$}};
			\draw[fill=gray!25] (I) circle (4pt);	
		\end{tikzpicture}}}
\ \mapsto \ 	\pm		\
		\vcenter{\hbox{\begin{tikzpicture}[scale=0.4]
			\coordinate (A) at (0,1);
			\draw[thin]
			(A) ++(0,-1) node[below] {$\scriptstyle{j_1}$} --
		(A) -- ++(0,0.5)
		(A) -- ++(-.3,0.5)
		(A) -- ++(.3,0.5)
			(A)++(0,.5) node[above] {${\scriptstyle{\Lambda_1}}$};
			\coordinate (A) at (1.5,1);
			\draw[thin]
			(A) ++(0,-1) node[below] {$\scriptstyle{j_2}$} --
		(A) -- ++(-.3,0.5)
		(A) -- ++(.3,0.5)
			(A)++(0,.5) node[above] {${\scriptstyle{\Lambda_2}}$};
			\draw (3.75,0) node[below] {$\scriptstyle{\cdots}$};
			\draw (3.75,1.5) node[above] {$\scriptstyle{\cdots}$};
			\coordinate (A) at (6,1);
			\draw[thin]
			(A) ++(0,-1) node[below] {$\scriptstyle{j_m}$} --
		(A) -- ++(-.5,0.5)
		(A) -- ++(-.2,0.5)
		(A) -- ++(.2,0.5)
		(A) -- ++(.5,0.5)		
			(A)++(0,.5) node[above] {${\scriptstyle{\Lambda_m}}$};
			\draw[fill=gray!25] (-0.3,0.5) rectangle (6.3,1);
		\end{tikzpicture}}}	
~,
\end{equation}
under the notations of \cref{lemm:Dpois!}. 
Notice that the range of indices is the same on both sides: for $m\geqslant 2$, the non-negative integers $\lambda_1, \ldots, \lambda_m$ form an ordered partition $\lambda_1+\cdots+\lambda_m=n$ of the number $n\geqslant 0$ of inputs and we have $\lambda_1\geqslant 1$,  for $m=1$~. 
This shows that the restriction of the morphism $\Phi$ to the genus $0$ part of the properad $\rmV^!$ is an isomorphism of dioperads. Its dual produces the requested isomorphism of codioperads: 
\[\rmV^\antish_{\rm{g}=0} \cong \rmC_{\rm pCY}~. \]
\end{proof}

\begin{definition}[{Genus 0 $\rmV_\infty$-gebra \cite{TZ07Bis}}]\label{def:Vinfty}
A \emph{genus 0 $\rmV_\infty$-gebra} is a gebra over the cobar construction of the genus 0 part of the Koszul dual codioperad $\rmV^\antish_{\rm{g}=0}$~.
\end{definition}	

\begin{remark}
This notion appears in the works of P. Seidel in symplectic geometry \cite[Definition~3.1]{Seidel12} under the name \emph{$\rmA_\infty$-algebra with boundary}.
\end{remark}

\begin{corollary}
The category of a pre-Calabi--Yau algebras is equivalent to the category of a genus 0 
$\rmV_\infty$-algebras. 
\end{corollary}

\begin{proof}
This is a direct corollary of \cref{prop:VinftypCY}.
\end{proof}

Let us summarise the various algebraic structures that we have encountered throughout this section.
\[\boxed{ 
\text{homotopy double Poisson gebras} \hookrightarrow  \text{pre-Calabi--Yau algebras}  
 \cong  \text{genus 0}\  \rmV_\infty\text{-gebras}
}\]

Regarding their homotopy properties, K. Poirier and T. Tradler have established the following homological result. 

\begin{theorem}[{\cite[Theorem~3.5]{PT19}}]
The dioperad $\rmV_{{\rm g}=0}$ is Koszul. 
\end{theorem}

In other words, this statement asserts  that the dioperadic cobar construction, that is the genus 0 cobar construction, of the codioperad $\rmC_{\rm pCY}$ is quasi-isomorphic to the dioperad $\rmV_{\rm{g}=0}$. However, the computation of the Koszul dual coproperad $\rmV^{\ac}$ performed in the proof of \cref{{prop:VinftypCY}} shows the presence of non-trivial elements of higher genera; this implies that the properadic cobar construction of the codioperad $\rmC_{\rm pCY}$ fails to be 
quasi-isomorphic to the properad $\rmV$~. For instance, the cycle 
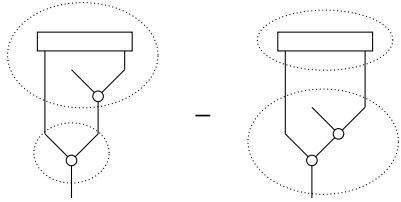
\[
\begin{tikzpicture}[scale=0.5,baseline=(n.base)]
	\coordinate (n) at (0,1);
	\coordinate (A) at (0,0);
	\draw (A) ++ (0,-1) -- (A) --++ (45:1) coordinate (AS1)	++(45:-1) --++(135:1) coordinate (AS2);
	\draw[fill=white] (A) circle (4pt);
	\draw (AS1) --++(0,1) coordinate (B);	 
	\draw (B) --++ (45:1) coordinate (BS1)	++(45:-1) --++(135:1) coordinate (BS2);
	\draw[fill=white] (B) circle (4pt);
	\draw (BS1)--++ (0,0.5)  coordinate (T) ++ (0.2,0.5) coordinate (BxR);
	\coordinate (Temp) at (AS2|-T);
	\draw (AS2) -- (Temp) ++(-0.2,0) coordinate (BxL);
	\draw[fill=white] (BxL) rectangle (BxR); 
	\draw[densely dotted] (0.3,2.8) circle[x radius = 2, y radius = 1.4];
	\draw[densely dotted] (0,0.2) circle[x radius = 1, y radius = 0.8];
\end{tikzpicture}
\quad - \quad 
\begin{tikzpicture}[scale=0.5,baseline=(n.base)]
	\coordinate (n) at (0,1);
	\coordinate (A) at (0,0);
	\draw (A) ++ (0,-1) -- (A) --++ (45:1) coordinate (AS1)	++(45:-1) --++(135:1) coordinate (AS2);
	\draw[fill=white] (A) circle (4pt);
	\draw (AS1) --++(0,0) coordinate (B);	 
	\draw (B) --++ (45:1) coordinate (BS1)	++(45:-1) --++(135:1) coordinate (BS2);
	\draw[fill=white] (B) circle (4pt);
	\draw (BS1)--++ (0,1.5)  coordinate (T) ++ (0.2,0.5) coordinate (BxR);
	\coordinate (Temp) at (AS2|-T);
	\draw (AS2) -- (Temp) ++(-0.2,0) coordinate (BxL);
	\draw[fill=white] (BxL) rectangle (BxR); 
	\draw[densely dotted] (0.3,0.5) circle[x radius = 2, y radius = 1.4];
	\draw[densely dotted] (0.35,3.2) circle[x radius = 1.8, y radius = 0.8];
\end{tikzpicture}		\]
produces a non-trivial homology class in the properadic cobar construction of the codioperad
$\rmC_{\rm pCY}$. But it is equal to the boundary of the following element of the Koszul dual coproperad $\rmV^{\ac}$ 
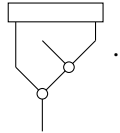
\[
	\begin{tikzpicture}[scale=0.5,baseline=(n.base)]
		\coordinate (n) at (0,1);
		\coordinate (A) at (0,0);
		\draw (A) ++ (0,-1) -- (A) --++ (45:1) coordinate (AS1)	++(45:-1) --++(135:1) coordinate (AS2);
		\draw[fill=white] (A) circle (4pt);
		\draw (AS1)  coordinate (B);	 
		\draw (B) --++ (45:1) coordinate (BS1)	++(45:-1) --++(135:1) coordinate (BS2);
		\draw[fill=white] (B) circle (4pt);
		\draw (BS1)--++ (0,0.5)  coordinate (T) ++ (0.2,0.5) coordinate (BxR);
		\coordinate (Temp) at (AS2|-T);
		\draw (AS2) -- (Temp) ++(-0.2,0) coordinate (BxL);
		\draw[fill=white] (BxL) rectangle (BxR); 
	\end{tikzpicture}	
~.\]
in the properadic cobar construction. 

\medskip 

This raises the questions of the Koszul property of the properad $\rmV$ and the computation of the homology of the properadic cobar construction of the codioperad $\rmC_{\rm pCY}$. Since this goes outside the scope of the present paper, we will address these interesting points elsewhere. 

\subsection{Combinatorial description of the decomposition maps} \label{subsec:CoproduitsCpCY}
Recall that, in the properadic calculus \cite{HLV20}, every homotopical construction relies on one of  the various decomposition maps of the (Koszul dual) coproperad.
The  relation between the different coproperads encountered in the previous sections are as follows: 
\[
\rmV^{\ac} \twoheadrightarrow \rmC_{\rm pCY} \twoheadrightarrow \DPois^{\ac}~.
\]
This implies the following morphisms of properads  
\[\rmV_\infty \twoheadrightarrow {\rm pCY}\twoheadrightarrow \DPois_\infty~,\]
which show that is any homotopy double Poisson gebra carries a canonical pre-Calabi--Yau algebra structure  and  that any pre-Calabi--Yau algebra gives the genus 0 part of a $\rmV_\infty$-gebra. More important, since these functors of categories of gebras come from morphisms of coproperads, they preserves \emph{automatically} all the various homotopical constructions. \\

Since the infinitesimal coproperad $\rmC_{\rm pCY}$ 
(respectively $\DPois^{\ac}$ and $\rmV^{\ac}$) 
encoding pre-Calabi--Yau algebras 
(respectively homotopy double Poisson gebras and $\rmV_\infty$-gebras) 
is conilpotent, it carries 
all the possible properadic decomposition maps 
(infinitesimal $\Delta_{(1,1)}$, full $\Delta$, left and right infinitesimal 
$\Cop{}{(*)}$  and $\Cop{(*)}{}$, comonadic $\widetilde{\Delta}$) 
 which will be the key elements in order to describe the various homotopical properties of pre-Calabi--Yau algebras 
 (respectively homotopy double Poisson gebras and $\rmV_\infty$-gebras) in the next section. In this section, we provide the reader with combinatorial descriptions of all these decomposition maps in the case of the codioperad $\rmC_{\rm pCY}$. The case of the codioperad 
$\DPois^{\ac}$ is a sub-case obtained by restricting the sets of indices of the inputs; 
 the case of the coproperad $\rmV^{\ac}$ requires more details and will be treated elsewhere. 

\begin{definition}[Roundabout]
A \emph{roundabout} is a disc with outgoing edges labelled bijectively from $1$ to $m$ and ingoing edges labelled bijectively from $1$ to $n$, see 
\cref{Fig:Roondabout}.  
\end{definition}

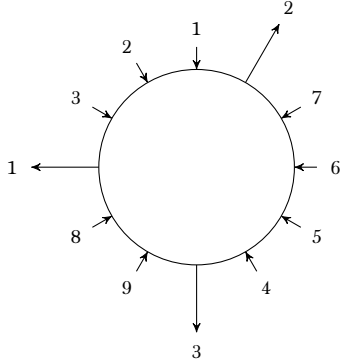
\begin{figure}[h!]
	\tikzmath{
		\radIn = 1.3; % rayon disque
		\radMid = 1.6; % extremite fleche entrante
		\radOut = 2.2; % extremite fleche sortante
		\dec = 0.25; % decalage indice
		}
\[
	\begin{tikzpicture}
		\coordinate (C) at (0,0);
		\draw[fill=white] (C) circle (\radIn);
		\foreach \angle/\i in {180/1,60/2,270/3}{
			\draw[->] (\angle:\radIn) -- (\angle:\radOut);
			\draw (\angle:\radOut) ++ (\angle:\dec) node {\tiny{$\i$}};
		}
		\foreach \angle/\i in {210/8,240/9,300/4,330/5,0/6,30/7,90/1,120/2,150/3}{
			\draw[<-] (\angle:\radIn) -- (\angle:\radMid);
			\draw (\angle:\radMid) ++ (\angle:\dec) node {\tiny{$\i$}};
		}
	\end{tikzpicture}
\]
	\caption{A roundabout with 3 inputs and 9 outputs.}
	\label{Fig:Roondabout}
\end{figure}

All roundabouts have at least one output. In the case of the codioperad $\rmC_{\rm pCY}$, there is no restriction on the number of inputs, except for roundabouts with only one output which are required to have at least one input. In the case of the codioperad $\DPois^{\ac}$, there should always be at least one input between two consecutive outputs. 

\begin{lemma}\label{lem:RoundaBasis}
The set of roundabouts forms a basis of the underlying $\Sy$-bimodule of the codioperad $\rmC_{\rm pCY}$~. 
\end{lemma}

\begin{proof}
\cref{lemm:Dpois!} and \cref{Lem:uDPois!} provided us with cyclic stairway basis elements  for the codioperad $\rmC_{\rm pCY}$: 
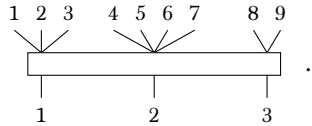
\[	
\vcenter{\hbox{\begin{tikzpicture}[scale=0.6]
		\coordinate (A) at (0,1);
		\draw[thin]
		(A) ++(0,-1) node[below] {${\scriptstyle{1}}$} --
		(A) -- ++(-0.6,0.5) node[above] {${\scriptstyle{1}}$}
		(A) -- ++(0,0.5) node[above] {${\scriptstyle{2}}$}
		(A) -- ++(.6,0.5) node[above] {${\scriptstyle{3}}$};
		\coordinate (A) at (2.5,1);
		\draw[thin]
		(A) ++(0,-1) node[below] {$\scriptstyle{2}$} --
		(A) -- ++(-.9,0.5) node[above] {${\scriptstyle{4}}$}
		(A) -- ++(-.3,0.5) node[above] {${\scriptstyle{5}}$}
		(A) -- ++(.3,0.5) node[above] {${\scriptstyle{6}}$}
		(A) -- ++(.9,0.5)	node[above] {${\scriptstyle{7}}$};	
		\coordinate (A) at (5,1);
		\draw[thin]
		(A) ++(0,-1) node[below] {$\scriptstyle{3}$} --
		(A) -- ++(-0.3,0.5) node[above] {${\scriptstyle{8}}$}
		(A) -- ++(0.3,0.5) node[above] {${\scriptstyle{9}}$};
		\draw[fill=white] (-0.3,0.5) rectangle (5.3,1);
	\end{tikzpicture}}}~. 
\]
They are in one-to-one correspondence with roundabouts: given a cyclic stairway, we start drawing a roundabout with a first outgoing edge labelled by the label of the first output of the stairway, we proceed clockwise with incoming edges labelled 
anti-clockwise with the first set of inputs of the stairways, etc. 
For instance, the abovementioned cyclic stairway corresponds to the roundabout of \cref{Fig:Roondabout} under this bijection. 
\end{proof}

\begin{definition}[Oriented arc]
An \emph{oriented arc} of a roundabout is an arc of the underlying disc that starts and ends on the boundary circle where there is no edge attached and equipped with an edge pointing in the direction of one of the two areas created by the arc. 
\end{definition}

\begin{definition}[Elementary coloured cutting]
An \emph{elementary coloured cutting} of a roundabout is a partition of a roundabout into two parts, black and white,  delimited by an oriented arc, see \cref{Fig:CuttingsRounda}; the orientation of the edge of the arc leaves the white part and goes into the black part. 
The white part (respectively the black part) is required to contain at least one edge (respectively one outgoing edge) of the roundabout. 
\end{definition}

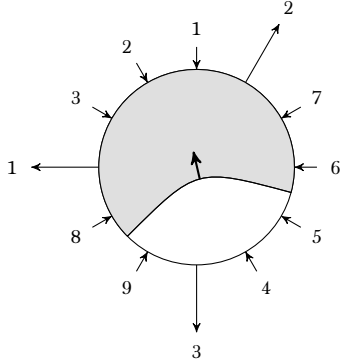
\begin{figure}[h!]
	\tikzmath{
		\radIn = 1.3; % rayon disque
		\radMid = 1.6; % extremite fleche entrante
		\radOut = 2.2; % extremite fleche sortante
		\dec = 0.25; % decalage indice
		}
\[
	\begin{tikzpicture}
		\coordinate (C) at (0,0);
		\draw[fill=white] (C) circle (\radIn);   
		\foreach \angle/\i in {180/1,60/2,270/3}{
			\draw[->] (\angle:\radIn) -- (\angle:\radOut);
			\draw (\angle:\radOut) ++ (\angle:\dec) node {\tiny{$\i$}};
		}
		\foreach \angle/\i in {210/8,240/9,300/4,330/5,0/6,30/7,90/1,120/2,150/3}{
			\draw[<-] (\angle:\radIn) -- (\angle:\radMid);
			\draw (\angle:\radMid) ++ (\angle:\dec) node {\tiny{$\i$}};
		}
		\draw[fill=gray!25] (225:\radIn) .. controls (0,0) .. (-15:\radIn) arc (-15:225:\radIn);
		\draw (225:\radIn) .. controls (0,0) .. (-15:\radIn) node[midway] (A) {};
		\draw[->,thick] (A.center) --++ (103:{1.5*\dec});
	\end{tikzpicture}
\]
	\caption{An elementary cutting of a roundabout.}
	\label{Fig:CuttingsRounda}
\end{figure}

\begin{proposition}\label{prop:CuttingInfDec}
The  terms appearing in the infinitesimal decomposition map 
\[\Delta_{(1,1)} \colon \rmC_{\rm pCY} \to \rmC_{\rm pCY} \ibt \rmC_{\rm pCY}\]
 of the codioperad $\rmC_{\rm pCY}$ are in one-to-one correspondence with 
the elementary coloured cuttings of roundabouts. 
\end{proposition}

\begin{proof}
The images of the cyclic stairway basis elements under the infinitesimal decomposition map of the codioperad 
$\rmC_{\rm pCY}$ 
are given by 2-steps stairways of basic cyclic stairways, see 
\cref{prop:InfDecDPoisAC}. 
Under the bijection given in the proof of \cref{lem:RoundaBasis} between cyclic stairways and roundabouts, such decompositions into 2-steps stairways coincide with elementary coloured cuttings: given an elementary coloured cutting, we interpret the white part as a roundabout with a new outgoing edge entering the black part also interpreted as a roundabout located under the white one. Notice that the twisting between $\overline{p}$ and $\overline{q}$ on the stairways is produced by the anti-clockwise labelling of the incoming edges of the roundabouts. 
For instance, the elementary coloured cutting displayed on \cref{Fig:CuttingsRounda} corresponds to the following 2-steps stairway: 
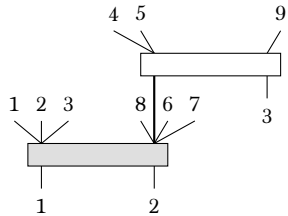
\[
\vcenter{\hbox{
\begin{tikzpicture}[scale=0.6]
		\coordinate (A) at (0,1);
		\draw[thin]
		(A) ++(0,-1) node[below] {${\scriptstyle{1}}$} --
		(A) -- ++(-0.6,0.5) node[above] {${\scriptstyle{1}}$}
		(A) -- ++(0,0.5) node[above] {${\scriptstyle{2}}$}
		(A) -- ++(.6,0.5) node[above] {${\scriptstyle{3}}$};
		\coordinate (A) at (2.5,1);
		\draw[thin]
		(A) ++(0,-1) node[below] {$\scriptstyle{2}$} --
		(A) -- ++(.3,0.5) node[above] {${\scriptstyle{6}}$}
		(A) -- ++(.9,0.5)	node[above] {${\scriptstyle{7}}$};	
		\coordinate (A) at (5,1);
		\draw[fill=gray!25] (-0.3,0.5) rectangle (2.8,1);
		\coordinate (B) at (2.5,1);
		\draw[thin]
		(B) --++(0,1.5) ++ (0,0.5) coordinate (D) ++ (2.5,0) coordinate (E) ++ (0, -0.5) coordinate (F)
		(D) -- ++(-.9,0.5) node[above] {${\scriptstyle{4}}$}
		(D) -- ++(-.3,0.5) node[above] {${\scriptstyle{5}}$}
		(E) -- ++(0.3,0.5) node[above] {${\scriptstyle{9}}$}
		(B) -- ++(-0.3,0.5) node[above] {${\scriptstyle{8}}$}
		(F) --++(0,-0.5)  node[below] {$\scriptstyle{3}$};				
		\draw[thick]
		(B) --++(0,1.5);
		\draw[fill=white] (2.2,2.5) rectangle (5.3,3);		
	\end{tikzpicture}}}
\]
\end{proof}

\begin{definition}[Bicoloured partition]
A \emph{bicoloured partition} of a roundabout is a partition of a roundabout into parts of alternating two colours, black and white: the frontier between a white part and a black part is always delimited by an oriented arc, with the oriented edge leaving the white part and entering into the black part, see \cref{Fig:BiColoredPartRouna}. 
All the incoming (respectively outgoing) edges of the roundabout have to belong to white (respectively black) parts. 
Each white part (respectively  black part) is required to contain at least
two leaving oriented edges (respectively two outgoing edges of the roundabout)
or at least one incoming edge of the roundabout (respectively one incoming oriented edge) . 
\end{definition}

In other words,  the data of a bicoloured partition of a roundabout is equivalent to the data of a set of non-crossing arcs with  edges oriented alternatively. 

\begin{figure}[h!]
	\tikzmath{
		\radIn = 2.4; % rayon disque
		\radMid = 2.8; % extremite fleche entrante
		\radOut = 3.4; % extremite fleche sortante
		\dec = 0.25; % decalage indice
		}
\[	\begin{tikzpicture}
		\coordinate (C) at (0,0);
		\draw[fill=white] (C) circle (\radIn);
		\foreach \angle/\i in {90/1, 0/2, 270/3, 180/4}{
			\draw[->] (\angle:\radIn) -- (\angle:\radOut);
			\draw (\angle:\radOut) ++ (\angle:\dec) node {\tiny{$\i$}};  %%% Rayons sortants
		}
		\foreach \angle/\i in 
		{45/1,-45/2, 240/3, 210/4, 120/5, 150/6
		}{
			\draw[<-] (\angle:\radIn) -- (\angle:\radMid);
			\draw (\angle:\radMid) ++ (\angle:\dec) node {\tiny{$\i$}}; %%% Rayons entrants
		}
		%%%%% GRIS 1
		\draw[fill=gray!25] (105:\radIn) .. controls (41.25:0) .. (-22.5:\radIn) arc (-22.5:105:\radIn);
		\draw (105:\radIn) .. controls (41.25:0) .. (-22.5:\radIn) node[midway] (A) {};
		\draw[->,thick] (A.center) --++ (41.25:{1.5*\dec});
		%%%%% GRIS 2
		\draw[fill=gray!25] (165:\radIn) .. controls (228.75:0.2) .. (292.5:\radIn) arc (292.5:165:\radIn);
		\draw (165:\radIn) .. controls (228.75:0.2) .. (292.5:\radIn) node[midway] (A) {};
		\draw[->,thick] (A.center) --++ (228.75:{1.5*\dec});		
		%%%%% BLANC 1
		\draw[fill=white] (60:\radIn) .. controls (45:1.7) .. (30:\radIn) arc (30:60:\radIn);
		\draw (60:\radIn) .. controls (45:1.7) .. (30:\radIn) node[midway] (B) {};
		\draw[->,thick] (B.center) --++ (225:{1.5*\dec});
		%%%%% BLANC 2
		\draw[fill=white] (200:\radIn) .. controls (225:1.5) .. (250:\radIn) arc (250:200:\radIn);
		\draw (200:\radIn) .. controls (225:1.5) .. (250:\radIn) node[midway] (B) {};
		\draw[->,thick] (B.center) --++ (45:{1.5*\dec});
	\end{tikzpicture} 
\]
	\caption{A bicoloured partition of a roundabout.}
	\label{Fig:BiColoredPartRouna}
\end{figure}
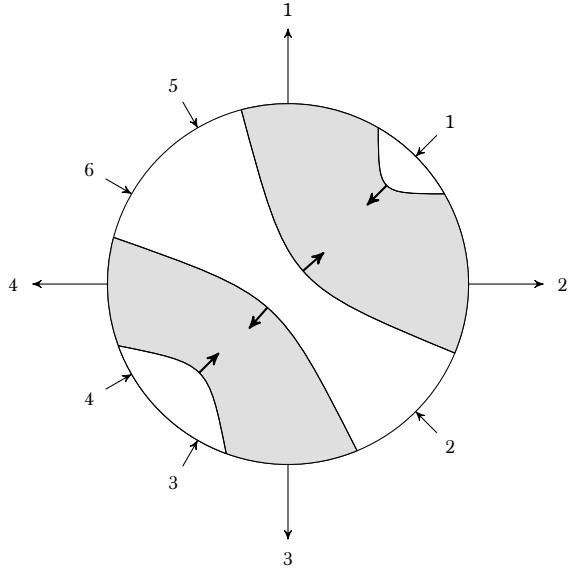

\begin{proposition}\label{prop:BiColPartFullDec}
The  terms appearing in the decomposition map 
\[\Delta \colon \rmC_{\rm pCY} \to \rmC_{\rm pCY} \boxtimes \rmC_{\rm pCY}\]
 of the codioperad $\rmC_{\rm pCY}$ are in one-to-one correspondence with 
the bicoloured partitions of roundabouts.
\end{proposition}

\begin{proof}
Recall that the decomposition map $\Delta \colon \rmC_{\rm pCY} \to \rmC_{\rm pCY} \boxtimes \rmC_{\rm pCY}$ is the linear dual of the composition map 
\[\gamma\  \colon\  \frac{\mathrm{u}\DPois^!}{\field u} \boxtimes \frac{\mathrm{u}\DPois^!}{\field u} \to \frac{\mathrm{u}\DPois^!}{\field u}~.\]
So the terms appearing on the image of a cyclic stairway basis element 
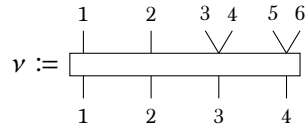
\[	\nu\coloneq 
\vcenter{\hbox{\begin{tikzpicture}[scale=0.6]
		\coordinate (A) at (0,1);
		\draw[thin]
		(A) ++(0,-1) node[below] {${\scriptstyle{1}}$} --
		(A) -- ++(0,0.5) node[above] {${\scriptstyle{1}}$};
		\coordinate (A) at (1.5,1);
		\draw[thin]
		(A) ++(0,-1) node[below] {$\scriptstyle{2}$} --
		(A) -- ++(0,0.5)	node[above] {${\scriptstyle{2}}$};	
		\coordinate (A) at (3,1);
		\draw[thin]
		(A) ++(0,-1) node[below] {$\scriptstyle{3}$} --
		(A) -- ++(-0.3,0.5) node[above] {${\scriptstyle{3}}$}
		(A) -- ++(0.3,0.5) node[above] {${\scriptstyle{4}}$};
		\coordinate (A) at (4.5,1);
		\draw[thin]
		(A) ++(0,-1) node[below] {$\scriptstyle{4}$} --
		(A) -- ++(-0.3,0.5) node[above] {${\scriptstyle{5}}$}
		(A) -- ++(0.3,0.5) node[above] {${\scriptstyle{6}}$};		
		\draw[fill=white] (-0.3,0.5) rectangle (4.8,1);
	\end{tikzpicture}}}
\]
under the decomposition map $\Delta$ are the connected 2-levels graphs of cyclic stairways 
whose image under the composition map $\gamma$ gives $\nu$, like for instance 
\[
\vcenter{\hbox{\begin{tikzpicture}[scale=0.6]
		\draw[thin] (7.5,0) --(7.5,2.5);	 				
		\draw[thick] (7.5,1) --(7.5,2);	 						
		\draw[thick] (5,1) --(5,2);	 						
		\draw[thick] (2.5,1) --(2.5,2);	 						
		\draw[thick] (0,1) --(0,2);	 														
		\draw[thin] (0,0) --(0,3) node[above] {$\scriptstyle{1}$};	
		\draw[thin] (2.5,2.5) --(2.5,3) node[above] {$\scriptstyle{2}$};	
		\draw[thin] (5,2.5) -- (4.7,3) node[above] {${\scriptstyle{5}}$};
		\draw[thin] (5,2.5) -- (5.3,3) node[above] {${\scriptstyle{6}}$};
		\draw[thin] (7.5,2.5) -- (7.2,3) node[above] {${\scriptstyle{3}}$};
		\draw[thin] (7.5,2.5) -- (7.8,3) node[above] {${\scriptstyle{4}}$};		
		\coordinate (A) at (0,1);
		\draw[thin]
		(A) --++(0,-1) node[below] {${\scriptstyle{1}}$};
		\coordinate (A) at (2.5,1);
		\draw[thin]
		(A) --++(0,-1) node[below] {$\scriptstyle{2}$};	
		\coordinate (A) at (5,1);
		\draw[fill=gray!25] (-0.3,0.5) rectangle (2.8,1);
		\coordinate (B) at (2.5,1);
		\draw[thin]
		(B) --++(0,1.5) ++ (0,0.5) coordinate (D) ++ (2.5,0) coordinate (E) ++ (0, -0.5) coordinate (F)
		(F) --++(0,-2.5)  node[below] {$\scriptstyle{4}$};		
		\draw[thin] (7.5, 0.5) --(7.5,0)  node[below] {$\scriptstyle{3}$};						
		\draw[fill=white] (2.2,2) rectangle (5.3,2.5);
		\draw[fill=white] (-0.3,2) rectangle (0.3,2.5);
		\draw[fill=white] (7.2,2) rectangle (7.8,2.5);
		\draw[fill=gray!25] (4.7,0.5) rectangle (7.8,1);		
	\end{tikzpicture}}}~.
\]
Recall from \cref{prop:InfDecDPoisAC} and \cref{Lem:uDPois!}  that all the composites along 2-levels graphs of positive genus vanish in the properad $\frac{\mathrm{u}\DPois^!}{\field u}$, so we only have to consider genus $0$ graphs here. 
Under the bijection between cyclic stairways and roundabouts given in \cref{lem:RoundaBasis}, 
we will show that these connected 2-levels graphs of cyclic stairways are in one-to-one correspondence with bicoloured partitions of roundabouts.
Like in the proof of \cref{prop:CuttingInfDec}, we interpret the white parts of a bicoloured partition of a roundabout 
as cyclic stairways, using \cref{lem:RoundaBasis}, placed at the top level and the black parts as cyclic stairways 
placed at the bottom level; the oriented edges from white parts to black parts provide us with the top-to-bottom edges on the connected 2-levels graph of cyclic stairways. Under this assignment, the example of a bicoloured partition 
 depicted on \cref{Fig:BiColoredPartRouna} gives the above-mentioned example of a connected 2-levels graph. 
 
 \medskip

Let $\nu$ be a cyclic stairway basis element and let $\rho$ be the corresponding roundabout under the bijection given in \cref{lem:RoundaBasis}. 
For any bicoloured partition $\pi$ of the roundabout $\rho$, we consider a total order on the set of arcs such that any arc contained in a roundabout delimited by another arc is less than that latter one. This induces a total order on the inner edges of the connected 2-levels graph $\rmg$ associated to the bicoloured partition $\pi$~. 

\medskip

Let us now prove, by induction on the number $k\in \NN^*$ of inner edges, that the above-mentioned assignment defines a bijection between the set of bicoloured partitions $\pi$ of the roundabout $\rho$ and the set of 
connected 2-levels graphs $\rmg$ whose image under the composition map $\gamma$ is equal to $\nu$. 
For $k=1$, this statement is precisely \cref{prop:CuttingInfDec}. 
Suppose now that this statement holds true for $k$ and let us prove it for $k+1$. \medskip

First we claim that any bicoloured partition $\pi$ of $\rho$ with $k+1$ inner edges gives a connected 2-levels graph $\rmg$ whose image under the composition map $\gamma$ is equal to  $\nu$. 
To see this, we notice that the action of the composition map $\gamma$ on the connected 2-levels graph $\rmg$ is equal to the action of the infinitesimal composition product applied 
to the 2-vertices subgraphs defined by the first inner edge followed by the action of the composition map on the resulting connected 2-levels graph $\rmg'$: since the first inner edge corresponds to the first arc, it delimits a single top (or bottom) cyclic stairway element that is composed inside the cyclic stairway element located below (or above). 
One can now see that the connected 2-levels graph $\rmg'$ corresponds to the bicoloured partition $\pi'$ obtained 
forgetting the first arc and by keeping the outside colour: this follows actually from \cref{prop:CuttingInfDec} applied to the sub-roundabout delimited by the smallest arc containing the first arc. Now we can apply the induction hypothesis since both the conected 2-level graph $\rmg'$ and the bicoloured partition $\pi'$ have $k$ inner edges. 

\medskip

In the other way round, let us show that any connected 2-levels graph $\rmg$ with $k+1$ inner edges, whose image under the composition map $\gamma$ is equal to  $\nu$, is produced by a bicoloured partition $\pi$ of $\rho$. 
The arguments are quite similar. We claim that the connected 2-levels graph $\rmg$ has at least one bottom or top vertex connected to just one top or bottom vertex: otherwise the total number of vertices would be infinite or the graph would not be of genus $0$. We decide that the associated inner edge is the first one. 
The connected 2-levels graph $\rmg'$ obtained by applying the infinitesimal composition product  
to the 2-vertices subgraphs defined by this first inner edge satisfies the same property as $\rmg$, that is its image under the composition map $\gamma$ is equal to  $\nu$. 
By the induction hypothesis, the connected 2-levels graph $\rmg'$ corresponds to a bicoloured partition $\pi'$ of the roundabout $\rho$. We consider the white or black part corresponding to the top or bottom vertex obtained by the abovementioned infinitesimal composition product. 
Viewing this part as a sub-roundabout, 
\cref{prop:CuttingInfDec} ensures that it admits an arc whose part with new colour corresponds to the original bottom or top vertex. The finer bicoloured partition $\pi$ obtained from $\pi'$ by adding this new arc is sent to the connected 2-levels graph $\rmg$ under the above assignment, which concludes the proof. 
\end{proof}

\begin{definition}[Infinitesimal black/white partition]
An \emph{infinitesimal black} (respectively \emph{white}) \emph{partition} of a roundabout is 
a partition of a roundabout into parts where one is black (respectively white) and all the others white (respectively black): the frontier between a white part and a black part is always delimited by an oriented arc with the oriented edge leaving the white part and entering into the black part, see \cref{Fig:BWSplitingRounda}. 

In an infinitesimal black partition, all the incoming edges of the roundabout have to belong to white parts, 
each white part should have at least one incoming edge of the roundabout or at least 
two edges leaving to black part, and 
the black part needs to have at least two outgoing edges or 
two incoming edges from white parts when it has only one output edge of the roundabout.

In an infinitesimal white partition, all the outgoing edges of the roundabout have to belong to black parts, 
each black part should have at least one outgoing edge, 
and the white part needs to have at least two input edges of the roundabout when there is only one black part. 
\end{definition}

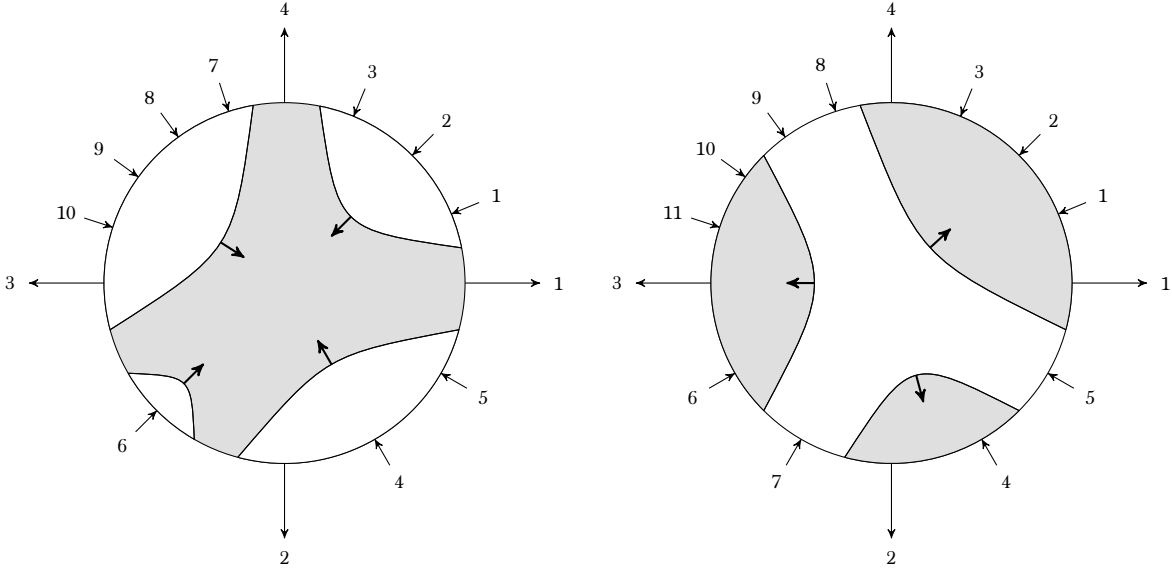
\begin{figure}[h!]
	\tikzmath{
		\radIn = 2.4; % rayon disque
		\radMid = 2.8; % extremite fleche entrante
		\radOut = 3.4; % extremite fleche sortante
		\dec = 0.25; % decalage indice
		}
\[
	\begin{tikzpicture}
		\coordinate (C) at (0,0);
		\draw[fill=gray!25] (C) circle (\radIn);
		\foreach \angle/\i in {0/1, 270/2, 180/3, 90/4}{
			\draw[->] (\angle:\radIn) -- (\angle:\radOut);
			\draw (\angle:\radOut) ++ (\angle:\dec) node {\tiny{$\i$}};  %%% Rayons sortants
		}
		\foreach \angle/\i in 
		{22.5/1,45/2, 67.5/3, -60/4, -30/5, 225/6, 108/7, 126/8, 144/9, 162/10
		}{
			\draw[<-] (\angle:\radIn) -- (\angle:\radMid);
			\draw (\angle:\radMid) ++ (\angle:\dec) node {\tiny{$\i$}}; %%% Rayons entrants
		}
		%%%%% BLANC 1
		\draw[fill=white] (100:\radIn) .. controls (147.5:0.8) .. (195:\radIn) arc (195:100:\radIn);
		\draw (100:\radIn) .. controls (147.5:0.8) .. (195:\radIn) node[midway] (A) {};
		\draw[->,thick] (A.center) --++ (327.5:{1.5*\dec});
		%%%%% BLANC 2
		\draw[fill=white] (78.75:\radIn) .. controls (45:1) .. (11.25:\radIn) arc (11.25:78.75:\radIn);
		\draw (78.75:\radIn) .. controls (45:1) .. (11.25:\radIn) node[midway] (A) {};
		\draw[->,thick] (A.center) --++ (225:{1.5*\dec});		
		%%%%% BLANC 3
		\draw[fill=white] (255:\radIn) .. controls (300:1.1) .. (345:\radIn) arc (345:255:\radIn);
		\draw (345:\radIn) .. controls (300:1.1) .. (255:\radIn) node[midway] (A) {};
		\draw[->,thick] (A.center) --++ (120:{1.5*\dec});				
		%%%%% BLANC 4
		\draw[fill=white] (210:\radIn) .. controls (225:1.75) .. (240:\radIn) arc (240:210:\radIn);
		\draw (210:\radIn) .. controls (225:1.75) .. (240:\radIn) node[midway] (A) {};
		\draw[->,thick] (A.center) --++ (45:{1.5*\dec});						
	\end{tikzpicture}
\quad
	\begin{tikzpicture}
		\coordinate (C) at (0,0);
		\draw[fill=white] (C) circle (\radIn);
		\foreach \angle/\i in {0/1, 270/2, 180/3, 90/4}{
			\draw[->] (\angle:\radIn) -- (\angle:\radOut);
			\draw (\angle:\radOut) ++ (\angle:\dec) node {\tiny{$\i$}};  %%% Rayons sortants
		}
		\foreach \angle/\i in 
		{22.5/1,45/2,67.5/3, -60/4, -30/5, 210/6, 240/7, 108/8, 126/9, 144/10, 162/11
		}{
			\draw[<-] (\angle:\radIn) -- (\angle:\radMid);
			\draw (\angle:\radMid) ++ (\angle:\dec) node {\tiny{$\i$}}; %%% Rayons entrants
		}
		%%%%% GRIS 1
		\draw[fill=gray!25] (135:\radIn) .. controls (180:0.8) .. (225:\radIn) arc (225:135:\radIn);
		\draw (135:\radIn) .. controls (180:0.8) .. (225:\radIn) node[midway] (A) {};
		\draw[->,thick] (A.center) --++ (180:{1.5*\dec});
		%%%%% GRIS 2
		\draw[fill=gray!25] (-15:\radIn) .. controls (42.5:0.5) .. (100:\radIn) arc (100:-15:\radIn);
		\draw (-15:\radIn) .. controls (42.5:0.5) .. (100:\radIn) node[midway] (A) {};
		\draw[->,thick] (A.center) --++ (42.5:{1.5*\dec});		
		%%%%% GRIS 3
		\draw[fill=gray!25] (255:\radIn) .. controls (285:1) .. (315:\radIn) arc (315:255:\radIn);
		\draw (255:\radIn) .. controls (285:1) .. (315:\radIn) node[midway] (A) {};
		\draw[->,thick] (A.center) --++ (285:{1.5*\dec});						
	\end{tikzpicture}
	\]
	\caption{An infinitesimal black and an infinitesimal white partition of a roundabout.}
	\label{Fig:BWSplitingRounda}
\end{figure}

Let us recall from \cite[Definition~3.12]{HLV20} that the \emph{left} and \emph{right infinitesimal composition products}
\[
	\rmM  \libt  \rmN
	\qquad \text{and} \qquad
	\rmM 	\ribt  \rmN\ ,
\]
are the sub-$\Sy$-bimodules of $(\Ibox \oplus \rmM)\boxtimes \rmN$ and $\rmM\boxtimes (\Ibox \oplus \rmN)$ made up of the linear parts in $\rmM$ and $\rmN$ respectively. For $n\geqslant 1$, we consider their following summands
\[
	\rmM \lhd_{(n)} \rmN \
	\qquad \text{and} \qquad
	\rmM \,{}_{(n)}\!\rhd \rmN \ ,
\]
made up of one element of $\rmM$ on the bottom (resp. of $\rmN$ at the top) and $n$ elements of $\rmN$ on the top (resp. of $\rmM$ at the bottom). Notice that the top (resp. bottom) level is saturated by elements of $\rmN$ (resp. $\rmM$) and that the bottom (resp. top) level contains one element of $\rmM$ (resp. $\rmN$) and possibly many copies of the identity element from $\Ibox$~, see \cref{Fig:Triangle}.

\begin{figure*}[h]
	\begin{tikzpicture}[scale=0.75]
		\draw[thick] (1.95,1) to[out=270,in=90] (3.5,-1) ;
		\draw[thick] (4,-1)-- (4,-2);
		\draw[thick] (3,-1)-- (3,-2);
		\draw[thick] (6,-1)-- (6,-2);
		\draw[thick] (1,1) to[out=270,in=90] (2,-2);
		\draw[thick] (0.5,2) to  (0.5,1);
		\draw[thick] (0,1) to  (0,-2);
		\draw[thick] (7.5,2) to[out=270,in=90] (5.95,1) to[out=270,in=90] (8,-1) to[out=270,in=90] (8,-2);
		\draw[thick] (5,2) to  (5,1) to[out=270,in=90] (4.5,-1);
		\draw[draw=white,double=black,double distance=2*\pgflinewidth,thick] (9,2) to  (9,1) to (9,-2);
		\draw[thick] (1.5,2) to  (1.5,1);
		\draw[draw=white,double=black,double distance=2*\pgflinewidth,thick] (8,1) to[out=270,in=90] (5.5,-1);
		\draw[thick] (5,-1) 	 -- (5,-2);
		\draw[draw=white,double=black,double distance=2*\pgflinewidth,thick] (4,2) to[out=270,in=90] (4,1);
		\draw[draw=white,double=black,double distance=2*\pgflinewidth,thick]  (4,1) to[out=270,in=90] (1,-2);
		\draw[draw=white,double=black,double distance=2*\pgflinewidth,thick] (6.5,2) to[out=270,in=90] (8,1) ;
		\draw[fill=white] (-0.3,0.8) rectangle (2.3,1.2);
		\draw[fill=white] (3.7,0.8) rectangle (6.3,1.2);
		\draw[fill=white] (7.7,0.8) rectangle (9.3,1.2);
		\draw[fill=white] (2.7,-1.2) rectangle (6.3,-0.8);
		\draw (1,1) node {\small{$\nu_1$}};
		\draw (5,1) node {\small{$\nu_2$}};
		\draw (8.5,1) node {\small{$\nu_3$}};
		\draw (4.5,-1) node {\small{$\mu$}};
	\end{tikzpicture}
	\caption{An element of $\rmM \lhd_{(3)} \rmN$~.}
	\label{Fig:Triangle}
\end{figure*}

Let us recall from \cite[Definition~3.13]{HLV20} that  the definitions of  the \emph{left} and \emph{right infinitesimal decomposition maps} of the coproperad $\rmC_{\rm pCY}$ are given by
\[
	\begin{tikzcd}[column sep=large, row sep=tiny]
		\Cop{}{(*)} \ \colon \ \overline{\rmC}_{\rm pCY} \arrow[r,"\Delta"] &
		\rmC_{\rm pCY} \boxtimes \rmC_{\rm pCY} \arrow[r,"(\eps; \id)\boxtimes \id"] & \overline{\rmC}_{\rm pCY}\libt  \rmC_{\rm pCY} \ , \\
		\Cop{(*)}{} \ \colon \ \overline{\rmC}_{\rm pCY} \arrow[r,"\Delta"] &
		\rmC_{\rm pCY} \boxtimes \rmC_{\rm pCY} \arrow[r,"\id \boxtimes (\eps; \id)"] & \rmC_{\rm pCY} \ribt  \overline{\rmC}_{\rm pCY}\ .
	\end{tikzcd}
\]
Heuristically speaking, they amount to decompose the elements of $\rmC_{\rm pCY}$ into two levels, using the  decomposition map $\Delta$, but keeping only a non-trivial one at the bottom or at the top respectively. 

\begin{proposition}\label{prop:LRInfDecMap}
The  terms appearing in the left infinitesimal decomposition map $\Cop{}{(*)}$
(respectively the right infinitesimal decomposition map $\Cop{(*)}{}$)
 of the codioperad $\rmC_{\rm pCY}$ are in one-to-one correspondence with 
the infinitesimal black partitions (respectively infinitesimal white partitions) of roundabouts.
\end{proposition}

\begin{proof}
This is a direct corollary of \cref{prop:BiColPartFullDec}. For instance, 
the two infinitesimal black and white partitions of a roundabout 
displayed on \cref{Fig:BWSplitingRounda} correspond to the following two 2-levels of stairways
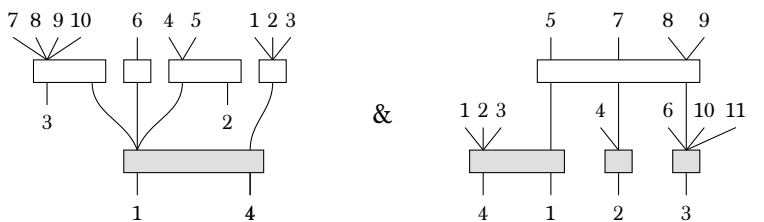
\[
\vcenter{\hbox{	\begin{tikzpicture}[scale=0.6]
		\coordinate (Z) at (8,3);		
		\draw[thin] (7.5,1) to[out=90,in=270] (8,2.5);	 					
		\draw[thin] (Z) --++ (-0.4,0.5) node[above] {${\scriptstyle{1}}$};
		\draw[thin] (Z) --++ (0,0.5) node[above] {${\scriptstyle{2}}$};		
		\draw[thin] (Z) --++ (0.4,0.5) node[above] {${\scriptstyle{3}}$};	
		\draw[thin] (7,2.5) --(7,2) 	node[below] {$\scriptstyle{2}$};		
		\draw[fill=white] (7.7,2.5) rectangle (8.3,3);
		\coordinate (B) at (3,3);			
		\draw[thin] (B) --++ (-0.75,0.5) node[above] {${\scriptstyle{7}}$};
		\draw[thin] (B) --++ (-0.25,0.5) node[above] {${\scriptstyle{8}}$};		
		\draw[thin] (B) --++ (0.25,0.5) node[above] {${\scriptstyle{9}}$};			
		\draw[thin] (B) --++ (0.75,0.5) node[above] {${\scriptstyle{10}}$};					
		\draw[thin] (5,1) to[out=70,in=270] (6,2.5);	 								
		\draw[fill=white] (5.7,2.5) rectangle (7.3,3);	
		\draw[thin] (5,3) --++ (0,0.5) node[above] {${\scriptstyle{6}}$};		
		\draw[fill=white] (4.7,2.5) rectangle (5.3,3);				
		\coordinate (A) at (6,3);
		\draw[thin] (A) --++(-0.3, 0.5) node[above] {${\scriptstyle{4}}$};	
		\draw[thin] (A) --++(0.3, 0.5) node[above] {${\scriptstyle{5}}$};			
		\draw[thin] (5,0.5) -- (5,0) node[below] {$\scriptstyle{1}$};		
		\draw[thin] (5,1) -- (5,2.5);
		\draw[thin] (7.5, 0.5) --(7.5,0)  node[below] {$\scriptstyle{4}$};						
		\draw[fill=gray!25] (4.7,0.5) rectangle (7.8,1);	
		\draw[thin] (5,1) to[out=110,in=270] (4,2.5);	 
		\draw[fill=white] (2.7,2.5) rectangle (4.3,3);
		\draw[thin] (7.5, 0.5) --(7.5,0)  node[below] {$\scriptstyle{4}$};											\draw[thin] (3, 2.5) --(3,2)  node[below] {$\scriptstyle{3}$};									
	\end{tikzpicture}
	}}
\qquad \& \qquad 
\vcenter{\hbox{\begin{tikzpicture}[scale=0.6]
		\draw[thin] (2,2.5) -- (2,1);
		\draw[thin] (3.5,2.5) -- (3.5,1);
		\draw[thin] (5,2.5) -- (5,1);		
		\draw[fill=white] (1.7,2.5) rectangle (5.3,3);
		\draw[fill=gray!25] (3.2,1) rectangle (3.8,0.5);	
		\draw[thin] (3.5,0.5) -- (3.5,0) node[below] {$\scriptstyle{2}$};									
		\draw[fill=gray!25] (4.7,1) rectangle (5.3,0.5);			
		\draw[thin] (5,0.5) -- (5,0) node[below] {$\scriptstyle{3}$};	
		\draw[thin] (2,0.5) -- (2,0) node[below] {$\scriptstyle{1}$};			
		\draw[thin] (0.5,0.5) -- (0.5,0) node[below] {$\scriptstyle{4}$};					
		\draw[fill=gray!25] (0.2,1) rectangle (2.3,0.5);					
		\coordinate (A) at (0.5,1);		
		\draw[thin] (A) --++ (-0.4,0.5) node[above] {${\scriptstyle{1}}$};
		\draw[thin] (A) --++ (0,0.5) node[above] {${\scriptstyle{2}}$};		
		\draw[thin] (A) --++ (0.4,0.5) node[above] {${\scriptstyle{3}}$};	
		\coordinate (A) at (3.5,1);			
		\draw[thin] (A) --++ (-0.4,0.5) node[above] {${\scriptstyle{4}}$};
		\coordinate (A) at (5,1);			
		\draw[thin] (A) --++ (-0.4,0.5) node[above] {${\scriptstyle{6}}$};
		\draw[thin] (A) --++ (0.4,0.5) node[above] {${\scriptstyle{10}}$};		
		\draw[thin] (A) --++ (1.1,0.5) node[above] {${\scriptstyle{11}}$};	
		\coordinate (A) at (2,3);			
		\draw[thin] (A) --++ (0,0.5) node[above] {${\scriptstyle{5}}$};		
		\coordinate (A) at (3.5,3);			
		\draw[thin] (A) --++ (0,0.5) node[above] {${\scriptstyle{7}}$};				
		\coordinate (A) at (5,3);			
		\draw[thin] (A) --++ (-0.4,0.5) node[above] {${\scriptstyle{8}}$};				
		\draw[thin] (A) --++ (0.4,0.5) node[above] {${\scriptstyle{9}}$};						
	\end{tikzpicture}}}~,
\]
appearing respectively on the image of $\Cop{}{(*)}$ and $\Cop{(*)}{}$~. 
\end{proof}

\begin{definition}[Oriented partition]
An \emph{oriented partition} of a roundabout is a partition of a roundabout into parts delimited by oriented arcs, see \cref{Fig:OrPartRounda}. 
Each part of an oriented partition is required to have at least two outgoing edges or two incoming edges when it has only one ongoing edge. 
\end{definition}

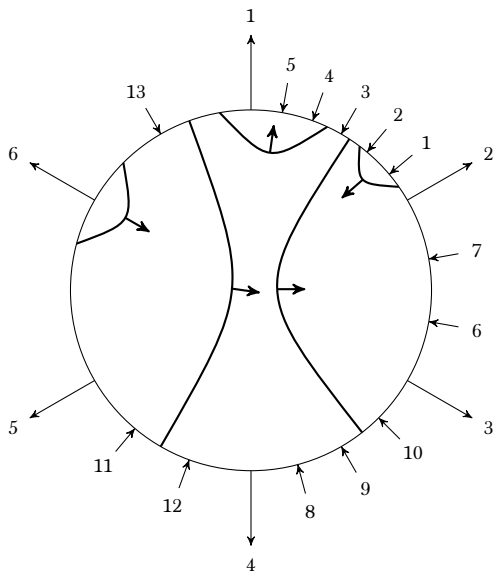
\begin{figure}[h!]
	\tikzmath{
		\radIn = 2.4; % rayon disque
		\radMid = 2.8; % extremite fleche entrante
		\radOut = 3.4; % extremite fleche sortante
		\dec = 0.25; % decalage indice
		}
\[
	\begin{tikzpicture}
		\coordinate (C) at (0,0);
		\draw[fill=white] (C) circle (\radIn);
		\foreach \angle/\i in {90/1, 30/2, -30/3, -90/4, -150/5, 150/6}{
			\draw[->] (\angle:\radIn) -- (\angle:\radOut);
			\draw (\angle:\radOut) ++ (\angle:\dec) node {\tiny{$\i$}};  %%% Rayons sortants
		}
		\foreach \angle/\i in 
		{40/1,50/2,60/3,70/4,80/5, -10/6, 10/7, -75/8, -60/9, -45/10, -130/11, 
		-110/12, 120/13
		}{
			\draw[<-] (\angle:\radIn) -- (\angle:\radMid);
			\draw (\angle:\radMid) ++ (\angle:\dec) node {\tiny{$\i$}}; %%% Rayons entrants
		}
		%%%%%  1
		\draw[thick] (110:\radIn) .. controls (0,0) .. (240:\radIn) node[midway] (A) {};
		\draw[->,thick] (A.center) --++ (-8:{1.5*\dec});
		%%%%%  2
		\draw[thick] (100:\radIn) .. controls (82:1.7) .. (65:\radIn) node[midway] (A) {};
		\draw[->,thick] (A.center) --++ (82:{1.5*\dec});		
		%%%%%  3
		\draw[thick] (57:\radIn) .. controls (0,0) .. (-52:\radIn) node[midway] (A) {};
		\draw[->,thick] (A.center) --++ (0:{1.5*\dec});		
		%%%%% 4
		\draw[thick] (165:\radIn) .. controls (150:1.8) .. (135:\radIn) node[midway] (B) {};
		\draw[->,thick] (B.center) --++ (-30:{1.5*\dec});
		%%%%% 5
		\draw[thick] (53:\radIn) .. controls (45:2) .. (35:\radIn) node[midway] (B) {};
		\draw[->,thick] (B.center) --++ (221:{1.5*\dec});
	\end{tikzpicture} 
\]
	\caption{An oriented partition of a roundabout.}
	\label{Fig:OrPartRounda}
\end{figure}

\begin{proposition}\label{prop:ComonadicDec}
The  terms appearing in the comonadic decomposition map 
\[
\widetilde{\Delta} \colon \overline{\rmC}_{\rm pCY} \to \scrG^c\left(\overline{\rmC}_{\rm pCY}\right)
\] 
 of the codioperad $\rmC_{\rm pCY}$ are in one-to-one correspondence with 
the oriented partitions of roundabouts. 
\end{proposition}

\begin{proof}
This proof is similar to one of \cref{prop:BiColPartFullDec}: the comonadic decomposition map $\widetilde{\Delta} \colon \overline{\rmC}_{\rm pCY} \to \scrG^c\left(\overline{\rmC}_{\rm pCY}\right)$ is the linear dual of the monadic composition map 
\[\widetilde{\gamma}\  \colon\  \scrG\left(\frac{\overline{\mathrm{u}\DPois}^!}{\field u}\right) \to \frac{\overline{\mathrm{u}\DPois}^!}{\field u}~,\]
so the terms appearing on the image of a cyclic stairway basis element $\nu$ 
under the comonadic decomposition map $\widetilde{\Delta}$ are the connected flow-directed genus $0$ graphs of cyclic stairways 
whose image under the monadic composition map $\widetilde{\gamma}$ gives $\nu$. 
Using again the bijection between cyclic stairways and roundabouts of \cref{lem:RoundaBasis} and the interpretation of parts of an oriented partition as sub-roundabouts, one can show that 
 these connected flow-directed genus $0$ graphs of cyclic stairways 
  are in one-to-one correspondence with oriented partitions of roundabouts applying \emph{mutatis mutandis} the arguments of the proof of \cref{prop:BiColPartFullDec}. For instance, the connected flow-directed genus $0$ graphs corresponding to the oriented partition displayed on \cref{Fig:OrPartRounda} is the following one 
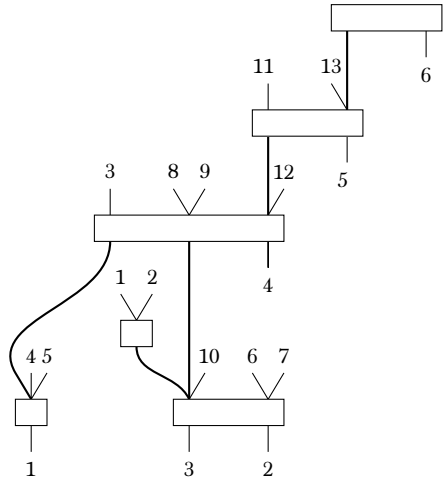
\[
\vcenter{\hbox{\begin{tikzpicture}[scale=0.7]
%%%%%%%%%%%%%%%%%%  E
		\coordinate (A1) at (0,4);
		\draw[thin]
		(A1) -- ++(0,1) node[above] {${\scriptstyle{3}}$};
		\coordinate (A2) at (1.5,4);
		\draw[thin]
		(A2) ++ (0,0.5) -- ++(-0.3,0.5)	node[above] {${\scriptstyle{8}}$}	
		(A2) ++ (0,0.5) -- ++(0.3,0.5)	node[above] {${\scriptstyle{9}}$};	
		\coordinate (A3) at (3,4);
		\draw[thin]
		(A3) --++(0,-0.5) node[below] {$\scriptstyle{4}$} --
		(A3) ++(0,0.5) -- ++(0.3,0.5) node[above] {${\scriptstyle{12}}$};
		\coordinate (A4) at (4.5,4);
		\draw[fill=white] (-0.3,4) rectangle (3.3, 4.5);
%%%%%%%%%%%%%%%%%%%  THICK EDGES
		\draw[thick] (A1) to[out=270,in=120] ++(-1.8,-2.5) --++ (0.3, -0.5) coordinate (F1);
		\draw[thick] (A2) to[out=270,in=90] ++(0,-3) coordinate (C1);		
		\draw[thick] (3,4.5) --++ (0, 1.5) coordinate (B1);
		\draw[thick] (4.5,6.5) to[out=90,in=270] (4.5, 8) coordinate (Z1);	
		\draw[thick] (0.5,2) to[out=270,in=115] (1.5, 1);						
%%%%%%%%%%%%%%%%%%  B 
		\draw[thin]
		(B1)  ++ (0,0.5) -- ++(0,0.5)	node[above] {${\scriptstyle{11}}\ $}	
		(B1)  ++(1.5, 0.5) --++(-0.3,0.5)	node[above] {${\scriptstyle{13}}$}
		(B1)  ++ (1.5,0) -- ++(0,-0.5)	node[below] {${\scriptstyle{5}}\ $};
		\draw[fill=white] (2.7,6) rectangle (4.8, 6.5);		
%%%%%%%%%%%%%%%%%%  C
		\draw[thin]
		(C1)  -- ++(0.3,0.5)	node[above] {$\ {\scriptstyle{10}}$};	
		\draw[thin]
		(C1)  -- ++(0,-1)	node[below] {${\scriptstyle{3}}$};		
		\coordinate (C2) at (3,1);		
		\draw[thin]
		(C2)  -- ++(-0.3,0.5)	node[above] {${\scriptstyle{6}}$}	
		(C2)  -- ++(0.3,0.5)	node[above] {${\scriptstyle{7}}$}
		(C2)  -- ++(0,-1)	node[below] {${\scriptstyle{2}}$};		
		\draw[fill=white] (1.2,1) rectangle (3.3, 0.5);
%%%%%%%%%%%%%%%%%%  A		
		\draw[fill=white] (4.2, 8) rectangle (6.3, 8.5);
		\draw[thin]
		(Z1) ++(1.5, 0)  -- ++(0,-0.5) node[below] {${\scriptstyle{6}}$}	;		
%%%%%%%%%%%%%%%%%%  D		
		\coordinate (D1) at (0.5,2.5);
		\draw[thin]
		(D1)  -- ++(-0.3,0.5)	node[above] {${\scriptstyle{1}}$}	
		(D1) -- ++(0.3,0.5)	node[above] {${\scriptstyle{2}}$};			
		\draw[fill=white] (0.2,2) rectangle (0.8, 2.5);
%%%%%%%%%%%%%%%%%%  F		
		\draw[thin]
		(F1)  -- ++(0,0.5)	node[above] {${\scriptstyle{4}}$}	
		(F1)  -- ++(0.3,0.5)	node[above] {${\scriptstyle{5}}$}
		(F1)  -- ++(0,-1)		node[below] {${\scriptstyle{1}}$};	
		\draw[fill=white] (-1.8,0.5) rectangle (-1.2, 1);
	\end{tikzpicture}}}~.
\]
\end{proof}

\begin{remark}
As we have seen throughout this section, the interpretation in terms of roundabouts provides us with an efficient combinatorial way to come up with all the terms appearing in the various decomposition maps of the codioperad 
$\rmC_{\rm pCY}$, and the Koszul dual coproperad $\DPois^{\ac}$ by restriction. 
However the only way we know to make the various signs explicit amounts to use 
the bijection with the stairways graphs, and to compute the signs of the associated composites in the properad 
$\frac{\mathrm{u}\DPois^!}{\field u}$~.
\end{remark}

%%%%%%%%%%%%%%%%%%%%%%%%  SECTION 4

\section{Homotopy theory via $\infty$-morphisms}\label{sec:InftyMor}

The categories of pre-Calabi--Yau algebras, homotopy double Poisson gebras, and $\rmV_\infty$-gebras are
 faithfully encoded by (co)properads. 
This key property opens the doors to the properadic calculus developed in \cite{HLV20, HLV22} and in \cite{CV22}. 
In this section, we apply the general results of \emph{loc. cit.} to the present algebraic structures to settle the notion of an $\infty$-morphism satisfying all the expected homotopical properties : homotopy transfer theorem, invertibility of $\infty$-quasi-isomorphisms, formality, Koszul hierarchy, and twisting procedure. 

\subsection{Moduli spaces}
\cref{thm:main} allows us to compare the moduli spaces of curved pre-Calabi--Yau algebras and 
the moduli spaces of curved homotopy double Poisson gebras. 
Let us first recall the definition of the moduli space of tensorial pre-Calabi--Yau algebras introduced by W.-K. Yeung  in \cite[Section~3.1]{Yeu18}.

\begin{definition}[Moduli space of tensorial pre-Calabi--Yau structures]
	The \emph{moduli space of tensorial pre-Calabi--Yau structures} on a graded vector space $A$ is the Kan complex $\mathrm{cpCY}\{A\}$ defined by
	\[
		\mathrm{cpCY}\{A\}\coloneq \mathrm{MC}(\cnec_A\otimes \Omega_\bullet)~,
	\]
	where $\Omega_\bullet$ is the Sullivan simplicial commutative algebra of polynomial differential forms of the standard simplicies $\Delta^\bullet$~.
\end{definition}

Similarly, the moduli space of homotopy gebras, introduced by S. Yalin in {\cite[Section 3.1]{Yal16moduli}}, takes the following form here. We denote by $\cDPois_\infty$ the cobar construction of the Infinitesimal coproperad $\cDPois^\antish$, which is given by the quasi-free properad on the desuspension of $\cDPois^\antish$~. (We do not consider any coaugmentation ideal here as the infinitesimal coproperad $\cDPois^\antish$ fails to be coaugmented.) Since 
$\cDPois^\antish$ is non-negatively graded, the dg properad $\cDPois_\infty$ is cofibrant. 

\begin{definition}[Moduli space of curved homotopy double Poisson gebra structures]
	The \emph{moduli space of curved homotopy double Poisson gebra structures} on a dg vector space $A$ is the Kan complex $\cDPois_\infty\{A\}$ defined by
	\[
		\cDPois_\infty\{A\}	\coloneq
		\Hom_{\mathsf{dg}\ \Properad}\left(\cDPois_\infty,\End_A\otimes \, \Omega_\bullet\right) \ ,
	\]
	where the properad  $\End_A\otimes\, \Omega_\bullet$ is defined by
	\[
		(\End_A\otimes\, \Omega_\bullet) (m,n) \coloneq \End_A(m,n) \otimes \Omega_\bullet \ .
	\]
\end{definition}

\begin{remark}
The same definition applies obviously to pre-Calabi--Yau algebras, homotopy double Poisson gebras, and $\rmV_\infty$-algebras and provides them with suitable moduli spaces of structures. 
\end{remark}

\begin{proposition}
The embedding $\cnec_A \hookrightarrow\mathfrak{c}\convDPois_A$ of dg Lie-admissible algebras induces 
an embedding of 
Kan complexes
	\[
		  \mathrm{cpCY}\{A\}  \hookrightarrow  \cDPois_\infty\{A\} \ ,
	\]
which is an isomorphism when $A$ is degree-wise finite dimensional. 
\end{proposition}

\begin{proof}
This is a direct corollary of the isomorphism of Kan complexes 
\[
		\cDPois_\infty\{A\}
		\cong
		\mathrm{MC}\left(\widehat{\Hom}\left(\cDPois^\antish,\End_A\right)\otimes \Omega_\bullet\right) \ , 
	\]
see  \cite[Theorem 3.12]{Yal16moduli}, 
and 
\cref{thm:main}.
\end{proof}

\subsection{Infinity-morphisms}\label{subsect::infty_morphism}
The  advantage of the approach of pre-Calabi--Yau algebras, homotopy double Poisson gebras, 
and $V_\infty$-gebras via structural coproperads is that
the properadic calculus recently developed in \cite{HLV20} automatically equips
them with a higher notion of $\infty$-morphism which behaves well with respect the homotopical properties: homological invertibility of $\infty$-quasi-isomorphisms,  homotopy transfer theorem and Deligne groupoid, for instance.
These various homotopical constructions actually rely on the various decomposition maps of the coproperads. After applying the general theory to the present cases, we make them explicit using the combinatorics introduced in \cref{subsec:CoproduitsCpCY}.
 \\

\paragraph*{\sc Notation.} In this section, we treat the homotopical properties of pre-Calabi--Yau algebras, homotopy double Poisson gebras, and $V_\infty$-gebras at the same time. To this extend, we will use the commun notation $\rmC$ for either the codioperad $\rmC_{\rm pCY}$ encoding pre-Calabi--Yau algebras, the codioperad $\DPois^\antish$ encoding homotopy double Poisson gebras, and the coproperad $\rmV^{\ac}$ encoding $\rmV_\infty$-gebras. 
\\

 Given
$f\in \Hom_{\Sy}\big(\rmC, \End^A_B\big)$, $\alpha \in \Hom_{\Sy}\big(\rmC, \End_A\big)$, and $\beta\in \Hom_{\Sy}\big(\rmC, \End_B\big)$, the \emph{left action}  of $\beta$ on $f$ and the \emph{right action} of $\alpha$ on $f$ are defined respectively by
\[
	\begin{tikzcd}[column sep=normal, row sep=tiny]
		\beta \lhd f  \ \colon \ &
		\overline{\rmC} \arrow[r,"\Cop{}{(*)}"]
		& \overline{\rmC} \libt \DPois^\antish
		\arrow[r,"\beta\libt f"] & \End_B \libt \End^A_B \arrow[r] & \End^A_B\ ,
		\\
		& \Ibox \arrow[r,"\cong"] &  \Ibox \boxtimes \Ibox
		\arrow[r,"\beta\boxtimes f"] & \End_B \boxtimes \End^A_B \arrow[r] & \End^A_B\ ,
		\\
		f \rhd \alpha  \ \colon \ &
		\overline{\rmC} \arrow[r,"\Cop{(*)}{}"] &  \rmC \ribt \overline{\rmC}
		\arrow[r,"f\ribt \alpha"] & \End^A_B \ribt \End_A \arrow[r] & \End^A_B \ ,\\
		& \Ibox \arrow[r,"\cong"] &  \Ibox \boxtimes \Ibox
		\arrow[r,"f\boxtimes \alpha"] & \End^A_B \boxtimes \End_A \arrow[r] & \End^A_B\ ,
	\end{tikzcd}
\]
where the rightmost arrows are given by the usual composition of functions.

\begin{definition}[$\infty$-morphism]
	\label{def:infmor}
	Let $A,B$ be dg vector spaces equipped respectively with two 
	pre-Calabi--Yau algebra structures (respectively homotopy double Poisson gebra or $\rmV_\infty$-gebra) 
	 viewed as two Maurer--Cartan elements $\alpha\in \widehat{\Hom}\big(\rmC , \End_A\!\big)$ and $\beta\in \widehat{\Hom}\big(\rmC , \End_B\!\big)$ in the respective convolution Lie-admissible algebras.
	An \emph{$\infty$-morphism} $f \colon (A,\alpha) \rightsquigarrow (B, \beta)$ is a degree $0$  map of $\Sy$-bimodules $f \colon \rmC  \to  \End^A_B$  satisfying the equation
	\begin{equation}\label{eq:Morph}
		\partial (f)= f  \rhd \alpha - \beta \lhd f\ .
	\end{equation}
	The \emph{composite} of $\infty$-morphisms is defined by
	\[
		\begin{tikzcd}
			g\circledcirc f \ \colon \ \rmC \arrow[r,"\Delta"]
			& \rmC  \boxtimes \rmC \arrow[r,"g\boxtimes f"]
			& \End^B_C \boxtimes \End_B^A \arrow[r]
			& \End_C^A \  .
		\end{tikzcd}
	\]
\end{definition}

\begin{proposition}\label{prop:HoInfiniCat}
Pre-Calabi--Yau algebras (respectively homotopy double Poisson gebras or $\rmV_\infty$-gebras) equipped with their $\infty$-morphisms and the composite $\circledcirc$ form a category.
\end{proposition}

\begin{proof}
This is a direct corollary of the general theory settled in {\cite[Section~3]{HLV20}}.
\end{proof}

By definition, an $\infty$-morphism of pre-Calabi--Yau algebras (respectively homotopy double Poisson gebras) is a collection of maps
\[
	f_{\lambda_1, \ldots, \lambda_m} \colon A^{\otimes n} \too B^{\otimes m}
\]
of degree $n-1$, for any ordered partition $\lambda_1+\cdots+\lambda_m=n$ of $n\geqslant 1$ into non-negative integers (respectively positive integers), for $m\geqslant 2$, and $\lambda_1\geqslant 1$, for $m=1$, satisfying the relations given by \eqref{eq:Morph}. 
We can make these relations explicit using \cref{prop:LRInfDecMap} as follows. 
The term $f  \rhd \alpha$ evaluated at cyclic stairway basis element or equivalently at a roundabout (\cref{lem:RoundaBasis}) is equal to the sum over the infinitesimal black partitions of that roundabout interpreted, up to a sign, as a composite of components of $\alpha$ (white parts) with one 
map $f_{\lambda_1, \ldots, \lambda_m}$ (black part). The other term $\beta \lhd f$ is interpreted similarly with infinitesimal white partitions where the white part corresponds to one map $f_{\lambda_1, \ldots, \lambda_m}$ and where the black parts correspond to components of $\beta$. 
Following the same method, the composite $g\circledcirc f$ of two $\infty$-morphisms $f$ and $g$ can be made explicit using \cref{prop:BiColPartFullDec}: 
the term $g\circledcirc f$ evaluated at a roundabout is equal to the sum over the bicoloured partitions interpreted, up to a sign, as a composite of components of $f$ corresponding to white parts with 
components of $g$ corresponding to black parts. 

\begin{proposition}
The present category of pre-Calabi--Yau algebras with their $\infty$-morphisms is equivalent to the one defined in \cite{KTV21}. 
\end{proposition}

\begin{proof}
One can pass from the present combinatorial objects (partitioned roundabouts) to the ones of \cite{KTV21} by contracting all the oriented arcs and keeping the oriented edges between the induced sub-roundabouts. 
This way, we recover the definition given in \cite[Definition~25]{KTV21} of an $\infty$-morphism of pre-Calabi--Yau algebras and the definition given in \cite[Definition~26]{KTV21} of their composite.
\end{proof}

A first immediate advantage of the present properadic definitions lies in the straightforward proofs that $\infty$-morphisms of pre-Calabi--Yau algebras are stable under composition and that this composition is associative. Here are now further homotopical properties which are direct applications of the general theory settled in \cite{HLV20}. 

\begin{definition}[$\infty$-isotopy, $\infty$-isomorphism, and $\infty$-quasi-isomorphism]
	An \emph{$\infty$-isotopy} (resp. \emph{$\infty$-iso\-mor\-phi\-sm}, \emph{$\infty$-quasi-isomorphism}) is an $\infty$-morphism  whose first component $f_{1}$ is the identity (resp. isomorphism, quasi-isomorphism).
\end{definition}

\begin{proposition}
The invertible $\infty$-morphisms 
of pre-Calabi--Yau algebras (respectively homotopy double Poisson gebras and $\rmV_\infty$-gebras)
are the $\infty$-isomorphisms.
\end{proposition}

\begin{proof}
This is a special case of \cite[Theorem~3.22]{HLV20} applied to the 
(Koszul dual) coproperad $\rmC_{\rm pCY}$ (respectively $\DPois^\antish$ or $\rmV^{\ac}$).
\end{proof}

Recall that the \emph{Deligne groupoid} $\mathsf{Del}(\g)$ of a complete Lie algebra $\g$ is made up of
the Maurer--Cartan elements for objects and the gauges for isomorphisms. We refer the reader to \cite[Chapter~1]{DSV21} for more details.

\begin{proposition}
	The Deligne groupoid associated to the convolution algebra
	of pre-Calabi--Yau algebras (respectively homotopy double Poisson gebras or $\rmV_\infty$-algebras)
	on $A$
	is isomorphic to the groupoid of of pre-Calabi--Yau algebra structures 
	(respectively homotopy double Poisson gebra or $\rmV_\infty$-gebra) on $A$ with their $\infty$-isotopies:
	\[
		\mathsf{Del}\left(\widehat{\Hom}\left(\rmC , \End_A\right)\right)\cong
		\left(\Omega \rmC\textrm{-}\,\mathrm{gebras}, \infty\textrm{-}\,\mathrm{isotopies}\right)\ .
	\]
\end{proposition}

\begin{proof}
	This is a special case of \cite[Theorem~1.2]{CV22} applied to the 
	(Koszul dual) 
coproperad $\rmC_{\rm pCY}$ (respectively $\DPois^\antish$ or $\rmV^{\ac}$).
\end{proof}

\begin{theorem}[Homological invertibility of $\infty$-quasi-isomorphisms]\label{thm:InvInfQI}
	For any $\infty$-quasi-isomorphism
	$f:A \overset{\sim}{\rightsquigarrow} B$ of 
	of pre-Calabi--Yau algebras (respectively homotopy double Poisson gebras or $\rmV_\infty$-gebras), there exists an $\infty$-quasi-isomorphism $g:B \overset{\sim}{\rightsquigarrow} A$
	whose first component induces 	the homology inverse of the first component of $f$.
\end{theorem}

\begin{proof}
	This is a special case of \cite[Theorem~4.18]{HLV20} applied to 
	(Koszul dual) 
coproperad $\rmC_{\rm pCY}$ (respectively $\DPois^\antish$ or $\rmV^{\ac}$).

\end{proof}

Recall that a \emph{contraction} of a dg vector space $(A, d_A)$ is another dg vector space $(H,d_H)$ equipped with chain maps $i$ and $p$ and a
homotopy $h$ of degree $1$
\[
	\begin{tikzcd}
		(A,d_A)\arrow[r, shift left, "p"] \arrow[loop left, distance=1.5em,, "h"] & (H,d_H) \arrow[l, shift left, "i"]
	\end{tikzcd} \ ,
\]
satisfying
\begin{align*}
	pi=\id_H\ , \quad \id_A-ip = d_Ah+hd_A\ , \quad
	hi=0\ , \quad ph=0\ , \quad \text{and} \quad h^2=0 \ .
\end{align*}

\begin{theorem}[Homotopy Transfer Theorem]\label{thm:HTT}
	For any 
	pre-Calabi--Yau algebra structure (respectively homotopy double Poisson gebra or $\rmV_\infty$-gebra)
	on a dg vector space $A$ and any contraction of it, there exists
	a pre-Calabi--Yau algebra structure (respectively homotopy double Poisson gebra or $\rmV_\infty$-gebra)
	 on $H$ and extensions of the chain maps $i$ and $p$ into $\infty$-quasi-isomorphisms.
\end{theorem}

\begin{proof}
	This is a special case of \cite[Theorem~4.14]{HLV20} applied to the (Koszul dual) coproperad $\rmC_{\rm pCY}$ (respectively $\DPois^\antish$ or $\rmV^{\ac}$).
\end{proof}

Notice that \emph{loc. cit.} provides us with explicit formulas for the transferred structure and the extensions of the two maps. For instance, the homotopy transferred structure is given by 
	\[
	\begin{tikzcd}[column sep=normal]
	\overline{\rmC} \arrow[r,"\widetilde{\Delta}"] & \scrG^c\left(\overline{\rmC} \right) \arrow[r,"\scrG^c(\susp\alpha)"] & \Bar\End_A \arrow[r,"\varphi"] & \End_H~, 
	\end{tikzcd}
	\]
where $\widetilde{\Delta}$ is the comonadic decomposition map of the (Koszul dual) coproperad $\rmC_{\rm pCY}$ (respectively $\DPois^\antish$ or $\rmV^{\ac}$)  made up of all the ways to decompose its elements, 
where $\alpha$ is the initial 
pre-Calabi--Yau algebra structure on $A$ (respectively homotopy double Poisson gebra or $\rmV_\infty$-gebra), 
and where 
$\varphi$ is the Van der Laan twisting morphism which amonts to labelling the edges of 
the graphs by the maps $i$, $p$, and $h$ of the contraction is a levelled way, see \cite[Section~4]{HLV20} for more details. 
Since pre-Calabi--Yau algebras and homotopy double Poisson gebras are actually encoded by codioperads, 
this level-wise universal formula for the homotopy transferred structure simplifies as follows according to \cref{appendiceHTT}. 

\begin{proposition}[Formula for the Homotopy Transfer Theorem]\label{prop:FormulaHTT}
	For any 
	pre-Calabi--Yau algebra structure (respectively homotopy double Poisson gebra)
	on a dg vector space $A$ and any contraction of it, 
	the pre-Calabi--Yau algebra structure (respectively homotopy double Poisson gebra) transferred onto $H$ is equal to the signed sum over all the oriented partitions of roundabouts with inputs labelled by $i$, outputs labelled by $p$, and oriented inner edges labelled by $h$, see \cref{Fig:HTTRounda}. 
\begin{figure}[h!]	
		\tikzmath{
		\radIn = 2.4; % rayon disque
		\radMid = 2.8; % extremite fleche entrante
		\radOut = 3.4; % extremite fleche sortante
		\semidecp = -0.5; % semi decalage pour les p
		\semideci = -0.7; % semi decalage pour les i		
		\dec = 0.25; % decalage indice
		}
\[
\vcenter{\hbox{\begin{tikzpicture}
		\coordinate (C) at (0,0);
		\draw[fill=white] (C) circle (\radIn);
		\foreach \angle/\i in {90/1, 30/2, -30/3, -90/4, -150/5, 150/6}{
			\draw[->] (\angle:\radIn) -- (\angle:\radOut);
			\draw (\angle:\radOut) ++ (\angle:\dec) node {\scalebox{0.6}{$\i$}};  %%% Rayons sortants
			\draw (\angle+3:\radOut) ++ (\angle+3:\semidecp) node {\scalebox{0.8}{$p$}};  %%% p		
		}
		\foreach \angle/\i in 
		{40/1,50/2,60/3,70/4,80/5, -10/6, 10/7, -75/8, -60/9, -45/10, -130/11, 
		-110/12, 120/13
		}{
			\draw[<-] (\angle:\radIn) -- (\angle:\radMid);
			\draw (\angle:\radMid) ++ (\angle:\dec) node {\scalebox{0.6}{$\i$}}; %%% Rayons entrants
			\draw (\angle+3:\radOut) ++ (\angle+3:\semideci) node {\scalebox{0.8}{$i$}};  %%% i					
		}
		%%%%%  1
		\draw[thick] (110:\radIn) .. controls (0,0) .. (240:\radIn) node[midway] (A) {};
		\draw[->,thick] (A.center) --++ (-8:{1.5*\dec});
		\draw (A.center) ++ (225:{0.4*\semideci}) node {\scalebox{0.8}{$h$}};
		%%%%%  2
		\draw[thick] (100:\radIn) .. controls (82:1.7) .. (65:\radIn) node[midway] (A) {};
		\draw[->,thick] (A.center) --++ (82:{1.5*\dec});	
		\draw (A.center) ++ (225:{0.4*\semideci}) node {\scalebox{0.8}{$h$}};
		%%%%%  3
		\draw[thick] (57:\radIn) .. controls (0,0) .. (-52:\radIn) node[midway] (A) {};
		\draw[->,thick] (A.center) --++ (0:{1.5*\dec});		
		\draw (A.center) ++ (225:{0.4*\semideci}) node {\scalebox{0.8}{$h$}};
		%%%%% 4
		\draw[thick] (165:\radIn) .. controls (150:1.8) .. (135:\radIn) node[midway] (B) {};
		\draw[->,thick] (B.center) --++ (-30:{1.5*\dec});		
		\draw (B.center) ++ (215:{0.4*\semideci}) node {\scalebox{0.8}{$h$}};
		%%%%% 5
		\draw[thick] (53:\radIn) .. controls (45:2) .. (35:\radIn) node[midway] (B) {};
		\draw[->,thick] (B.center) --++ (221:{1.5*\dec});
		\draw (B.center) ++ (340:{0.35*\semideci}) node {\scalebox{0.8}{$h$}};
	\end{tikzpicture}}}
\quad \longleftrightarrow \quad \vcenter{\hbox{
\begin{tikzpicture}[scale=0.7]
%%%%%%%%%%%%%%%%%%  E
		\coordinate (A1) at (0,4);
		\draw[thin]
		(A1)  ++(0,1) ++(-0.1,0.1) node[below right] {\scalebox{0.8}{$i$}}		
		(A1) -- ++(0,1) node[above] {${\scalebox{0.6}{$3$}}$};
		\coordinate (A2) at (1.5,4);
		\draw[thin]
		(A2) ++ (0,0.5)  ++(-0.3,0.5) ++(0.15,0.1) node[below left] {\scalebox{0.8}{$i$}}		
		(A2) ++ (0,0.5) -- ++(-0.3,0.5)	node[above] {${\scalebox{0.6}{$8$}}$}	
		(A2) ++ (0,0.5) ++(0.3,0.5) ++(-0.15,0.1) node[below right] {\scalebox{0.8}{$i$}}		
		(A2) ++ (0,0.5) -- ++(0.3,0.5)	node[above] {${\scalebox{0.6}{$9$}}$};	
		\coordinate (A3) at (3,4);
		\draw[thin]
		(A3) --++(0,-0.5) node[below] {$\scalebox{0.6}{$4$}$} --
		(A3) ++(0,0.5) -- ++(0.3,0.5) node[above] {${\scalebox{0.6}{$12$}}$}
		(A3) ++(0,0.5) ++(0.3,0.5) ++(-0.15,0.1) node[below right] {\scalebox{0.8}{$i$}}		
		(A3)  ++(0,-0.5) ++ (-0.1,-0.1) node[above right] {\scalebox{0.8}{$p$}};
		\coordinate (A4) at (4.5,4);
		\draw[fill=white] (-0.3,4) rectangle (3.3, 4.5);
%%%%%%%%%%%%%%%%%%%  THICK EDGES
		\draw[thick] (A1) to[out=270,in=120] ++(-1.8,-2.5) --++ (0.3, -0.5) coordinate (F1);
		\draw[thick] (A2) to[out=270,in=90] ++(0,-3) coordinate (C1);		
		\draw[thick] (3,4.5) --++ (0, 1.5) coordinate (B1);
		\draw[thick] (4.5,6.5) to[out=90,in=270] (4.5, 8) coordinate (Z1);	
		\draw[thick] (0.5,2) to[out=270,in=115] (1.5, 1);
%%%%%%%%%%%%%%%%%%  B 
		\draw[thin]
		(B1)  ++ (0,0.5) -- ++(0,0.5)	node[above] {${\scalebox{0.6}{$11$}}\ $}
		(B1)  ++ (0,0.5)  ++(0,0.5) ++(-0.1,0.1) node[below right] {\scalebox{0.8}{$i$}}				
		(B1)  ++(1.5, 0.5) --++(-0.3,0.5)	node[above] {${\scalebox{0.6}{$13$}}$}
		(B1)  ++(1.5, 0.5) ++(-0.3,0.5) ++(0.15,0.1) node[below left] {\scalebox{0.8}{$i$}}			
		(B1)  ++ (1.5,0) ++(0,-0.5) ++ (-0.1,-0.1) node[above right] {\scalebox{0.8}{$p$}}		
		(B1)  ++ (1.5,0) -- ++(0,-0.5)	node[below] {${\scalebox{0.6}{$5$}}\ $};
		\draw[fill=white] (2.7,6) rectangle (4.8, 6.5);		
%%%%%%%%%%%%%%%%%%  C
		\draw[thin]
		(C1)   ++(0.3,0.5) ++(-0.15,0.1) node[below right] {\scalebox{0.8}{$i$}}	
		(C1)  -- ++(0.3,0.5)	node[above] {$\ {\scalebox{0.6}{$10$}}$};	
		\draw[thin]
		(C1)  -- ++(0,-1)	node[below] {${\scalebox{0.6}{$3$}}$};		
		\coordinate (C2) at (3,1);		
		\draw[thin]
		(C2)  -- ++(-0.3,0.5)	node[above] {${\scalebox{0.6}{$6$}}$}	
		(C2)  ++(-0.3,0.5) ++(0.15,0.1) node[below left] {\scalebox{0.8}{$i$}}			
		(C2)  -- ++(0.3,0.5)	node[above] {${\scalebox{0.6}{$7$}}$}
		(C2)   ++(0.3,0.5) ++(-0.15,0.1) node[below right] {\scalebox{0.8}{$i$}}				
		(C2)  -- ++(0,-1)	node[below] {${\scalebox{0.6}{$2$}}$}		
		(C2)  ++(0,-1) ++ (-0.1,-0.1) node[above right] {\scalebox{0.8}{$p$}}				
		(C1)  ++(0,-1) ++(-0.1,-0.1) node[above right] {\scalebox{0.8}{$p$}};		
		\draw[fill=white] (1.2,1) rectangle (3.3, 0.5);
%%%%%%%%%%%%%%%%%%  A		
		\draw[fill=white] (4.2, 8) rectangle (6.3, 8.5);
		\draw[thin]
		(Z1)  ++(1.5, 0)   ++(0,-0.5) ++ (-0.1,-0.1) node[above right] {\scalebox{0.8}{$p$}}			
		(Z1) ++(1.5, 0)  -- ++(0,-0.5) node[below] {${\scalebox{0.6}{$6$}}$}	;		
%%%%%%%%%%%%%%%%%%  D		
		\coordinate (D1) at (0.5,2.5);
		\draw[thin]
		(D1)  -- ++(-0.3,0.5)	node[above] {${\scalebox{0.6}{$1$}}$} 
		(D1)  ++ (-0.3,0.5) ++(0.1,0.1) node[below left] {\scalebox{0.8}{$i$}}
		(D1) -- ++(0.3,0.5)	node[above] {${\scalebox{0.6}{$2$}}$}
		(D1)  ++ (0.3,0.5) ++(-0.1,0.1) node[below right] {\scalebox{0.8}{$i$}};  		
		\draw[fill=white] (0.2,2) rectangle (0.8, 2.5);
%%%%%%%%%%%%%%%%%%  F		
		\draw[thin]
		(F1)  -- ++(0,0.5)	node[above] {${\scalebox{0.6}{$4$}}$}	
		(F1)  ++(0,0.5) ++(-0.1,0.2) node[below right] {\scalebox{0.8}{$i$}}  		
		(F1)  -- ++(0.5,0.5)	node[above] {${\scalebox{0.6}{$5$}}$}
		(F1)  ++(0.5,0.5) ++(-0.18,0.2) node[below right] {\scalebox{0.8}{$i$}}
		(F1)  -- ++(0,-1)		node[below] {${\scalebox{0.6}{$1$}}$}
		(F1)  ++(0,-1) ++(-0.1,-0.1) node[above right] {\scalebox{0.8}{$p$}}	;	
		\draw[fill=white] (-1.8,0.5) rectangle (-1.2, 1);
%%%%%%%%%%%%%%%%%%   h
		\draw (4.5,7.25) node[right] {\scalebox{0.8}{$h$}}; 
		\draw (3,5.25) node[left] {\scalebox{0.8}{$h$}}; 
		\draw (-1.5,2.75) node {\scalebox{0.8}{$h$}}; 
		\draw (1.5,2.5) node[right] {\scalebox{0.8}{$h$}}; 
		\draw (1.1,1.75) node {\scalebox{0.8}{$h$}}; 
	\end{tikzpicture}}}
\]
	\caption{One term appearing in the homotopy transferred structure.}
	\label{Fig:HTTRounda}
\end{figure}
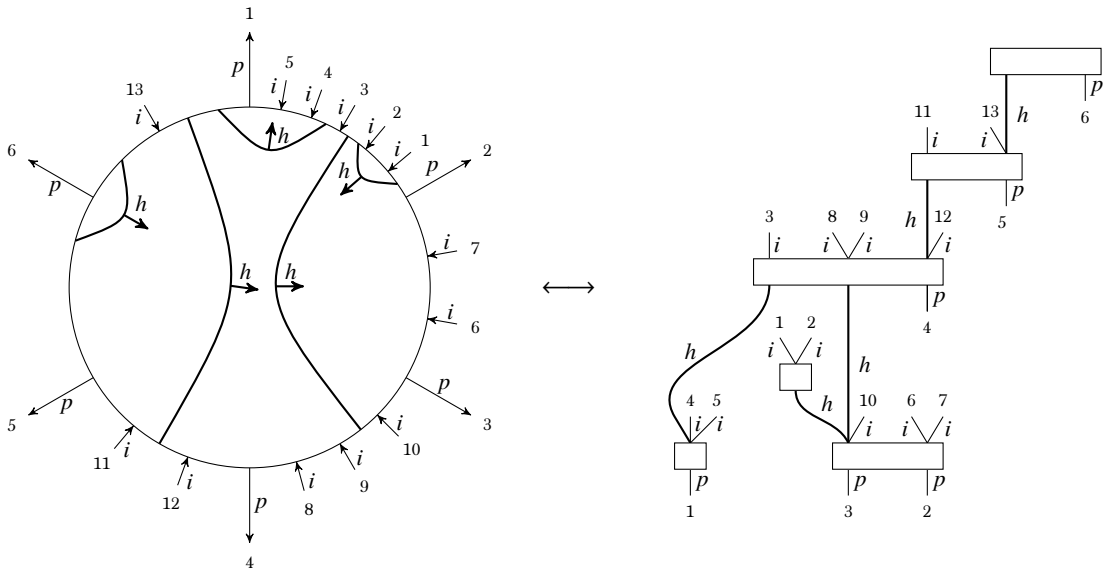

\end{proposition}

\begin{proof}
This is a direct application of \cref{thm:HTTCodioperad}. 
\end{proof}

\begin{remark}
The methods of B. Fresse, especially \cite[Theorem~7.3]{Fre10}, applied to the cobar construction of the 
(Koszul dual) coproperad $\rmC_{\rm pCY}$ (respectively $\DPois^\antish$ or $\rmV^{\ac}$), 
  show the existence of homotopy transferred structures through quasi-isomorphisms over a ring of characteristic $0$. Such a statement was also proved directly, i.e. without the use of the properadic calculus, in the case of pre-Calabi--Yau algebras by M. Boucrot in \cite{Boucrot2023}. 
\end{remark}

\medskip

In the recent paper \cite{HLV22}, we went further and established even more structural properties for $\infty$-morphisms associated to coproperads, like the following formality property, that automatically apply to the cases of 
pre-Calabi--Yau algebras, homotopy double Poisson gebras, and $\rmV_\infty$-gebras.

\begin{proposition}\label{prop:InftyIsot}
If two  pre-Calabi--Yau algebra structures 
 $\alpha$ and $\alpha'$ (respectively homotopy double Poisson gebra or $\rmV_\infty$-algebra) on the same underlying chain complex $A$ are $\infty$-isotopic then they are related by a two-arrows zig-zag of (strict) quasi-isomorphisms
\[
\begin{array}{c}
\exists\  \text{$\infty$-isotopy} \\
\begin{tikzcd}[column sep=normal]
	(A,\alpha) \ar[r,squiggly,"="] 
	& (A, \alpha')
\end{tikzcd} 
\end{array}
\Longrightarrow
\begin{array}{c}
\exists\  \text{two-arrows zig-zag of quasi-isomorphisms} \\
\begin{tikzcd}
	(A, \alpha) 
	& \point \ar[l,"\sim"', "g"] \ar[r,"\sim", "h"'] 
	& (A, \alpha')
	\end{tikzcd}
  \ ,
\end{array}
\]
where the respective isomorphisms in homology satisfy $H(g)=H(h)^{-1}$~.
\end{proposition}

\begin{proof}
	This is a special case of \cite[Proposition~1.14]{HLV22} applied to the 
	(Koszul dual) coproperad $\rmC_{\rm pCY}$ (respectively $\DPois^\antish$ or $\rmV^{\ac}$).
\end{proof}

\begin{theorem}\label{thm:MainInftyQi}
Two
pre-Calabi--Yau algebra structures  
 $(A,\alpha)$ and $(B, \beta)$ 
 (respectively homotopy double Poisson gebra or $\rmV_\infty$-algebra)
 are $\infty$-quasi-isomorphic if and only if they are related by a zig-zag of (strict) quasi-isomorphisms:
\[
\begin{array}{c}
\exists\  \text{$\infty$-quasi-isomorphism} \\
\begin{tikzcd}[column sep=normal]
	(A,\alpha) \ar[r,squiggly,"\sim"] 
	& (B, \beta)
\end{tikzcd} 
\end{array}
\Longleftrightarrow
\begin{array}{c}
 \exists\  \text{zig-zag of quasi-isomorphisms} \\
\begin{tikzcd}[column sep=small]
	(A, \alpha) \ar[r,"\sim"] 
	& \point
	& \point \ar[l,"\sim"'] \ar[r,dotted, no head]
	&\point
	& \point \ar[l,"\sim"'] \ar[r,"\sim"] 
	& (B, \beta)
	\end{tikzcd}  \ .
\end{array}
\]
\end{theorem}

\begin{proof}
	This is a special case of \cite[Theorem~1.16]{HLV22} applied to the (Koszul dual) coproperad $\rmC_{\rm pCY}$ (respectively $\DPois^\antish$ or $\rmV^{\ac}$) .
\end{proof}

The results of \cite[Section~3]{HLV22} also provide us with a simplicial enrichment of the category of 
pre-Calabi--Yau algebras (respectively homotopy double Poisson gebras or $\rmV_\infty$-algebras)
 and $\infty$-morphisms such that the associated homotopy category coincides with the localisation with respect to ($\infty$-)quasi-isomorphisms. 

\medskip

We finish this section with an application of the general theory of Koszul hierarchy given in \cite[Section~3]{CV22} 
which provides us here with a way to construct homotopy double Poisson gebras (and thus pre-Calabi--Yau algebras) from shifted $\DPois^!$-gebra structures  via universal formulas. 
The notion of a \emph{shifted $\DPois^!$-gebra} is defined as a $\DPois^!$-gebra (\cref{subsec::DPois_antish}) except that we require the generating binary product and ``dual'' double bracket to be of homological degree $+1$ instead of $-1$.
Recall that the original construction, given by J.-L. Koszul in \cite{Koszul85}, produces a (shifted) homotopy Lie algebra from a commutative product on a dg vector space. 

\begin{theorem}[Koszul hierarchy]\label{thm:KoszulHierach}
	Given  a shifted $\DPois^!$-algebra structure on graded vector space $A$ and a differential $d$ on $A$, the action of the associated canonical element of the gauge group on the trivial double Poisson gebra structure on $(A,d)$ produces a homotopy double Poisson gebra. 
\end{theorem}

\begin{proof}
This is a special case of \cite[Section~3]{CV22} applied to the Koszul dual coproperad $\DPois^\antish$, but let us give it more details. Under the notation $\rms \DPois^!$ for the properad encoding shifted $\DPois^!$-gebras, a structure of a gebra over it amounts to a morphism of properads
$\theta\ \colon \rms \DPois^!\to \End_A$~. 
Identifying the dual  treewise stairways basis given in \cref{subsec::DPois_antish}, we get a canonical isomorphism of graded $\Sy$-bimodules $\mathrm{K} \colon {\DPois}^\antish \xrightarrow{\cong} {\rms\DPois}^!$ and thus  a canonical element $\Theta\coloneq  \theta(\mathrm{K})$ of the 
deformation gauge group integrating the dg Lie algebra $\convDPois_A$~.
This gauge action, producing a homotopy double Poisson gebra structure by definition, is given by the explicit formula 
\[\Theta^{-1} \cdot d =  -\left(\Theta^{-1}; \d\left( \Theta^{-1} \right)\right)\circledcirc\Theta=
\Theta^{-1} \circledcirc \left(\Theta; \d \Theta \right) \in \MC\left(\convDPois_A\right)\ ,\]
see \cite[Theorem~2.29]{CV25I} for more details. 
\end{proof}

Notice that we do \emph{not} require the shifted $\DPois^!$-gebra structure to be compatible with the differential. 
This construction applied to a shifted associative algebra viewed as a shifted $\DPois^!$-algebra, with trivial dual double bracket, provides us with the $\rmA_\infty$-algebra of K. Borjeson \cite{B13}.

\subsection{Infinity-morphisms in the curved case}\label{subsec:HoThCHDPois}
In this section, we extend the properadic calculus of \cite{HLV20} along the lines of \cite{DSV21, CV22} to handle the case of curved pre-Calabi--Yau algebras and equivalently homotopy curved double Poisson gebras. 
The first mandatory move is to work over the category of \emph{complete} graded vector spaces in order to manage the infinite sums that  appear in  this case.  

\medskip

Let us recall that a graded vector space $A$ is \emph{complete} when it is equipped with 
a degree-wise decreasing filtration 
	\[
A_n=\rmF_0 A_n \supseteq \rmF_1 A_n \supseteq \rmF_2 A_n \supseteq \cdots \supseteq \rmF_k A_n \supseteq \rmF_{k+1}A_n \supseteq \cdots
	\]
	of sub-vector spaces such that the associated topology is complete, i.e. $A_n\cong \lim_{k\to\infty} A_n/\rmF_k  A_n$~. 
The category of complete graded vector space is made up of maps which preserve the respective filtrations. 
Equipped with the complete tensor product $\widehat{\otimes}$, it forms a symmetric monoidal category, where one can perform properadic calculus, see \cite[Chapter~2]{DSV21} for more details. For instance, its complete endomorphism properad is denoted by  $\eend_A$~. 

\medskip

On the opposite to the Koszul dual coproperad $\DPois^\antish$ which encodes homotopy double Poisson gebras, the infinitesimal coproperad $\mathrm{c}\DPois^\antish$ which encodes curved homotopy double Poisson gebras (equivalently curved pre-Calabi--Yau algebras) fails to form a coproperad: one can iterate its infinitesimal decomposition map in order to form two-levels graphs, but the upshot will be made up of infinite series and not finite sums. 
As a consequence  the operations $\Delta$, $\Cop{(*)}{}$ and $\Cop{}{(*)}$ are not well-defined on the level of $\mathrm{c}\DPois^\antish$~. 
However, the global operations $\circledcirc$, $\lhd$ and $\rhd$, defined by the same overall formulas, make sense as soon as one works with complete graded modules $A, B$, and with elements $f$ satisfying $f_{0}(1)\in \rmF_1 A$~.

\begin{definition}[$\infty$-morphism in the  complete curved case]
	\label{def:infmor}
	Let $A,B$ be complete graded vector spaces equipped respectively with two curved homotopy double Poisson gebra structures (equivalently curved pre-Calabi--Yau algebras) $\alpha\in \MC\left(\widehat{\Hom}\left(\mathrm{c}\DPois^\antish, \eend_A\right)\right)$ and 
	$\beta\in \MC\left(\widehat{\Hom}\left(\mathrm{c}\DPois^\antish, \eend_B\right)\right)$~.
	An \emph{$\infty$-morphism} $f \colon (A,\alpha) \rightsquigarrow (B, \beta)$ is a degree $0$ map of complete $\Sy$-bimodules $f \colon \mathrm{c}\DPois^\antish \to  \eend^A_B$  such that $f(1)\in \rmF_1 B_0$ and  satisfying the equation
	\begin{equation}\label{eq:MorphCurved}
		\partial (f)= f  \rhd \alpha - \beta \lhd f\ .
	\end{equation}
	The \emph{composite} of $\infty$-morphisms is defined by $g\circledcirc f$~. 
\end{definition}

\begin{proposition}
	Complete curved homotopy double Poisson gebras (equivalently complete curved pre-Calabi--Yau algebras) equipped with their $\infty$-morphisms and the composite $\circledcirc$ form a category.
\end{proposition}

\begin{proof}
This proof is similar to that of \cref{prop:HoInfiniCat} once one checks that all the terms make sense thanks to the completion assumption. 
\end{proof}

By definition, an $\infty$-morphism of complete curved homotopy double Poisson gebras (equivalently complete curved pre-Calabi--Yau algebras) is a collection of filtration preserving maps 
\[
	f_{\lambda_1, \ldots, \lambda_m} \colon A^{\widehat{\otimes} n} \too B^{\widehat{\otimes} m}
\]
of degree $n-1$, for any ordered partition $\lambda_1+\cdots+\lambda_m=n$ of any $n\geqslant 0$  into non-negative integers, satisfying the relations given by \eqref{eq:MorphCurved}. 
One can make this notion of $\infty$-morphisms and their composite explicit like in the previous section, using the  combinatorics of \cref{subsec:CoproduitsCpCY} suitably extended to take into account the curvature terms corresponding to $\lambda_1=0$. 

\medskip

We conclude the present study with the \emph{twisting procedure}, which is a universal way to produce  new curved homotopy double Poisson gebras (equivalently curved pre-Calabi--Yau algebras) from the data of one element. 

\begin{theorem}[Twisting procedure]\label{thm:TwistProc}
Given any complete curved homotopy double Poisson algebra structure 
(equivalently complete curved pre-Calabi--Yau algebra) 
$\left\{\calm_{\lambda_1, \ldots, \lambda_m}\right\}_{\lambda_1, \ldots, \lambda_m\geqslant 0}$  on $A$ and any element 
$a\in \rmF_1 A_{-1}$~, the sum of all the insertions of $a$ at any possible inputs forms a complete curved homotopy double Poisson gebra (equivalently complete curved pre-Calabi--Yau algebra): 
\[
\calm_{\lambda_1, \ldots, \lambda_m}^a\coloneq 
\sum_{\substack{
r^1_0, \ldots, r^1_{\lambda_1}\geqslant 0 \\
\rule{0pt}{6pt}\ldots \\
\rule{0pt}{8pt} r^m_0, \ldots, r^m_{\lambda_m}\geqslant 0}}
(-1)^{\eta} 
\calm_{\lambda_1+r^1, \ldots, \lambda_m+r^m}
\left(a^{r^1_0}, -, a^{r^1_1}, -, \cdots, -, a^{r^1_{\lambda_1}}, \cdots, 
a^{r^m_0}, -, a^{r^m_1}, -, \cdots, -, a^{r^m_{\lambda_m}}
\right)~, 
\]
where $r^j\coloneq r^j_0+\cdots+r^j_{\lambda_j}$~, for $1\leqslant j \leqslant m$, and where 
\[\eta= rn+\frac{r(r-1)}{2}+\sum_{j=1}^{m-1}r^j\left(1+\sum_{k=j+1}^m \lambda_k+r^k\right)+
\sum_{j=1}^m\left(\sum_{i=0}^{\lambda_j} \left(r^j_i +1\right)\left(r^j_0+\cdots+r^j_{i-1}\right)+\frac{r^j_i\left(r^j_i-1\right)}{2} \right)  
~, 
\]
with $r\coloneq r^1+\cdots+r^m$~. 
\end{theorem}

\begin{proof}
	This is a direct corollary of \cite[Definition~4.11]{CV22}. Let us denote by $\alpha\in \mathfrak{c}\convDPois_A$ the Maurer--Cartan element corresponding to the original complete curved homotopy double Poisson gebra. Its image under the action of the element $\1-a$ of the deformation gauge group integrating the complete convolution dg Lie algebra $\mathfrak{c}\convDPois_A$ is equal to 
	\[(\1-a)\cdot \alpha = \pap{(\1-a)}{\alpha}{(\1+a)}=\alpha\lhd (\1+a)~, \]
by \cite[Theorem~2.29]{CV25I} and under the notations of \emph{loc. cit.}. 
The right-hand side produces the formula for the twisted structure claimed in the statement. 
The sign is computed as in the proof of \cref{prop:InfDecDPoisAC}.
\end{proof}

Notice that the original and the twisted complete curved homotopy double Poisson gebras (equivalently complete curved pre-Calabi--Yau algebras) are gauge equivalent by the conceptual definition. 

\begin{definition}[Maurer--Cartan equation of a complete curved pre-Calabi--Yau algebra]
The \emph{Maurer--Cartan equation} of a complete curved pre-Calabi--Yau algebra
$\left(A, \left\{\calm_{\lambda_1, \ldots, \lambda_m}\right\}\right)$ is 
\begin{equation}
\sum_{n\geqslant 0} \calm_n\left(a^n\right)=\calm_0 + \calm_1(a)+\calm_2(a,a)+\calm_3(a,a,a)+\cdots =0~. 
\end{equation}
Its solutions $a\in \rmF_1 A_{-1}$ are called the \emph{Maurer--Cartan elements} of the complete curved pre-Calabi--Yau algebra. 
\end{definition}

The Maurer--Cartan equation of a complete curved pre-Calabi--Yau algebra is noting but the Maurer--Cartan equation of its underlying complete curved $\rmA_\infty$-algebra. 

\begin{proposition}\label{prop:TwistProcNocurv}
Given any complete pre-Calabi--Yau algebra $\left(A, \left\{\calm_{\lambda_1, \ldots, \lambda_m}\right\}\right)$ and any Maurer--Cartan element $a\in \rmF_1 A_{-1}$, the twisted operations 
$\left\{\calm_{\lambda_1, \ldots, \lambda_m}^a\right\}$ form a complete pre-Calabi--Yau algebra. 
\end{proposition}

\begin{proof}
This is a direct corollary of \cref{thm:TwistProc} as follows. First, we interpret the original complete pre-Calabi--Yau algebra as a curved one with trivial curvature. Then, the action of the element $a$ produces a new twisted complete curved pre-Calabi--Yau algebra. Finally, the curvature of this new structure is equal to the left-hand side of the Maurer--Cartan equation, so it vanishes and one actually gets a complete pre-Calabi--Yau algebra. 
\end{proof}

\appendix

\section{Proof of the main theorem}\label{appendice}

This appendix contains the proof of the main result (\cref{thm:main}) of the present paper, that is the existence of  a canonical and functorial isomorphism of dg Lie-admissible algebras
\[
	\mathfrak{c}\convDPois_{A}  \cong  \hhc_A \ .
\]

\begin{proof}[Proof of \cref{thm:main}]
	The isomorphism $\Upsilon \colon \mathfrak{c}\convDPois_{A}  \to  \hhc_A$ from left to right is given by the following assignment:
	\[
		\Upsilon\left(\bar{\rmF}\right)_{\lambda_1, \ldots, \lambda_m}\ \colon \ \susp a_1 \otimes \cdots \otimes \susp a_n \mapsto
		(-1)^{\left|\bar{\rmF}\right|+\nabla(a_1, \ldots, a_n)+\lambda_m+\frac{n(n+1)}{2}} \, \susp \bar{\rmF}\left(\nu_{\lambda_1, \ldots, \lambda_m}\right)(a_1\otimes \cdots \otimes a_n)~,
	\]
	for $\bar{\rmF} \colon \mathrm{c}\DPois^\antish \to \End_A$,
	where
	$\nu_{\lambda_1, \ldots, \lambda_m}$ denotes the basis elements of the coproperad
	$\mathrm{c}\DPois^\antish$,
	and where
	\[
		\nabla(a_1, \ldots, a_n)\coloneq (n-1)|a_1|+(n-2)|a_2|+\cdots+2|a_{n-2}|+|a_{n-1}|~.
	\]
	First, we check that this map is equivariant with respect to the actions of the cyclic groups; so it will induce a well-defined map on the level of invariant elements. Under the  notation,
	\[
		\bar{\rmF}\left(\nu_{\lambda_1, \ldots, \lambda_m}\right)(a_1\otimes \cdots \otimes a_n)
		= \sum  b_1\otimes \cdots \otimes b_m~,  \]
	we get
	\begin{align*}
		\left(\tau_m^{-1}\cdot \Upsilon\left(\bar{\rmF}\right)_{\lambda_1, \ldots, \lambda_m}
		\cdot \tau_{\lambda_1, \cdots, \lambda_m}\right)
		(\susp a_{\lambda_1+1}\otimes \cdots \otimes \susp a_n\otimes \susp a_1 \otimes \cdots \otimes \susp a_{\lambda_1})
		=\sum (-1)^A \,  \susp\, b_2\otimes \cdots \otimes b_m\otimes b_1~,
	\end{align*}
	with
	\[A=(|a_1|+\cdots+|a_{\lambda_1}|+\lambda_1)(|a_{\lambda_1+1}|+\cdots+|a_n|+n-\lambda_1)
		+\left|\bar{\rmF}\right|+\nabla(a_1, \ldots, a_n)+\lambda_m+\frac{n(n+1)}{2}+
		|b_1|(|b_2|+\cdots+|b_m|)~.
	\]
	On the other hand, since the action of the cyclic groups on
	${\Hom}\left(\mathrm{c}\DPois^\antish , \End_A\right)$ is given by conjugation,
	we have
	\begin{align*}
		\Upsilon\left(\bar{\rmF}\cdot \tau_m\right)_{\lambda_2, \ldots, \lambda_m, \lambda_1}
		(\susp a_{\lambda_1+1}\otimes \cdots \otimes \susp a_n\otimes \susp a_1 \otimes \cdots \otimes \susp a_{\lambda_1})
		=\sum (-1)^B \,  \susp\, b_2\otimes \cdots \otimes b_m\otimes b_1~,
	\end{align*}
	with
	\begin{multline*}
		B=
		\left|\bar{\rmF}\right|+\nabla(a_{\lambda_1+1}, \ldots, a_n,  a_1,\ldots, a_{\lambda_1})
		+\lambda_1+\frac{n(n+1)}{2}+
		n\lambda_1+\lambda_m
		+\\
		(|a_1|+\cdots+|a_{\lambda_1}|)(|a_{\lambda_1+1}|+\cdots+|a_n|)+
		|b_1|(|b_2|+\cdots+|b_m|)~,
	\end{multline*}
	after the symmetry relation~\eqref{eq:SignCyclicDual}. Since $A\equiv B \mod 2$, this concludes that part of the proof.

	\medskip

	Lemmata~\ref{lemm:Dpois!} and \ref{Lem:uDPois!} show that
	\[
		\Upsilon \ \colon \
		\widehat{\Hom}\left(\mathrm{c}\DPois^\antish , \End_A\right)
		\xrightarrow{\cong}
		\prod_{N \geqslant 1}
		\left(
		\bigoplus_{1\leqslant m < N}
		\Hom _{\C_m}
		\left(	\bigoplus_{\lambda_1+\cdots+\lambda_m=n} \bigotimes_{j=1}^m
		(\susp A)^{\otimes \lambda_j},  A^{\otimes m}
		\right)\right)
	\]
	is indeed an isomorphism. It remains to check that it preserves the respective Lie-admissible products, that is
	\[
		\Upsilon\left(\bar{\rmF}\star \bar{\rmG}\right)=
		\Upsilon\left(\bar{\rmF}\right)
		\circledast
		\Upsilon\left(\bar{\rmG}\right)~.
	\]
Comparing the formulas given in \cref{prop:InfDecDPoisAC} and in \cref{def:HighHochCx}, one can see that 
both $\Upsilon\left(\bar{\rmF}\star \bar{\rmG}\right)$ and 
$\Upsilon\left(\bar{\rmF}\right)\circledast\Upsilon\left(\bar{\rmG}\right)$ will be made up of the same terms, involving the same composites, possibly up to a sign. 

\medskip	
	
	For one of these terms, we will compute the signs appearing in the two formulas. 
	Let 
	$\lambda_1+\cdots+\lambda_m=n$ and
	$\lambda'_1+\cdots+\lambda'_{m'}=n'$
	be two  ordered partitions and let
	$p+1+q=\lambda_m$~, with $p, q \geqslant 0$~.
	For any elements
	$a_1, \ldots, a_{\cev{\lambda}_m+p}, a_{\cev{\lambda}_m+p+2}, \ldots, a_n, a'_1, \ldots, a'_{n'}\in A$, we consider the following notations
	\begin{align*}
		 & \bar{\rmG}\left(\nu_{\lambda'_1, \ldots, \lambda'_{m'}}\right)\big(a'_1\otimes \cdots \otimes a'_{n'}\big)=
		\sum b'_1\otimes \cdots \otimes b'_{m'} \quad \text{and}                                                       \\
		 &
		\bar{\rmF}\left(\nu_{\lambda_1, \ldots, \lambda_m}\right)
		\left(a_1\otimes \cdots \otimes a_{\cev{\lambda}_m+p}\otimes b'_1\otimes a_{\cev{\lambda}_m+p+2}
		\otimes \cdots \otimes
		a_n\right)
		= \sum  b_1\otimes \cdots \otimes b_m~.
	\end{align*}
	On the one hand, we have
	\begin{multline*}
		\left(\Upsilon\left(\bar{\rmF}\right)_{\lambda_1, \ldots, \lambda_m-1}
		\circledast_i^1 \Upsilon\left(\bar{\rmG}\right)_{\lambda'_1, \ldots, \lambda'_m}\right)
		\left(
		\susp a_1\otimes \cdots \otimes \susp a_{\cev{\lambda}_m}\otimes
		\susp a'_{1}\otimes \cdots \otimes \susp a'_{\lambda'_1}\otimes
		\susp a_{\cev{\lambda}_m+p+2}\otimes \cdots \otimes \susp a_n \otimes  \right. \\
		\left. \susp a'_{{\lambda'}_1+1}\otimes \cdots \otimes a'_{\cev{\lambda'}_{m'}}\otimes
		\susp a_{\cev{\lambda}_m+1}\otimes \cdots \otimes \susp a_{\cev{\lambda}_m+p}\otimes
		\susp a'_{\cev{\lambda'}_{m'}+1}\otimes \cdots \otimes \susp a'_{n'}
		\right)=\\
		\sum (-1)^C\, \susp\, b_1\otimes \cdots \otimes b_m \otimes b'_2 \otimes \cdots \otimes b'_{m'}~,
	\end{multline*}
	where $i=\cev{\lambda}_m+p+1$ and where
	\begin{align*}
		C= &
		\left(\left|a_{\cev{\lambda}_m+p+2}\right| + \cdots + \left|a_n\right|+q\right)
		\left(\left|a'_{1}\right| + \cdots + \left|a'_{\lambda'_1}\right|+\lambda'_1\right)
		\\&
		\ + \left(\left|a_{\cev{\lambda}_m+1}\right|+ \cdots + \left| a_{\cev{\lambda}_m+p}\right|+ p\right)
		\left(\left|a_{\cev{\lambda}_m+p+2}\right| + \cdots + \left|a_n\right|+
		\left|a'_{1}\right|+ \cdots + \left|a'_{\cev{\lambda'}_{m'}}\right|+q+\cev{\lambda'}_{m'}
		\right)
		\\&
		\ + \bar{\rmG}(|a|+n-1)+
		\left|\bar{\rmG}\right|+\nabla\left(a'_1, \ldots, a'_{n'}\right)+\lambda'_m+\frac{n'(n'+1)}{2}
		\\&
		\ + \left(\left|b'_1\right|+1\right)\left(\left|a_{\cev{\lambda}_m+p+2}\right| + \cdots + \left|a_n\right|+q\right)
		\\&
		\ + \left|\bar{\rmF}\right|
		+\nabla\left(a_1, \ldots, a_{\cev{\lambda}_{m}+p}, b'_1,a_{\cev{\lambda}_{m}+p+2}, \ldots, a_n\right)
		+\lambda_m
		+\frac{n(n+1)}{2}~,
	\end{align*}
	under the convention $a\coloneq a_1\otimes \cdots \otimes a_{\cev{\lambda}_m+p}\otimes  a_{\cev{\lambda}_m+p+2}\otimes \cdots \otimes  a_n$~.
	In plain words, one starts by permuting the elements $\susp a_j$ and $\susp a'_k$ in order to have the first ones on the left-hand side of the second ones (first two lines of $C$). 
	Then, one permutes $\Upsilon\left(\bar{\rmG}\right)_{\lambda'_1, \ldots, \lambda'_m}$ with the elements $\susp a_j$ and one applies it to the elements $\susp a'_k$ (third line of $C$). Finally, one permutes the output element $\susp b'_1$ with the last $q$ elements  $\susp a_j$ (fourth line of $C$) and one applies
	$\Upsilon\left(\bar{\rmF}\right)_{\lambda_1, \ldots, \lambda_m-1}$
	to the elements $\susp a_j$ and $b'_1$
	(fifth line of $C$).

	\medskip

	On the other hand, we have
	\begin{multline*}
		\Upsilon\left(\bar{\rmF}\star\bar{\rmG}\right)_{\lambda_1, \ldots, \lambda_{m-1},
		\lambda'_1+q, \lambda'_2, \ldots, \lambda'_{m'-1}, p+\lambda'_{m'}}
		\left(
		\susp a_1\otimes \cdots \otimes \susp a_{\cev{\lambda}_m}\otimes
		\susp a'_{1}\otimes \cdots \otimes \susp a'_{\lambda'_1}\otimes \right.\\
		\left. \susp a_{\cev{\lambda}_m+p+2}\otimes \cdots \otimes \susp a_n \otimes
		\susp a'_{{\lambda'}_1+1}\otimes \cdots \otimes a'_{\cev{\lambda'}_{m'}}\otimes
		\susp a_{\cev{\lambda}_m+1}\otimes \cdots \otimes \susp a_{\cev{\lambda}_m+p}\otimes
		\susp a'_{\cev{\lambda'}_{m'}+1}\otimes \cdots \otimes \susp a'_{n'}
		\right)
	\end{multline*}
	is equal to
	\[
		\sum (-1)^D\, \susp\, b_1\otimes \cdots \otimes b_m \otimes b'_2 \otimes \cdots \otimes b'_{m'}~,
	\]
	with
	\begin{align*}
		D= & \left|\bar{\rmF}\right|+\left|\bar{\rmG}\right|+ \\&
		\nabla\left(a_1, \ldots,  a_{\cev{\lambda}_m},
		a'_{1}, \ldots,  a'_{\lambda'_1},  a_{\cev{\lambda}_m+p+2},  \ldots,  a_n,  a'_{\cev{\lambda'}_1+1}, \ldots, a'_{\cev{\lambda'}_{m'}},
		a_{\cev{\lambda}_m+1},  \ldots,  a_{\cev{\lambda}_m+p},
		\ a'_{\cev{\lambda'}_{m'}+1},  \ldots, a'_{n'}
		\right)+                                              \\&
		p+\lambda'_m+\frac{(n+n'-1)(n+n')}{2}+                \\&
		\left(\left|a_{\cev{\lambda}_m+p+2}\right| + \cdots + \left|a_n\right|\right)
		\left(\left|a'_{1}\right| + \cdots + \left|a'_{\lambda'_1}\right|\right)
		+                                                     \\&
		\left(\left|a_{\cev{\lambda}_m+1}\right|+ \cdots + \left| a_{\cev{\lambda}_m+p}\right|\right)
		\left(\left|a_{\cev{\lambda}_m+p+2}\right| + \cdots + \left|a_n\right|+
		\left|a'_{1}\right|+ \cdots + \left|a'_{\cev{\lambda'}_{m'}}\right|
		\right)+                                              \\&
		q\lambda'_1+p\left(q+\cev{\lambda}'_{m'}\right) +(n-1)(n'-1)+
		\left|\bar{\rmG}\right|(n-1)+                         \\&
		\left(\bar{\rmG}+n'-1\right)|a|+
		\left|b'_1\right|\left(\left|a_{\cev{\lambda}_m+p+2}\right| + \cdots + \left|a_n\right|\right)~.
	\end{align*}
	In plain words, the first signs (first three lines of $D$) come from the application of
	\[\Upsilon\left(\bar{\rmF}\star\bar{\rmG}\right)_{\lambda_1, \ldots, \lambda_{m-1},
		\lambda'_1+q, \lambda'_2, \ldots, \lambda'_{m'-1}, p+\lambda'_{m'}}~,\]
	which implies to unshift all the elements $\susp a_j$ and $\susp a'_k$ for instance.
	As above, one permutes the elements $a_j$ and $a'_k$ in order to have the first ones on the left-hand side of the second ones (forth and fifth lines of $D$).
	Then, one computes the corresponding term in $\bar{\rmF} \star \bar{\rmG}$~:
	first one decomposes the element
	$\nu_{\lambda_1, \ldots, \lambda_{m-1}, \lambda'_1+q, \lambda'_2, \ldots, \lambda'_{m'-1}, p+\lambda'_{m'}}$
	in the infinitesimal coproperad $\mathrm{c}\DPois^\antish$ according to \cref{prop:InfDecDPoisAC}
	(first three terms of the sixth line of $D$) and then
	one applies $\bar{\rmG}$ to the above element
	$\nu_{ \lambda'_1,  \ldots, \lambda'_{m'}}$~,
	thereby permuting
	$\bar{\rmG}$ with the bottom element
	$\nu_{\lambda_1, \ldots, \lambda_{m}}$
	(last term of the sixth line of $D$).
	It remains to permute the map
	$\bar{\rmG}\left(\nu_{ \lambda'_1,  \ldots, \lambda'_{m'}}\right)$ with the elements $a_j$
	and to apply it to the elements $a'_k$ (first term of the seventh line of $D$).
	Finally, one permutes the output element $b'_1$ with the last $q$ elements  $a_j$ (last term of the seventh line of $D$) and one applies
	$\bar{\rmF}\left(\nu_{\lambda_1, \ldots, \lambda_m}\right)$
	to the elements $a_j$ and $b'_1$~, which produces no extra sign.

	\medskip

	Notice first that
	\begin{equation}\label{eq:frac}
		\frac{n(n+1)}{2}+\frac{n'(n'+1)}{2}  \equiv \frac{(n+n'-1)(n+n')}{2}+(n-1)(n'-1) +1 \mod 2 ~.
	\end{equation}
	Then we compute the difference between the $\nabla$ terms appearing in $C$ and in $D$:
	\begin{align*}
		 & \quad \nabla\left(a_1, \ldots,  a_{\cev{\lambda}_m},
		a'_{1}, \ldots,  a'_{\lambda'_1},  a_{\cev{\lambda}_m+p+2},  \ldots,  a_n,
		a'_{{\lambda'}_1+1}, \ldots, a'_{\cev{\lambda'}_{m'}},
		a_{\cev{\lambda}_m+1},  \ldots,  a_{\cev{\lambda}_m+p},
		\ a'_{\cev{\lambda'}_{m'}+1},  \ldots, a'_{n'}
		\right)-                                                                           \\&\quad
		\nabla\left(a_1, \ldots, a_{\cev{\lambda}_{m}+p}, b'_1,a_{\cev{\lambda}_{m}+p+2}, \ldots, a_n\right)
		-\nabla\left(a'_1, \ldots, a'_{n'}\right)
		\\
		&=
		(n'-1)\left(\left|a_1\right| +  \cdots + \left| a_{\cev{\lambda}_m}\right|\right)
		+\left(\lambda'_{m'}-q-1\right)\left(\left|a_{\cev{\lambda}_{m}+1}\right|+\cdots+\left|a_{\cev{\lambda}_{m}+p}\right|\right)
		\\ & \quad 
		+ \left(n'-\lambda'_1 +p\right)\left(\left|a_{\cev{\lambda}_{m}+p+2}\right|+\cdots+ \left|a_n\right|\right)
		+ q\left(\left|a'_1\right|+\cdots + \left|a'_{\lambda'_1}\right|\right) 
		+ p\left(\left|a'_1\right| + \cdots + \left|a'_{\cev{\lambda'}_{m'}}\right|\right) 
		- q\left|b_1'\right| \ .
	\end{align*}
	So the difference, computed modulo $2$, of all the terms in $C$ and $D$ of the form $\left|a_j\right|$ 
	is equal to $(n'-1)|a|$, which matches with the one present in $D$~. It remains to check the combinatorial terms which do not depend on the degrees of the elements $a_j$ and $a
	'_k$: after removing the same terms which obviously appear both in $C$ and in $D$, we get
	\[q\lambda'_1+p\left(q+\cev{\lambda'}_{m'}\right)+q +\lambda_m \]
	in $C$ and
	\[p+ q\lambda'_1+p\left(q+\cev{\lambda}'_{m'}\right)+1  \]
	in $D$, where the final term $1$ comes from \eqref{eq:frac}.
	Since $\lambda_m=p+q+1$, we get $C\equiv D \mod 2$,
	which concludes the proof.
\end{proof}

%%%%%%%%%%%%%   APPENDIX B 

\section{The homotopy transfer formula in genus 0}\label{appendiceHTT}
The homotopy transfer theorem amounts to extending the data of a contraction and a $\Omega \rmC$-gebra structure to a transferred $\Omega \rmC$-gebra structure together with the extensions of the two chain maps into $\infty$-quasi-isomorphisms and the homotopy into an $\infty$-homotopy. In the end, four formulas are required. 
In the operadic case, that is for rooted trees, an interesting phenomenon appears: the first two formulas rely on \emph{trees} with edges labelled by the homotopy and the two last formulas rely on \emph{trees with levels} labelled by homotopies. This was first shown by M. Markl in \cite{Markl06} in the example of $A_\infty$-algebras and a general conceptual explanation was given in \cite{DSV16} in terms of two canonical gauge group elements. In the properadic case, the four formulas of the homotopy transfer theorem rely on levelled graphs \cite[Section~4]{HLV20} and there is a similar interpretation in terms of gauge group action \cite[Section~3.2]{CV22}. 

\medskip

In this appendix, we simplify the formula given in \cite[Theorem~4.14]{HLV20} for the homotopy transferred algebraic structure removing the levels in the case where only genus 0 graphs are involved, that is for $\Omega \rmC$-gebras where $\rmC$ is a conilpotent codioperad. This applies to pre-Calabi--Yau algebras and homotopy double Poisson gebras in this paper, see \cref{prop:FormulaHTT}, but it also applies to homotopy Lie bialgebras and homotopy infinitesimal bialgebras, see \cite[Section~5.5]{MV09I}. Throughout this appendix, we use the conventions of \cite{HLV20} and we add the index ``$\gg=0$'' for the dioperadic analogues.

\medskip

We consider a contraction of a dg vector space $(A, d_A)$ into another dg vector space $(H,d_H)$:
\begin{eqnarray*}
	&\begin{tikzcd}
		(A,d_A)\arrow[r, shift left, "p"] \arrow[loop left, distance=1.5em,, "h"] & (H,d_H) \arrow[l, shift left, "i"]
	\end{tikzcd} \ ,&\\
&\text{with}\quad 	pi=\id_H\ , \quad \id_A -ip= d_Ah+hd_A\ , \quad
	hi=0\ , \quad ph=0\ , \quad \text{and} \quad h^2=0 \ .&
\end{eqnarray*}
We denote by $\scrG^c_{\rmg =0}$ the 
sub-comonad of $\scrG^c$ made up of reduced connected flow-directed graphs of genus $0$ and we denote by 
$\rmB_{\rmg=0}$ the dioperadic bar construction, which is made up the summand of the properadic bar construction restricted to genus 0 flow-directed connected graphs. It forms a functor from dg properads (or dg dioperads) to dg codioperads that are particular dg coproperads. 

\begin{definition}[Dioperadic Van der Laan map]
The \emph{dioperadic Van der Laan map} 
\[\psi \in \Hom_{\Sy}\left(\rmB_{\gg =0}\End_A,\End_H\right)\]
is defined on non-trivial elements $\scrG^c_{\rmg =0}(\susp\End_A)\to \End_H$ 
by removing all the suspensions 
$\susp$ and by labelling the input edges by $i$, the output edges by $p$, and the internal edges by $h$:
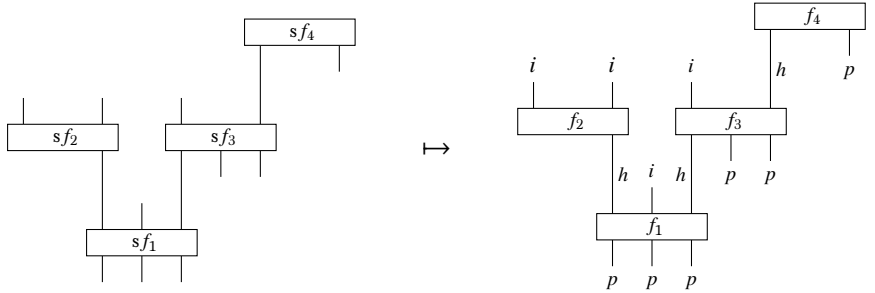
\[
\vcenter{\hbox{\begin{tikzpicture}[scale=0.7]
%%%%%%%%%%%%%%%%%%  E
		\coordinate (A1) at (1.5,4);
		\draw[thin]
		(A1)  --++(0,-0.5) % node[below] {\scalebox{0.8}{$i$}}		
		(A1) ++(0.75,0) --++(0,-0.5) % node[below] {\scalebox{0.8}{$i$}}		
		(A1) ++(1.5,0) --++(0,-0.5) % node[below] {\scalebox{0.8}{$i$}}				
		(A1) ++(0.75,0.5) --++(0,0.5);%  node[above] {\scalebox{0.8}{$p$}}				
		\draw[fill=white] (1.2,4) rectangle (3.3, 4.5);
%%%%%%%%%%%%%%%%%%%  THICK EDGES
		\draw (3,4.5) --++ (0, 1.5) coordinate (B1);
		\draw (1.5,4.5) --++ (0, 1.5) coordinate (C1);		
		\draw (4.5,6.5) to[out=90,in=270] (4.5, 8) coordinate (Z1);	
%%%%%%%%%%%%%%%%%%  B 
		\draw[thin]
		(B1) ++(0,0.5) --++(0,0.5) % node[above] {\scalebox{0.8}{$i$}}		
		(B1) ++(0.75,0) --++(0,-0.5) % node[below] {\scalebox{0.8}{$p$}}				
		(B1) ++(1.5,0) --++(0,-0.5); % node[below] {\scalebox{0.8}{$p$}}						
		\draw[fill=white] (2.7,6) rectangle (4.8, 6.5);		
%%%%%%%%%%%%%%%%%%  Z
		\draw[thin]
		(Z1) ++(1.5,0) --++(0,-0.5); % node[below] {\scalebox{0.8}{$p$}}								
		\draw[fill=white] (4.2, 8) rectangle (6.3, 8.5);
%%%%%%%%%%%%%%%%%%  C
	\draw[thin]
	(C1)  ++(0,0.5) --++(0,0.5) % node[above] {\scalebox{0.8}{$i$}}		
	(C1)  ++(-1.5,0.5) --++(0,0.5); % node[above] {\scalebox{0.8}{$i$}}		
	\draw[fill=white] (-0.3,6) rectangle (1.8, 6.5);		
%%%%%%%%%%%%%%%%%%   h
%		\draw (4.5,7.25) node[right] {\scalebox{0.8}{$h$}}; 
%		\draw (3,5.25) node[left] {\scalebox{0.8}{$h$}}; 
%%%%%%%%%%%%%%%%%%   BOXES 
	\draw (2.3,4.25) node {\scalebox{0.7}{$\susp f_1$}}; 
	\draw (0.8,6.25) node {\scalebox{0.7}{$\susp f_2$}}; 
	\draw (3.8,6.25) node {\scalebox{0.7}{$\susp f_3$}}; 
	\draw (5.3,8.25) node {\scalebox{0.7}{$\susp f_4$}}; 			
	\end{tikzpicture}}}
\qquad \mapsto \qquad 
\vcenter{\hbox{\begin{tikzpicture}[scale=0.7]
%%%%%%%%%%%%%%%%%%  E
		\coordinate (A1) at (1.5,4);
		\draw[thin]
		(A1)  --++(0,-0.5)  node[below] {\scalebox{0.7}{$p$}}		
		(A1) ++(0.75,0) --++(0,-0.5)  node[below] {\scalebox{0.7}{$p$}}		
		(A1) ++(1.5,0) --++(0,-0.5)  node[below] {\scalebox{0.7}{$p$}}				
		(A1) ++(0.75,0.5) --++(0,0.5)  node[above] {\scalebox{0.7}{$i$}};				
		\draw[fill=white] (1.2,4) rectangle (3.3, 4.5);
%%%%%%%%%%%%%%%%%%%  THICK EDGES
		\draw (3,4.5) --++ (0, 1.5) coordinate (B1);
		\draw (1.5,4.5) --++ (0, 1.5) coordinate (C1);		
		\draw (4.5,6.5) to[out=90,in=270] (4.5, 8) coordinate (Z1);	
%%%%%%%%%%%%%%%%%%  B 
		\draw[thin]
		(B1) ++(0,0.5) --++(0,0.5)  node[above] {\scalebox{0.7}{$i$}}		
		(B1) ++(0.75,0) --++(0,-0.5)  node[below] {\scalebox{0.7}{$p$}}				
		(B1) ++(1.5,0) --++(0,-0.5) node[below] {\scalebox{0.7}{$p$}};						
		\draw[fill=white] (2.7,6) rectangle (4.8, 6.5);		
%%%%%%%%%%%%%%%%%%  Z
		\draw[thin]
		(Z1) ++(1.5,0) --++(0,-0.5) node[below] {\scalebox{0.7}{$p$}};
		\draw[fill=white] (4.2, 8) rectangle (6.3, 8.5);
%%%%%%%%%%%%%%%%%%  C
	\draw[thin]
	(C1)  ++(0,0.5) --++(0,0.5)  node[above] {\scalebox{0.8}{$i$}}		
	(C1)  ++(-1.5,0.5) --++(0,0.5)  node[above] {\scalebox{0.8}{$i$}}	;
	\draw[fill=white] (-0.3,6) rectangle (1.8, 6.5);		
%%%%%%%%%%%%%%%%%%   h
		\draw (4.4,7.25) node[right] {\scalebox{0.7}{$h$}}; 
		\draw (3.1,5.25) node[left] {\scalebox{0.7}{$h$}};
		\draw (2,5.25) node[left] {\scalebox{0.7}{$h$}}; 		
%%%%%%%%%%%%%%%%%%   BOXES 
	\draw (2.3,4.25) node {\scalebox{0.7}{$f_1$}}; 
	\draw (0.8,6.25) node {\scalebox{0.7}{$f_2$}}; 
	\draw (3.8,6.25) node {\scalebox{0.7}{$f_3$}}; 
	\draw (5.3,8.25) node {\scalebox{0.7}{$f_4$}}; 			
	\end{tikzpicture}}}~.
\]
It vanished on the trivial element: $\psi|_\rmI\coloneq 0$~.
\end{definition}

\begin{remark}
Notice that the map $\psi$ produces signs due to the application of the Koszul sign rule when permuting elements. 
 \end{remark}
 
\begin{lemma}\label{lem:VdLTw}
The dioperadic Van der Laan map 
is a properadic twisting morphism 
\[\psi \in \mathrm{Tw}\left(\rmB_{\rmg=0} \End_A, \End_H\right)~.\]
\end{lemma}

\begin{proof}
The proof is similar to the one performed in the operadic case that is with rooted trees, see \cite[Proposition~10.3.2]{LV12} for complete details. Let us give the arguments quickly here.  
We claim that the evaluation of 
$\partial \psi +\psi\star \psi$
on any  flow-directed connected genus 0 graph $\gg(\susp f_1, \ldots, \susp f_k)$ with vertices labelled by suspension of multilinear operations in $A$ vanishes. 
The term $(\partial \psi)\left( \gg(\susp f_1, \ldots, \susp f_k)\right)$ is equal to the differential of $\End_H$ applied to 
$\psi\left( \gg(\susp f_1, \ldots, \susp f_k)\right)$ minus the composite of the differential of the dioperadic bar construction of $\End_A$ followed by the Van der Laan map $\psi$. 
Both terms produce almost graphs of type $\gg$ with vertices labelled by the operations $f_1, \ldots, f_k$, with input edges labelled by $i$, output edges labelled by $p$, and internal edges labelled by $h$. 
Actually 
the first term is the sum of two summands: the first one is equal to the sum over the vertices of these graphs 
where the vertex is labelled by $\partial f_i$ and the second one is equal to the sum over the internal edges where the internal edge  is labelled by $\partial h=\id_A -ip$~. 
The second term, that comes with an overall minus sign, 
is the sum of two summands: 
the first one is equal to the sum over the internal edges where the internal edge is labelled by the identity $\id_A$ of $A$ and the second one is equal to the sum over the vertices of these graphs where the vertex is labelled by $\partial f_i$. 
This shows that the term $(\partial \psi)\left( \gg(\susp f_1, \ldots, \susp f_k)\right)$ is equal to the sum over the internal edges where the internal edge is labelled by $-ip$~. This term cancels precisely with 
$(\psi \star \psi)\left( \gg(\susp f_1, \ldots, \susp f_k)\right)$, which concludes the proof. 
\end{proof}

\begin{proposition}\label{prop:NewTransForm}
Let $\rmC$ be a conilpotent codioperad, let $\alpha \in \Tw(\rmC, \End_A)$ be a $\Cobar \rmC$-gebra structure on $A$, and let $(h,i,p)$ be a contraction from $A$ onto $H$, the composite 
\[\begin{tikzcd}[column sep=normal]
	{\rmC}\arrow[r,"G_\alpha"] &  \Bar_{\gg=0}\End_A \arrow[r,"\psi"] & \End_H 
\end{tikzcd}\]
defines a $\Cobar \rmC$-gebra structure on $H$.
\end{proposition}

\begin{proof}
This is a direct corollary of \cref{lem:VdLTw} and the general theory of properadic twisting morphisms \cite{MV09I}. 
Recall first that the canonical morphism of conilpotent coproperads $G_\alpha : \rmC \to \Bar\End_A$ associated to the twisting morphism $\alpha$ is explicitly given by the composite 
\[
\overline{\rmC} \xrightarrow{\widetilde{\Delta}}
\scrG^c\Big(\overline{\rmC}\Big)
\xrightarrow{\scrG^c(\susp \alpha)}
\scrG^c(\susp\End_A)~. 
\]
Since the coproperad $\rmC$ is a conilpotent codioperad, its comonadic decomposition map $\widetilde{\Delta}$ produces only genus 0 flow-directed connected graphs. Therefore, the image of the map $G_\alpha$ actually lands in 
$\scrG^c_{\gg=0}(\susp\End_A)$, that is the dioperadic bar construction $\Bar_{\gg=0}\End_A$. In the end, the map $G_\alpha$ is a morphism of dg codioperads, that can be composed with the dioperad Van der Laan twisting morphism to produce a new properadic twisting morphism. 
\end{proof}

\cref{prop:NewTransForm} provides us with a universal genus 0 formula for homotopy gebra structures controlled codioperads transferred  through contractions. One might now wonder how compatible is this new structure with respect to  the original one, that is how to extend the two chain maps of the contraction into $\infty$-quasi-isomorphisms. We are going to settle this question by proving that  this formula actually coincides with the general one of \cite[Theorem~4.14]{HLV20} in the case of codioperads.
Let us recall that we introduced in \emph{loc. cit.} 
the properadic Van der Laan twisting morphism $\varphi \in\Tw(\rmB\End_A,\End_H)$ 
defined by the composite 
\[\begin{tikzcd}
\scrG^c(\susp\End_A) \arrow[r,"\text{lev}"] & \scrG_{\text{lev}}(\susp\End_A)  \arrow[r,"\rmP\rmH\rmI"] & \End_H~, 
\end{tikzcd}\]
where $\scrG_{\text{lev}}$ is made up of flow-directed connected graphs with levels and only one vertex per level, 
where the levelisation map ``$\text{lev}$'' amounts to sending any flow-directed connected graphs to all the ways to put its vertices (differentiating all of them) on levels, one at a time, and where the map $\rmP\rmH\rmI$ amounts to labelling the input edges by $i$, the output edges by $p$, and the intermediate levels, that are the levels between vertices, by the symmetric homotopies 
$$ h_n:=\frac{1}{n!}\sum_{\sigma \in \Sy_n}  \sum_{k=1}^n \left(
\id_A^{\otimes (k-1)} \otimes h \otimes \pi^{\otimes (n-k)}
\right)^\sigma$$
from $\pi^{\otimes n}=(ip)^{\otimes n}$ to $\id_A^{\otimes n}$, where $n$ stands for the number of edges crossing that intermediate level. 

\begin{theorem}\label{thm:HTTCodioperad}
Let $\rmC$ be a conilpotent codioperad, let $\alpha \in \Tw(\rmC, \End_A)$ be a $\Cobar \rmC$-gebra structure on $A$, and let $(h,i,p)$ be a contraction from $A$ onto $H$, the two 
$\Cobar \rmC$-gebra structures on $H$ defined respectively 
by $\psi \circ G_\alpha$ and by $\varphi \circ G_\alpha$ are equal and given by 
\[
\overline{\rmC} \xrightarrow{\widetilde{\Delta}}
\scrG^c\Big(\overline{\rmC}\Big)
\xrightarrow{\scrG^c(\susp \alpha)}
\scrG^c(\susp\End_A)
\xrightarrow{\psi} \End_H~. 
\]
\end{theorem}

We will prove this result by adapting the arguments given in the proofs of Point~$(1)$ of \cite[Lemma~6]{DSV16} and by generalising \cite[Proposition~11]{DSV16}. This latter proposition gives a formula which computes the cardinal of the automorphism group of an unlabelled rooted tree with the weights of its levelisations. We define the \emph{weight} $\omega(\lambda)$ of a levelisation $\lambda$ of a flow-directed connected genus 0 graph by the product over the intermediate levels of the inverse of the number of crossing internal edges:
\[
\omega\left(\vcenter{\hbox{\begin{tikzpicture}[scale=0.7]
%%%%%%%%%%%%%%%%%%  E
		\coordinate (A1) at (1.5,2);
		\draw[thin]
		(A1)  --++(0,-0.3) % node[below] {\scalebox{0.8}{$i$}}		
		(A1) ++(0.75,0) --++(0,-0.3) % node[below] {\scalebox{0.8}{$i$}}		
		(A1) ++(1.5,0) --++(0,-0.3) % node[below] {\scalebox{0.8}{$i$}}				
		(A1) ++(0.75,0.5) --++(0,0.3);%  node[above] {\scalebox{0.8}{$p$}}				
		\draw[fill=white] (1.2,2) rectangle (3.3, 2.5);
%%%%%%%%%%%%%%%%%%%  INTERNAL EDGES
		\draw (3,2.5) --++ (0, 1) coordinate (B1);
		\draw (1.5,2.5) --++ (0, 2.5) coordinate (C1);		
		\draw (4.5,4) -- (4.5, 6.5) coordinate (Z1);	
%%%%%%%%%%%%%%%%%%  B 
		\draw[thin]
		(B1) ++(0,0.5) --++(0,0.3) % node[above] {\scalebox{0.8}{$i$}}		
		(B1) ++(0.75,0) --++(0,-0.3) % node[below] {\scalebox{0.8}{$p$}}				
		(B1) ++(1.5,0) --++(0,-0.3); % node[below] {\scalebox{0.8}{$p$}}						
		\draw[fill=white] (2.7,3.5) rectangle (4.8, 4);		
%%%%%%%%%%%%%%%%%%  Z
		\draw[thin]
		(Z1) ++(1.5,0) --++(0,-0.3); % node[below] {\scalebox{0.8}{$p$}}								
		\draw[fill=white] (4.2, 6.5) rectangle (6.3, 7);
%%%%%%%%%%%%%%%%%%  C
	\draw[thin]
	(C1)  ++(0,0.5) --++(0,0.3) % node[above] {\scalebox{0.8}{$i$}}		
	(C1)  ++(-1.5,0.5) --++(0,0.3); % node[above] {\scalebox{0.8}{$i$}}		
	\draw[fill=white] (-0.3,5) rectangle (1.8, 5.5);		
%%%%%%%%%%%%%%%%%%   intermediate levels 
	\draw[dotted] 
	(-1,3) -- (7,3)
	(-1,4.5) -- (7,4.5)
	(-1,6) -- (7,6);				
%%%%%%%%%%%%%%%%%%   BOXES 
	\draw (2.3,2.25) node {\scalebox{0.7}{$1$}}; 
	\draw (0.8,5.25) node {\scalebox{0.7}{$2$}}; 
	\draw (3.8,3.75) node {\scalebox{0.7}{$3$}}; 
	\draw (5.3,6.75) node {\scalebox{0.7}{$4$}}; 				
	\end{tikzpicture}}}\right)=\frac12\times \frac12 \times 1=\frac14~. 
\]
By convention, the weight of the unique levelisation of any 1-vertex graph is equal to 1. 

\begin{lemma}\label{lem:CoefSumOmega=Aut}
For any flow-directed connected graph $\gg$ of genus $0$, the following relation holds 
\[
\sum_{\text{levelisation} \atop \lambda\ \text{of}\ \gg} \omega(\lambda)=1~.
\]
\end{lemma}

\begin{proof}
The proof is performed by induction on the number $k$ of vertices of the connected flow-directed graphs $\gg$ of genus $0$. For $k=1$ and $k=2$, the statement is obvious since the automorphism group is trivial and since there is one possible levelisation with weight equals to 1. Suppose now that the statement holds for 
flow-directed connected graphs of genus $0$ with $k-1$ vertices and let us prove it for any 
flow-directed connected graph $\gg$ of genus $0$ with $k$ vertices.
We consider one vertex $v$ connected to the rest of the graph $\gg$ by only one internal edge $e$: it exists since otherwise the total number of vertices would be infinite or the graph would not be of genus 0. 
We denote by $w$ the vertex of $\gg$ which lies on the other end of the edge $e$.
Let us suppose here that $v$ lies below $w$; the other case is treated by the same arguments. 
We denote by $\bar \gg$ the sub-graph obtained from the graph $\gg$ by removing the vertex $v$ and the edge $e$. Any levelisation of the graph $\gg$ induces a canonical levelisation of the sub-graph $\bar \gg$. 
In the other way round, from any levelisation $\bar \lambda$ of the sub-graph $\bar \gg$, one recovers all the levelisations of $\gg$ where it comes from as follows. Let $l\geqslant 1$ be the level, counted from bottom to top, on which sits the vertex $w$. Let us denote by $n_1, \ldots, n_{k-2}$ the number of internal edges lying respectively on the intermediate levels $1, \ldots, k-2$. The weight of the levelisation $\bar \lambda$ is thus equal to 
\[\omega\left(\bar \lambda\right)=\frac{1}{n_1\cdots n_{l-1}}\cdot \frac{1}{n_l\cdots n_{k-2}}~.\]
From the levelisation $\bar \lambda$ of the sub-graph $\bar \gg$, one gets the levelisations of $\gg$ by placing the vertex $v$ at levels $1, \ldots, l$ respectively. These levelisations come with the following respective weights
\begin{multline*}
\frac{a}{(n_1+1)\cdots (n_{l-1}+1)}~, \ 
\frac{a}{n_1(n_1+1)\cdots (n_{l-1}+1)}~, \ 
\frac{a}{n_1n_2(n_2+1)\cdots (n_{l-1}+1)}~, \ \ldots  \\
\frac{a}{n_1n_2 \ldots n_{l-1} (n_{l-2}+1)(n_{l-1}+1)}~, \ 
\frac{a}{n_1n_2 \ldots n_{l-1} (n_{l-1}+1)}~, 
\end{multline*}
where $a=\frac{1}{n_l\cdots n_{k-2}}$~. 
Their sum is equal to the weight $\omega\left(\bar \lambda\right)=\frac{a}{n_1\cdots n_{l-1}}$ of the levelisation $\bar \lambda$. Therefore we have 
\[
\sum_{\text{levelisation} \atop \lambda\ \text{of}\ \gg} \omega(\lambda)=
\sum_{\text{levelisation} \atop \bar \lambda\ \text{of}\ \bar \gg}\omega\left(\bar \lambda\right)=1~, 
\]
by the induction hypothesis applied to the sub-graph $\bar \gg$; this concludes the proof. 
\end{proof}

Here is a direct corollary of \cref{lem:CoefSumOmega=Aut} that will not use here but that might be helpful outside the present context. 
For any graph $\gg$, we consider its set of \emph{labelled} graphs obtained by labelling  bijectively its vertices by $1, \ldots, k$~. 
The automorphism group of a flow-directed graph $\gg$ acts canonical on its set of labelled flow-directed graphs. 
An \emph{unlabelled} flow-directed graph $\gamma$ is an orbit  under this action. 
The set of levelled labelled graphs with underlying flow-directed graph $\gg$ admits two compatible actions of the automorphism group of $\gg$: one permuting the levels and one preserving them. 
A levelisation of an unlabelled 
flow-directed graph $\gamma$ is an obit of the set of labelled levelled graphs under the action of 
$\mathrm{Aut}\, \gg \times \mathrm{Aut}\, \gg$~. 
The weight of a levelisation of an unlabelled flow-directed connected graph is defined in the same way as before. 
By extension, we define the automorphism group of an unlabelled flow-directed graph $\gamma$ to be the automorphism group of its underlying flow-directed graph $\gg$.

\begin{proposition}\label{prop:LevelGraph}
For any unlabelled flow-directed connected graph $\gamma$ of genus $0$, the following relation holds 
\[
\sum_{\text{levelisation} \atop \lambda\ \text{of}\ \gamma} \omega(\lambda)=\frac{1}{\left|\mathrm{Aut}\,  \gamma\right|}~.
\]
\end{proposition}

\begin{proof}
Let $\gamma$ be an unlabelled flow-directed connected graph of genus $0$ with underlying graph $\gg$. 
The automorphism group of $\gg$ acts canonically on the set of levelisations of $\gg$. 
The set of orbits is in one-to-one correspondence with the set of levelisations of the unlabelled flow-directed connected graph $\gamma$. 
Since this action if free and since the weight of a levelled flow-directed graph is constant on orbits, we get 
\[
\sum_{\text{levelisation} \atop {\Lambda}\ \text{of}\ {\gg}} \omega({\Lambda})=
\sum_{\text{levelisation} \atop \lambda\ \text{of}\ \gamma} \left|\mathrm{Aut} \, {\gg}\right|\omega(\lambda)=
\left|\mathrm{Aut}\,  \gamma\right| \sum_{\text{levelisation} \atop \lambda\ \text{of}\ \gamma} \omega(\lambda)=1
~,\]
by \cref{lem:CoefSumOmega=Aut}.
\end{proof}

\begin{remark}
\cref{prop:LevelGraph} includes the case of unlabelled rooted trees, previously settled in \cite[Proposition~11]{DSV16}. 
The arguments given in \emph{loc. cit.} do not apply to unlabelled connected flow-directed genus 0 graphs:
the present proof is different and somehow more elementary. 
\end{remark}

\begin{proof}[Proof of \cref{thm:HTTCodioperad}]
It is enough to prove that the restriction of the map $\varphi$ to genus 0 flow-directed connected graphs is equal to $\psi$, that is $\varphi|_{\gg=0}=\psi$. 

\medskip

Let  $\gg_\lambda(\susp f_1, \ldots, \susp f_k)$ be a flow-directed connected graph $\gg$ of genus 0 with vertices labelled by the suspension of a multilinear operations $f_1, \ldots, f_k$ of $A$ together with the data of a levelisation $\lambda$ made up of one vertex per level. 
We claim that 
\begin{equation}\label{eq:PHI=omegaPsi}
\rmP\rmH\rmI\big(\gg_\lambda(\susp f_1, \ldots, \susp f_k)\big)
=\omega(\lambda)\psi\big(\gg(\susp f_1, \ldots, \susp f_k)\big)~.
\end{equation}
The relations 
$p\pi =p$, $\pi i=i$, $ph=0$, and $hi=0$ 
of the contraction imply first that the left-hand side is a composite of operations along flow-directed connected graph of genus 0 with the input edges labelled by $i$ and the the output edges labelled by $p$. 
Since the flow-directed connected graph $\gg$ has genus 0, its number of internal edges is equal to the number of intermediate levels.
This implies that these internal edges can only be labelled by one $h$ and the remaining ones by the identity $\id_A$ at each intermediate level. In the end, one gets a graph without any levelisalition, and with input edges labelled by $i$, output edges labelled by $p$, and internal edges labelled by $h$, that is $\psi\big(\gg(\susp f_1, \ldots, \susp f_k)\big)$~.

\medskip

It remains to compute the overall coefficient. If an intermediate level crosses $m$ internal edges, then, the 
symmetric homotopy $h_n$ labels one internal edge by $h$, the other $m-1$ internal edges by the identity $\id_A$ and the input and output edges by $\pi$ or $\id_A$. In the end, each such term produces an intermediate level 
with one internal edge labelled by $h$, the other $m-1$ internal edges labelled by $\id_A$ and the input and output edges by $\pi$. This term appears with the coefficient $\frac{1}{m}$:
notice first that 
\[\left(p^{\otimes k}\otimes \id_A^{\otimes (n-k)}\right)h_{n}\left(\id_A^{\otimes (m+k)}\otimes i^{\otimes (n-m-k)}\right)=p^{\otimes k}\otimes h_m\otimes i^{\otimes (n-m-k)}~,\] 
see the proof of Point~(1) of \cite[Lemma~7]{DSV16} (with $i$ instead of $\pi$), and notice then that the term $h\otimes \id_A^{\otimes (m-1)}$ appears with coefficient $\frac{(m-1)!}{m!}=\frac{1}{m}$ in $h_m$~.
Globally, many terms $\psi\big(\gg(\susp f_1, \ldots, \susp f_k)\big)$ can appear, with coefficient $\pm \omega(\lambda)$ each time: all the configurations with one $h$ per internal edge and one $h$ per intermediate level produce such a term. We claim that almost all these terms cancel, due to sign issues, and that we get only one of them with coefficient $\omega(\lambda)$ in the end. 
Let us proof this claim by induction on the number of vertices. 
This is clear when the number of vertices is equal to $1$ or $2$. 
Suppose now that the result holds true up to $k-1$ vertices and let us prove it for $k$ vertices. 
Let $\gg_\lambda(\susp f_1, \ldots, \susp f_k)$ be the underlying flow-directed connected genus 0 levelled graph. 
We consider one vertex $v$ connected to the rest of the graph by only one internal edge $e$: this exists since otherwise the total number of vertices would be infinite or the graph would not be of genus 0.
If the homotopy $h$ appears on the internal edge $e$ on the intermediate level located just next to $v$, above if $e$ is above or below otherwise, then the other homotopies $h$ appear on the other intermediate levels of the graph and we conclude by the induction hypothesis applied to the flow-directed connected genus 0 levelled sub-graph obtained by removing the vertex $v$, all the edges attached to it, and the level on which it sits. 
If the homotopy $h$ appears at an intermediate level located not next to $v$, then the other homotopies $h$ appear
at the crossings of the internal edges and intermediate levels of the the flow-directed connected genus 0 levelled sub-graph $\bar{\gg}_{\bar{\lambda}}$ obtained by removing the vertex $v$, all the edges attached to it, and the level on which $h$ sits this time. 
If the vertex $v$ is sitting at the top (respectively at the bottom) of the levelled graph $\gg_\lambda$, the labelling 
of the levelled sub-graph $\bar{\gg}_{\bar{\lambda}}$ would be made up of a first level of $i$'s followed directly below by a homotopy $h_n$ (respectively a last level of $p$'s preceded directly above by a homotopy $h_n$): such a composite vanishes. 
Otherwise, it means that the vertex $v$ is sitting between two levels carrying two vertices. Removing $v$ creates a levelled sub-graph $\bar{\gg}_{\bar{\lambda}}$ whose labelling will be made up of two consecutive homotopies $h_n$ and $h_{n}$ whose composite vanishes. This concludes this part of the proof. 

\medskip

Notice that the levelisation map ``$\text{lev}$'', which appears in the definition of the twisting morphism 
$\varphi=\rmP\rmH\rmI \circ\text{lev}$, 
puts all the (differentiated) vertices on levels, one at a time. 
Therefore \cref{eq:PHI=omegaPsi} and \cref{lem:CoefSumOmega=Aut} show that
\begin{align*}
\varphi\big(\gg(\susp f_1, \ldots, \susp f_k)\big)
&=
\sum_{\text{levelisation} \atop \lambda\ \text{of}\ \gg}\rmP\rmH\rmI\big(\gg_\lambda(\susp f_1, \ldots, \susp f_k)\big)
=\left(\sum_{\text{levelisation} \atop \lambda\ \text{of}\ \gg} \omega(\lambda)\right){\psi}\big(\gg(\susp f_1, \ldots, \susp f_k)\big)\\
&={\psi}\big(\gg(\susp f_1, \ldots, \susp f_k)\big)~,
\end{align*}
which concludes the entire proof. 
\end{proof}

\begin{remark}
Already in the operadic case, that is for rooted trees, the formula for the extension of the chain map $p$ to an 
$\infty$-morphism (respectively the extension of the chain homotopy $h$ to an $\infty$-homotopy) cannot be simplified in general: considering levels is mandatory there, see \cite[Section~8]{DSV16}. 
In the dioperadic case, that is for flow-directed connected graphs of genus $0$, it is not possible to simplify the formula for the extension of the chain map $i$ to an $\infty$-morphism: it is still mandatory to consider levels, as the example of upside down trees, with the root at the top and the leaves at the bottom, encoding coalgebras show. In this latter case, one gets the same combinatorics as for the extension of the chain map $p$ to an $\infty$-morphism for rooted trees. 
\end{remark}

\bibliographystyle{alpha}
\bibliography{biblio}
\end{document}